\documentclass[sn-mathphys,Numbered]{sn-jnl}

\usepackage{graphicx}%
\usepackage{multirow}%
\usepackage{amsmath,amssymb,amsfonts}%
\usepackage{amsthm}%
\usepackage{mathrsfs}%
\usepackage[title]{appendix}%
\usepackage{xcolor}%
\usepackage{textcomp}%
\usepackage{manyfoot}%
\usepackage{booktabs}%
\usepackage{algorithm}%
\usepackage{algorithmicx}%
\usepackage{algpseudocode}%
\usepackage{listings}%

\usepackage{longtable,booktabs}
\usepackage{geometry}
\usepackage{ulem}
\usepackage{bm}
\usepackage{diagbox}
\usepackage{appendix}
\usepackage{enumerate}
\usepackage[numbers,sort&compress]{natbib}
\usepackage{pgfplots}
\usepackage{bbding}
\usepackage{subfigure}
\usepackage{bbding}
\usepackage{indentfirst}
\usepackage{cases}
\usepackage{array}
\usepackage{subfigure}
\usepackage{float}
\newcommand{\mc}{\multicolumn}
\renewcommand{\Re}{{\rm I}\! {\rm R}}

\theoremstyle{thmstyleone}%
\newtheorem{theorem}{Theorem}
\newtheorem{proposition}[theorem]{Proposition}%

\theoremstyle{thmstyletwo}%
\newtheorem{remark}{Remark}%
\newtheorem{lemma}{Lemma}%
\theoremstyle{thmstylethree}%

\begin{document}

\title[Article Title]{A Highly Efficient Adaptive-Sieving-Based Algorithm for the High-Dimensional Rank Lasso Problem}


\author[1]{\fnm{Xiaoning} \sur{Bai}}\email{bxn\_2922@163.com}

\author*[2]{\fnm{Qingna}\sur{Li}}\email{qnl@bit.edu.cn}

\affil[1]{\orgdiv{School of Mathematics and Statistics}, \orgname{Beijing Institute of Technology}, \orgaddress{\street{No. 5 Zhongguancun South Street, Haidian District}, \city{Beijing}, \postcode{100081},  \country{China}}}

\affil*[2]{\orgdiv{School of Mathematics and Statistics/ Beijing Key Laboratory on MCAACI/ Key Laboratory of Mathematical Theory and Computation in Information Security}, \orgname{Beijing Institute of Technology}, \orgaddress{\street{No. 5 Zhongguancun South Street, Haidian District}, \city{Beijing}, \postcode{100081}, \country{China}}}

\abstract{The high-dimensional rank lasso (hdr lasso) model is an efficient approach to deal with high-dimensional data analysis. It was proposed as a tuning-free robust approach for the high-dimensional regression and was demonstrated to enjoy several statistical advantages over other approaches. The hdr lasso problem is essentially an $L_1$-regularized optimization problem whose loss function is Jaeckel's dispersion function with Wilcoxon scores. Due to the nondifferentiability of the above loss function, many classical algorithms for lasso-type problems are unable to solve this model. In this paper, inspired by the adaptive sieving strategy for the exclusive lasso problem \cite{SDF}, we propose an adaptive-sieving-based algorithm to solve the hdr lasso problem. The proposed algorithm makes full use of the sparsity of the solution. In each iteration, a subproblem with the same form as the original model is solved, but in a much smaller size. We apply the proximal point algorithm to solve the subproblem, which fully takes advantage of the two nonsmooth terms. Extensive numerical results demonstrate that the proposed algorithm (AS-PPA) is robust for different types of 
	noises, which verifies the attractive statistical property as shown in \cite{Wang}. Moreover, AS-PPA is also highly efficient, especially for the case of high-dimensional features, compared with other methods.}



\maketitle
\section{Introduction}\label{section1}
In this paper, we will design a highly efficient and robust algorithm for solving the convex composite optimization problems including the following high-dimensional rank lasso (hdr lasso) problem
\begin{equation}\label{eq19}
	\min\limits_{x\in\mathcal{X}}\ h\left(Ax-b\right)+\lambda
	\left\|x\right\|_1,
\end{equation}
where $h\ :\ \mathcal{Y}\to\Re$ is a nonsmooth convex function, $A\ :\ \mathcal{X}\to \mathcal{Y}$ is a linear map whose adjoint is denoted as $A^*$, and $\lambda>0$ is a given data. Here $\mathcal{X}$ and $\mathcal{Y}$ are two finite real dimensional Euclidean spaces, equipped with standard inner product $\langle\cdot,\cdot\rangle$ and $\left\|\cdot\right\|_1$ is the $L_1$ norm.

With the development of information technology and artificial intelligence, people face great challenges in data analysis due to the high dimensional features and the large-scale datasets. In the low-dimensional case, where the number of features is less than the number of samples, i.e. ($p<n$), the rank lasso problem was investigated by Wang and Li in 2009 \cite{low-rank-1}, where $h(u)$ in \eqref{eq19} is given by
\begin{equation}\label{eq55}
	h(u):=\frac{2}{n(n-1)}\mathop{\sum}\limits_{1\le i< j\le n}|u_i-u_j|,\ \text{where}\ u=b-Ax,
\end{equation}
and $A,\ b$ are given data defined by
\begin{equation*}
	A=\left[a_1,a_2,a_3,\dots,a_n\right]^\top\in\Re^{n\times p},\ b=\left[b_1;b_2;b_3;\dots;b_n\right]\in\Re^n.
\end{equation*}
That is,
\begin{equation}\label{eq30}
	\min\limits_{x\in\Re^p}\ \mathop{\sum}\limits_{1\le i< j\le n}\frac{2}{n(n-1)}\Big|(b_i-a_i^\top x)-(b_j-a_j^\top x)\Big|+\lambda\left\|x\right\|_1\\.
\end{equation}
The rank lasso model \eqref{eq30} basically replaces the quadratic loss function $\|Ax-b\|_2^2$ in the traditional lasso model by the loss function $h(\cdot)$ as defined in \eqref{eq55}. Minimizing the loss function $h(\cdot)$ in \eqref{eq55} is equivalent	to minimizing Jaeckel’s dispersion function with Wilcoxon scores:
\begin{equation*}
	\sqrt{12}\sum\limits_{i=1}^{n}\left[\frac{R(b_i-a_i^\top x)}{n+1}-\frac{1}{2}\right]\left(b_i-a_i^\top x\right),
\end{equation*}
where $R(b_i-a_i^\top x)$ denotes the rank of $b_i-a_i^\top x$ among $b_1-a_1^\top x,\dots ,b_n-a_n^\top x$. Therefore, \eqref{eq30} is referred to as the rank lasso model. The advantages of such rank-based methods include ``better power, efficiency at heavy-tailed distributions and robustness against various model violations and pathological data'' \cite[Page xv]{TP}. It is well-known that heavy-tailed distribution is ubiquitous in modern statistical analysis and machine learning problems, and may be caused by chance of extreme events or by the complex data generating process \cite{FJQ}. Heavy-tailed errors usually arise in climate data, insurance claim data, e-commerce data and many other scenarios \cite{Wang}. The traditional lasso model, though well-studied and popular used, has difficulty in dealing with heavy-tailed noises. In contrast, the hdr lasso model proposed by Wang et al. \cite{Wang} deals with the data where the number of features is much larger than the number of samples and it can adapt to a variety of errors, including heavy-tailed errors. Besides the above excellent property, the hdr lasso model also enjoys the following attractive property. A tuning-free parameter $\lambda$ was derived in \cite{Wang}, based on which, one does not need to solve a sequence of hdr lasso models with different $\lambda$'s. Fan et al. \cite{FJQcom} evaluated the tuning-free parameter $\lambda$ in \cite{Wang} is more interpretable, easier to select, and is independent of noise variance.

Due to the importance of problem \eqref{eq19}, a natural question from the numerical point of view is how to solve it efficiently, particularly in the scenario of the hdr lasso problem. To address this issue, below we will briefly review the existing methods for the lasso and rank lasso problems, respectively, then we will review the recent sieving strategies for these problems.

In terms of the lasso problem, many traditional algorithms are developed, such as the least angle regression (LARS) method \cite{Ef}, the pathwise coordinate descent method \cite{Tseng,Tseng2,Frie,Wu}. Since the proximal mapping of the $L_1$ norm is easy to calculate, algorithms such as fast iterative shrinkage-thresholding algorithm (FISTA) \cite{fista} and alternating direction method of multipliers (ADMM) \cite{ADMM} can also be applied to solve the lasso problem. Recently, Li et al. proposed the semismooth Newton augmented Lagrangian (Ssnal) \cite{LXD} method and fully exploited the second order sparsity of the problem through the semismooth Newton's method. Moreover, Ssnal can deal with a lasso-type of models where the loss function is smooth and convex. 

For the low-dimensional rank lasso (ldr lasso) problem, Wang and Li \cite{low-rank-1} combined the two parts of $L_1$ functions into an $L_1$ norm of a vector in the dimension of $(n(n-1)/2+p)$, which is then solved by transforming to the least absolute deviation (LAD) problem. Similarly, in \cite{ranklasso}, Kim et al. also transformed into LAD to deal with the low-dimensional datasets for $n>p$ and $p<50$, which can be solved by Barrodale-Roberts modified simplex algorithm \cite{BR}. In \cite{code}, Zoubir et al. used iterative reweighted least squares (IRWLS) to solve the rank lasso problem. IRWLS essentially approximates the $L_1$ norm by a sequence of the weighted $L_2$ regularized problems. Another way to deal with the rank lasso problem is to reformulate it as a linear programming problem, which may lead to the large scale of problems when dealing with the high-dimensional case. To summarize, the above recent efforts have been devoted to dealing with small-scale cases, i.e., the ldr lasso problem. In terms of solving the hdr lasso problem, very recently, Tang et al. \cite{WCJ} applied a proximal-proximal majorization-minimization algorithm. It can be seen from the numerical results that this method can efficiently solve the hdr lasso problem. 

Taking the sparsity of solutions into account, there are many literatures trying to use sieving strategies to solve the lasso problems. Ghaoui et al. \cite{Safe} proposed a strategy (referred to as the safe screening rule) in 2010 to eliminate features that are guaranteed not to exist after solving the learning problem. Tibshirani et al. \cite{Tib} proposed the so-called strong rules that are very simple and yet screened out far more predictors than the safe rules in \cite{Safe}. As explained in \cite{Tib}, the cost of this improvement is the possibility of mistakenly discarding active predictors. In 2014, Bonnefoy et al. \cite{dyna} proposed a dynamic safe rule which speeds up many optimization algorithms by reducing the size of the dictionary during iterations, discarding elements that are not part of the lasso solution. Wang et al. \cite{DPP} proposed an efficient and effective screening rule via dual polytope projections (DPP), which is mainly based on the uniqueness and nonexpansiveness of the optimal dual solution due to the fact that the feasible set in the dual space is a convex and closed polytope. In \cite{Fer}, Fercoq et al. leveraged the computations of duality gaps to propose a simple strategy unifying both the safe rule \cite{Safe} and the dynamic safe rule \cite{dyna}.

To summarize, all the above sieving algorithms aim at solving the lasso problem, and filter active variables through dual problems or dual gaps. In \cite{SDF}, an adaptive sieving strategy was proposed to solve the exclusive lasso problem where the $L_1$ regularizer is replaced by the weighted exclusive lasso regularizer. Different from the safe rules as in \cite{Tib,DPP,Safe}, the adaptive sieving strategy in \cite{SDF} starts with a few nonzero components and gradually adds more nonzero components by checking the KKT condition of the exclusive lasso problem. Very recently, based on the dual problem of \eqref{eq19}, Shang et al. \cite{SafeRank} applied the dual circumscribed sphere technique to build up a safe feature screening rule for the rank lasso problem in the scenario that the number of features $p$ is much larger than the number of samples $n$. Numerical results in \cite{SafeRank} indicate that this screening method can shorten the time by about half. In \cite{SDF2}, Yuan et al. generalized the results of \cite{SDF} to apply this adaptive sieving strategy to the case where $x$ after a linear map is sparse. 

To summarize, based on the above analysis, to solve \eqref{eq19} in the high-dimensional case, the efficient Ssnal proposed in \cite{LXD} can not be applied since $h(\cdot)$ in \eqref{eq19} is nonsmooth. Meanwhile, the above sieving strategies mentioned in \cite{Safe,Tib,DPP,Fer,SDF} can not be applied to the rank lasso problem directly. Therefore, a natural question is whether one can design a sieving strategy and proposes an efficient algorithm to solve the hdr lasso problem in \eqref{eq19}. This motivates the work in this paper. 

In this paper, inspired by \cite{SDF}, we design an adaptive sieving (AS) strategy for \eqref{eq19}. Here we would like to highlight that for the lasso problem, there is only one nondifferentiable term (which is the $L_1$ norm term). It leads to only one block of constraints in the dual problem, and the variables can be screened by checking whether such block of constraint is satisfied. However, for the rank lasso problem, it contains two nondifferentiable terms. The corresponding dual problem has two blocks of constraints. Due to the above observations, there are two challenges to design the AS strategy for \eqref{eq19}. One is the nondifferentiablity of the loss function $h(\cdot)$ in \eqref{eq19}, which brings difficulty in extending the AS strategy in \cite{SDF} to solving \eqref{eq19}. The second challenge is how to solve the resulting subproblem efficiently. We tackle the first challenge by introducing an extra variable to decouple the problem. By doing this, a new efficient sieving criterion is developed. For the second challenge, we apply the proximal point algorithm (PPA) to solving the subproblem. 

The contributions of the paper are mainly in three aspects. Firstly, we propose an AS strategy to reduce the scale of the hdr lasso problem, based on checking the KKT conditions of the problem. We also prove the convergence of the AS algorithm. Secondly, taking into account of the nonsmoothness of the loss function and the $L_1$ regularizer in the hdr lasso problem, we apply PPA for solving the subproblem. Finally, for each subproblem of PPA, we use the semismooth Newton's based augmented Lagrangian method. Numerical results demonstrate that the proposed adaptive sieving strategy is very efficient in reducing the scale of the hdr lasso problem and the resulting AS-PPA significantly outperforms other methods. 

The organization of the paper is as follows. In Sect. 2, we propose our AS strategy and prove the convergence. In Sect. 3, we propose PPA to solve the subproblems in AS and illustrate the convergence of PPA. In Sect. 4, we use the augmented Lagrange method (ALM) to solve each subproblem in PPA. In Sect. 5, we apply the semismooth Newton's method (SSN) to solve the subproblems in ALM. Numerical results on different datasets are presented in Sect. 6, which verify the efficiency of our proposed AS-PPA for the hdr lasso problem. Final conclusions are given in Sect. 7.

\textbf{Notations:} Denote $\mathbb{B}_\infty$ as the infinity norm unit ball. For $z\in\Re^n$, ${\rm sign}(z)=\left\{{\rm sign}(z_1),\dots,{\rm sign}(z_n)\right\}$, where ${\rm sign}(z_i)$ denotes the sign function of $z_i,\ i=1,\dots,n$. Denote ${\rm Diag}(z)$ as the diagonal matrix whose diagonal elements are given by vector $z$. We use ${\rm Diag}(Z_1,\dots,Z_n)$ as the block diagonal matrix whose $i$-th diagonal block is the matrix $Z_i$, $i=1,\dots,n$. $A_{J}$ represents the matrix formed by the corresponding columns of $A$ with $J$ as the index set. Denote the identity matrix of order $n$ by $I_n$. Similarly, $\textbf{O}_n$ and $\textbf{E}_n$ denote the $n\times n$ zero matrix and the $n\times n$ matrix of all ones, respectively. We use ${\rm dist}\left(x,\mathcal{C}\right)$ to denote the Euclidean distance of $x\in\Re^n$ to a set $\mathcal{C}\subset \Re^n$. That is, ${\rm dist}\left(x,\mathcal{C}\right):=\inf\limits_{x'\in\mathcal{C}}\left\|x-x'\right\|$.

\section{Adaptive Sieving Strategy (AS) for \eqref{eq19} and Its Convergence}\label{section2}
In this section, we will propose an adaptive sieving (AS) strategy to solve \eqref{eq19} and discuss the convergence of AS. 
\subsection{AS Strategy for \eqref{eq19}}\label{section2.2}
Due to the $L_1$ norm in \eqref{eq19}, the optimal solution of \eqref{eq19} enjoys the sparsity property. However, different from the exclusive lasso problem that is considered in \cite{SDF}, the function $h(\cdot)$ in \eqref{eq19} is not differentiable. Moreover, there is a composition $Ax$ as the variable of $h(\cdot)$ rather than $x$. Consequently, compared with the adaptive sieving strategy in \cite{SDF}, in order to design an adaptive sieving strategy for \eqref{eq19}, we face the following challenges: how to deal with the composition $Ax$ in $h(\cdot)$ and how to deal with the nondifferentiability of $h(\cdot)$. We leave the second challenge to the subproblem in Section \ref{section3}, where the proximal point method is applied to deal with the nondifferentiability of $h(\cdot)$.

To tackle the first challenge, we will introduce a new variable to decouple the problem. Let $u\in\Re^n$ be defined by $u=b-Ax$. We have the equivalent form of \eqref{eq19} as follows
\begin{equation}\label{eq49}
	\begin{split}
		\min\limits_{x\in\Re^p,u\in\Re^n}\ &h(u)+\lambda\left\|x\right\|_1\\
		\hbox{s.t.}\quad\quad&u=b-Ax.
	\end{split}
\end{equation}

To derive our AS strategy, we need the KKT conditions of \eqref{eq49}. The Lagrangian function of \eqref{eq49} can be written as 
\begin{equation*}\label{eq47}
	L(x,u;\alpha)=h(u)+\lambda \left\|x\right\|_1-\langle u-b+Ax,\alpha\rangle,
\end{equation*}
where $\alpha\in\Re^{n}$ is the Lagrange multiplier corresponding to the equality constraints in \eqref{eq49}. The KKT conditions of \eqref{eq49} are given as follows
	\begin{equation}\label{eq38}
		\begin{split}
			\begin{cases}
				&0\in\partial h(u)-\alpha,\\
				&0\in\lambda\partial \left\|x\right\|_1-A^\top\alpha,\\
				&u=b-Ax.\\
			\end{cases}
		\end{split}
	\end{equation}
	Define the following two sets: $\Omega_u:=\partial h(u),\ \Omega_x:=\partial \left\|x\right\|_1.$ Define the residual function ${\rm Res}\left(x,u,\alpha\right)$ as follows
	\begin{equation}\label{eq12}
			{\rm Res}\left(x,u,\alpha\right):=\Big(x-{\rm Prox}_{\lambda \|\cdot\|_1}\left(x+A^\top\alpha\right),u-{\rm Prox}_{h(\cdot)}\left(u+\alpha\right),\ u-b+Ax\Big).
	\end{equation}
	Therefore, \eqref{eq38} is equivalent to ${\rm Res}\left(x,u,\alpha\right)=0.$

The idea of our AS strategy is as follows. In each iteration $l$, let $I^l$ be the set of indices corresponding to the nonzero elements in $x^l$. To estimate the set of indices $I^{l+1}$ in the next iteration, we check whether the KKT conditions in \eqref{eq38} are satisfied approximately. That is,
	\begin{equation}\label{eq39}
		\|{\rm Res}\left(x^l,u^l,\alpha^l\right)\|\le\epsilon,
	\end{equation}
where $\epsilon>0$ is a prescribed parameter. If \eqref{eq39} fails, we select the following indices based on the violation of KKT conditions:
	\begin{equation*}
		J^{l+1}:=\left\{j\in\bar{I}^l\ \left|\ \left(A^{\top}\alpha^l\right)_j\notin\left(\lambda\partial\left\|x^l\right\|_1+\frac{\epsilon-\tilde{\epsilon}}{\sqrt{\left|\bar{I}^l\right|}}\mathbb{B}_\infty\right)_j\right.\right\},
	\end{equation*}
where $\tilde{\epsilon}<\epsilon$ is a prescribed parameter that represents the accuracy of solving the subproblem
	\begin{equation*}
		\min\limits_{x\in\Re^p,u\in\Re^n} \left\{h(u)+\lambda \left\|x\right\|_1\ \left|\ u=b-Ax,\ x_{\bar{I}^{l}}=0\right. \right\},
	\end{equation*}
	that is, $\|{\rm Res}(x^l_{I^l},u^l,\alpha^l)\|\le\tilde{\epsilon}$, where $\alpha^l$ is the Lagrange multiplier. Here, $\bar{I}^l=\left\{1,\dots,p\right\}\backslash I^l$. The set of nonzero elements in the next iteration is given by $I^{l+1}=I^l\bigcup J^{l+1}.$

The details of our AS strategy are summarized as below. 

\begin{algorithm}[!htbp]
	\caption{AS Algorithm for \eqref{eq19}}\label{alg1}
	\hspace*{0.02in} {\bf Input:} 
	$A\in\Re^{n\times p},\ b\in\Re^n,\ \lambda\in\Re^+$, tolerance $\epsilon>\tilde{\epsilon}>0$.\\
	\hspace*{0.02in} {\bf Initialization:}
	Choose a fast algorithm to find an approximate solution $\tilde{x}\in\Re^p$ to problem \eqref{eq19}.
	
		\hspace*{0.02in} {\bf S0:} Let $I^0:=\left\{j\ |\ \tilde{x}_j\neq 0,\ j=1,\dots,p\right\}.$ Solve the problem
		\begin{equation*}\label{eq4}
			(x^0,u^0)\in\mathop{\arg\min}\limits_{x\in\Re^p,u\in\Re^n} \left\{h(u)+\lambda \left\|x\right\|_1 \ \left|\ u=b-Ax,\ x_{\bar{I}^0}=0\right. \right\},
		\end{equation*}\\
		\hspace*{0.02in} {
			such that $\|{\rm Res}(x^0_{I^0},u^0,\alpha^0)\|\le\tilde{\epsilon}$, where $\alpha^0$ is the Lagrange multiplier corresponding to the equality constraints $u=b-Ax$, $U:=\{1,\dots,p\}$, $\bar{I}^0:=U\backslash I^0$. Set $l=0$.}\\
		\hspace*{0.02in} {\bf S1:} Create $J^{l+1}$ by
		\begin{equation*}\label{eq2}
			J^{l+1}:=\left\{j\in\bar{I}^l\ \left|\ \left(A^{\top}\alpha^l\right)_j\notin\left(\lambda\partial\left\|x^l\right\|_1+\frac{\epsilon-\tilde{\epsilon}}{\sqrt{\left|\bar{I}^l\right|}}\mathbb{B}_\infty\right)_j\right.\right\},
		\end{equation*}
		Let $I^{l+1}= I^l\cup J^{l+1}$.\\
		\hspace*{0.02in} {\bf S2:} Solve the following problem:
		\begin{equation}\label{eq6}	
			(x^{l+1},u^{l+1})\in\mathop{\arg\min}\limits_{x\in\Re^p,u\in\Re^n} \left\{h(u)+\lambda \left\|x\right\|_1\ \left|\ u=b-Ax,\ x_{\bar{I}^{l+1}}=0\right. \right\},
		\end{equation}
		such that $\|{\rm Res}(x^{l+1}_{I^{l+1}},u^{l+1},\alpha^{l+1})\|\le\tilde{\epsilon}$,\ where $\alpha^{l+1}$ is the Lagrange multiplier corresponding to the equality constraints $u=b-Ax$.\\
	\hspace*{0.02in} {\bf S3:} Let $\bar{I}^{l+1}=U\backslash I^{l+1},\ x^{l+1}_{\bar{I}^{l+1}}=0$.\\
		\hspace*{0.02in} {\bf S4:} If $\|{\rm Res}\left(x^{l+1},u^{l+1},\alpha^{l+1}\right)\|\le\epsilon$, stop; otherwise, $l:=l+1$, go to {\bf S1}.
\end{algorithm}
\subsection{Convergence of Algorithm \ref{alg1}}\label{section2.3}
In Algorithm \ref{alg1}, we check the violation of KKT conditions by $J^{l+1}$, and gradually add the elements in $\bar{I}^l$ to $I^{l+1}$. As long as the required accuracy is not achieved, the set $J^{l+1}$ is not empty. This is demonstrated in the following theorem. 

\begin{theorem}\label{th1}
	If the residual function ${\rm Res}\left(x^l,u^l,\alpha^l\right)$ does not satisfy \eqref{eq39}, the set $J^{l+1}$ is not empty.
\end{theorem}
\begin{proof}
	For contradiction, assume that $J^{l+1}=\emptyset$. There is
		\begin{equation*}
			\left(A^{\top}\alpha^l\right)_j\in\left(\lambda\partial\left\|x^l\right\|_1+\frac{\epsilon-\tilde{\epsilon}}{\sqrt{\left|\bar{I}^l\right|}}\mathbb{B}_\infty\right)_j,\quad\forall\ j\in\bar{I}^l.
		\end{equation*}
		Then there exists $\tilde{\delta}\in\Re^{\left|\bar{I}^l\right|}$, such that $\|\tilde{\delta}\|_\infty\le \frac{\epsilon-\tilde{\epsilon}}{\sqrt{\left|\bar{I}^l\right|}}$ and $\Big(A_{\bar{I}^l}\Big)^{\top}\alpha^l\in\lambda\partial\left\|x^l_{\bar{I}^l}\right\|_1-\tilde{\delta}$.
		So we can get
		\begin{equation*}
			\Big(A_{\bar{I}^l}\Big)^{\top}\alpha^l+\tilde{\delta}\in\lambda\partial\left\|x^l_{\bar{I}^l}\right\|_1,
		\end{equation*}
		which means
		\begin{equation}\label{eq3}
			x^l_{\bar{I}^l}={\rm Prox}_{\lambda \|\cdot\|_1}\left(x^l_{\bar{I}^l}+\Big(A_{\bar{I}^l}\Big)^{\top}\alpha^l+\tilde{\delta}\right).
		\end{equation}
	
		On the other hand, note that $\left(x^{l},u^{l},\alpha^l\right)$ is the approximate solution of the following problem
		\begin{equation*}\label{eq17}	
			\min\limits_{x\in\Re^p,u\in\Re^n}\left\{h(u)+\lambda \left\|x\right\|_1\ \left|\ u=b-Ax,\ x_{\bar{I}^{l}}=0\right. \right\},
		\end{equation*}
		which satisfies $\|{\rm Res}(x^l_{I^l},u^l,\alpha^l)\|\le\tilde{\epsilon}$, that is
		\begin{equation*}\label{eq15}
			\left\|\Big(x^l_{I^l}-{\rm Prox}_{\lambda \|\cdot\|_1}\left(x^l_{I^l}+\Big(A_{I^l}\Big)^\top\alpha^l\right),u^l-{\rm Prox}_{h(\cdot)}\left(u^l+\alpha^l\right),\ u^l-b+A_{I^l}x^l_{I^l}\Big)\right\|\le\tilde{\epsilon}.
		\end{equation*}
Together with \eqref{eq3}, we have
		\begin{equation*}
			\begin{split}
				&\|{\rm Res}\left(x^l,u^l,\alpha^l\right)\|\\
				=&\left\|\left(x^l-{\rm Prox}_{\lambda \|\cdot\|_1}\left(x^l+A^\top\alpha^l\right),u^l-{\rm Prox}_{h(\cdot)}\left(u^l+\alpha^l\right),u^l-b+Ax^l\right)\right\|\\
				=&\Bigg\|\left(x^l_{I^l}-{\rm Prox}_{\lambda \|\cdot\|_1}\left(x^l_{I^l}+\Big(A_{I^l}\Big)^\top\alpha^l\right),x^l_{\bar{I}^l}-{\rm Prox}_{\lambda \|\cdot\|_1}\left(x^l_{\bar{I}^l}+\Big(A_{\bar{I}^l}\Big)^\top\alpha^l\right)\right.,\\				&\left.u^l-{\rm Prox}_{h(\cdot)}\left(u^l+\alpha^l\right),u^l-b+A_{I^l}x^l_{I^l}\right)\Bigg\|\\
				\le&\tilde{\epsilon}+\left\|{\rm Prox}_{\lambda \|\cdot\|_1}\left(x^l_{\bar{I}^l}+\Big(A_{\bar{I}^l}\Big)^{\top}\alpha^l+\tilde{\delta}\right)-{\rm Prox}_{\lambda \|\cdot\|_1}\left(x^l_{\bar{I}^l}+\Big(A_{\bar{I}^l}\Big)^{\top}\alpha^l\right)\right\|\\
				\le&\tilde{\epsilon}+\|\tilde{\delta}\|\\
				\le&\tilde{\epsilon}+\sqrt{\left|\bar{I}^l\right|}\frac{\epsilon-\tilde{\epsilon}}{\sqrt{\left|\bar{I}^l\right|}}\quad\left(\text{by}\ \left\|\tilde{\delta}\right\|_\infty\le \frac{\epsilon-\tilde{\epsilon}}{\sqrt{\left|\bar{I}^l\right|}}\right)\\
				\le&\epsilon.
			\end{split}
	\end{equation*}
	
		It implies that $\|{\rm Res}\left(x^l,u^l,\alpha^l\right)\|\le \epsilon$, which means if $J^{l+1}=\emptyset$, the condition \eqref{eq39} must be satisfied. It contradicts with the assumption that $\left(x^l,u^l,\alpha^l\right)$ violates \eqref{eq39}. Therefore if the condition \eqref{eq39} is not satisfied, the set $J^{l+1}$ must be nonempty. 
\end{proof}
\begin{remark}
	Theorem $\ref{th1}$ shows that Algorithm $\ref{alg1}$ is well-defined. In other words, notice that the number of the components in $x$ is finite, Algorithm $\ref{alg1}$ must be terminated in a finite number of iterations and meets the required stopping criterion in $\eqref{eq39}$. 
\end{remark}
\begin{remark}
	 Our AS strategy is different from that in \cite{SDF2}. Firstly, in \cite{SDF2}, Yuan et al. exploit the sparsity of $x$ after a linear map, based on which, the AS strategy is designed. However, for problem \eqref{eq19}, we have the sparsity of $x$. The term $b-Ax$ in $h(\cdot)$ does not enjoy the sparsity property, but brings challenge in designing AS strategy. Secondly, we have one extra nonsmooth term $h(\cdot)$ in \eqref{eq19} whereas in \cite{SDF2}, there is only one nonsmooth term.
\end{remark}

\subsection{An Enhanced Adaptive Sieving Technique}\label{sec2.3}

	In this part, we adopt the idea in \cite{SDF2}, that is an enhanced technique to construct the best multiplier. For simplicity, in this subsection we denote the optimal solution $(x^{l+1},u^{l+1})$ to the problem \eqref{eq6} as $(\tilde{x},\tilde{u})$, denote the index sets $I^{l+1}$ and $\bar{I}^{l+1}$ as $I$ and $\bar{I}$ respectively. Define a new index set $\tilde{I}$ as follows: 
\begin{equation*}
	\tilde{I}:=\{i\in[p]\ |\ \tilde{x}_i\neq0\},\ \tilde{I}^C=\{i\in[p]\ |\ \tilde{x}_i=0\}.
\end{equation*}
It is obvious that $\bar{I}\subset \tilde{I}^C$ and $\tilde{I}\subset I$. In fact, $(\tilde{u},\tilde{x})$ is an optimal solution to the following constrained optimization problem:
\begin{equation}\label{eq27}
	\begin{split}
		\min\limits_{x\in\Re^p,u\in\Re^n}\ &h(u)+\lambda\left\|x\right\|_1\\
		\hbox{s.t.}\quad\quad&u=b-Ax,\\
		&x_{\tilde{I}^C}=0.
	\end{split}
\end{equation}
The corresponding KKT conditions of \eqref{eq27} is
\begin{equation}\label{eq23}
	\begin{cases}
		\left(A^\top \beta\right)_{\tilde{I}}\in\lambda\partial \|\tilde{x}_{\tilde{I}}\|_1,\\
		\left(A^\top \beta\right)_{\tilde{I}^C}\in\lambda\partial \|\tilde{x}_{\tilde{I}^C}\|_1-y,\\
		\beta\in\partial h(\tilde{u}),\\
		\tilde{u}=b-A\tilde{x},
	\end{cases}
\end{equation}
where $\beta$ and $y$ are the corresponding Lagrange multipliers, respectively.

Since $\partial \|\tilde{x}_{\tilde{I}^C}\|_1$ and $\partial h(\tilde{u})$ are both sets, there may be more than one pair of Lagrange multipliers $(\beta,y)$ that satisfy the above KKT conditions. If there exists a $\beta$ such that $(\beta,0)$ is a Lagrange multiplier pair satisfying \eqref{eq23}, then the current solution $\tilde{x}$ is an optimal solution to \eqref{eq49}. Based on this consideration, we propose to construct the pair $(\beta, y)$ such that $y$ has the minimum Euclidean norm.

We further analyze the properties brought by $\tilde{I}$ and $I$. Since $\tilde{I}\subset I$, we can get $\left(A^\top \beta\right)_{\tilde{I}}=\left(A^\top \alpha\right)_{\tilde{I}}$. Since $\tilde{x}_{\tilde{I}^C}=0$, $\partial \|\tilde{x}_{\tilde{I}^C}\|_1=\left[-\lambda e_{|\tilde{I}^C|},\lambda e_{|\tilde{I}^C|}\right]$, where $e_{|\tilde{I}^C|}=[1,1,\dots,1]\in\Re^{|\tilde{I}^C|}$. According to \eqref{eq28} in the appendix \ref{app1}, $\partial h(\tilde{u})$ takes the following form
\begin{equation}\label{eq29}
	\partial h\left(\tilde{u}\right)=\left\{g+s\in\Re^n\ |\ g_i=c_i-d_i,i=1,\dots,n,\ s=\sum\limits_{k=1}^{m}s^k\right\},
\end{equation}
where $c_i$ is the number of $\tilde{u}$ less than $\tilde{u}_i$, $d_i$ is the number of $\tilde{u}$ greater than $\tilde{u}_i$, $m$ is the number of equal groups in the vector $\tilde{u}$, $Q_k$ is the index corresponding to each equality group and each $s^k\in\Re^n,\ s^k_{\bar{Q}_k}=0,\ \sum\limits_{p\in Q_k} s^k_p=0,\ s^k_p\in[-|Q_k|+1,|Q_k|-1],\ \forall\ p\in Q_k$.
It can be noted that $s_{\bar{Q}}=0$, where $Q=\bigcup\limits_{k=1}^m Q_k$. Since $\left(\partial h(\tilde{u})\right)_{\bar{Q}}$ is a vector equal to $g_{\bar{Q}}$. We can just set $\beta_{\bar{Q}}=g_{\bar{Q}}$. Thus the KKT conditions \eqref{eq23} is equal to
\begin{equation*}
	\begin{cases}
		A_{Q,\tilde{I}}^\top \tilde{\beta}+A_{\bar{Q},\tilde{I}}^\top g_{\bar{Q}}=\left(A^\top \alpha\right)_{\tilde{I}},\\
		A_{Q,\tilde{I}^C}^\top \tilde{\beta}+A_{\bar{Q},\tilde{I}^C}g_{\bar{Q}}+y\in\left[-\lambda e_{|\tilde{I}^C|},\lambda e_{|\tilde{I}^C|}\right],\\
		\tilde{\beta}\in\left(\partial h(\tilde{u})\right)_Q,\\
		\tilde{u}=b-A\tilde{x}.
	\end{cases}
\end{equation*}

Then we can get the Lagrange multiplier $(\beta,y)$ by solving the following problem:
\begin{equation}\label{eq21}
	\begin{split}
		\min\limits_{\tilde{\beta}\in\Re^{|Q|},y\in\Re^{|\tilde{I}^C|}}\ &\frac{1}{2}\|y\|^2\\
		\hbox{s.t.}\quad\quad&A_{Q,\tilde{I}}^\top \tilde{\beta}+A_{\bar{Q},\tilde{I}}^\top g_{\bar{Q}}=\left(A^\top \alpha\right)_{\tilde{I}},\\
		&A_{Q,\tilde{I}^C}^\top \tilde{\beta}+A_{\bar{Q},\tilde{I}^C}g_{\bar{Q}}+y\in\left[-\lambda e_{|\tilde{I}^C|},\lambda e_{|\tilde{I}^C|}\right],\\
		&\tilde{\beta}\in\left(\partial h(\tilde{u})\right)_Q.
	\end{split}
\end{equation}
Actually, it is a quadratic programming problem that can be solved by many mature solvers, such as Gurobi, etc. After obtaining $(\tilde{\beta},y)$ from \eqref{eq21}, we can construct $\beta$ by letting $\beta_Q=\tilde{\beta},\beta_{\bar{Q}}=g_{\bar{Q}}$. Now, we present the enhanced adaptive sieving technique in Algorithm \ref{alg5}.
\begin{algorithm}[!htbp]
	\caption{An Enhanced Adaptive Sieving Technique for \eqref{eq19}}\label{alg5}
	\hspace*{0.02in} {\bf Input:} An approximate solution of \eqref{eq6}, denoted as $(\tilde{u},\tilde{x})$, and the Lagrange multiplier $\alpha$ corresponding to the equality constraints $\tilde{u}=b-A\tilde{x}$.\\
	\hspace*{0.02in} {\bf S1:} Define new index sets $\tilde{I}:=\{i\in[p]\ |\ \tilde{x}_i\neq0\},\ \tilde{I}^C=\{i\in[p]\ |\ \tilde{x}_i=0\}$, calculate $\partial h(\tilde{u})$ and $g,\ s,\ c,\ d,\ Q$ in \eqref{eq29} .\\
	\hspace*{0.02in} {\bf S2:} Find a solution of \eqref{eq21}, denote as $\tilde{\beta},\ y$.\\
	\hspace*{0.02in} {\bf S3:} Let $\beta_Q=\tilde{\beta},\beta_{\bar{Q}}=g_{\bar{Q}}$, stop.
\end{algorithm}

\section{Proximal Point Algorithm for Subproblem $\eqref{eq6}$}\label{section3}
In order to make full use of the AS strategy, we must design an efficient method to solve subproblem \eqref{eq6}, which is the second challenge of designing an efficient sieving algorithm for \eqref{eq19}. We will address this challenge in this section. Comparing problem \eqref{eq19} with subproblem \eqref{eq6}, the differences lie in the dimension of $x_{I^l}$ and the corresponding coefficient matrix $A_{I^{l+1}}$. That is, subproblem \eqref{eq6} can be regarded as solving problem \eqref{eq19} in a subspace indexed by $I^l$. Therefore, we consider how to solve problem \eqref{eq19} instead.

Since the two terms in \eqref{eq19} are both nonsmooth, we apply PPA \cite{Rockafellar} to solve it. For any starting point $x^0\in\Re^p$, PPA generates a sequence $\left\{x^k\right\}$ by solving the following problem \cite{Rockafellar}:
\begin{equation}\label{eq11}
	\begin{split}
		\mathcal{P}_k(x^k):=\mathop{\arg\min}\limits_{x\in\Re^p}\ &h(b-Ax)+\lambda \|x\|_1+\frac{1}{2\sigma_k}\left\|x-x^k\right\|^2.
	\end{split}
\end{equation}
Here $\left\{\sigma_{k}\right\}$ is a sequence of positive real numbers such that $0<\sigma_{k}\uparrow\sigma_\infty\le +\infty$. Now, we present PPA in Algorithm \ref{alg2}.

\begin{algorithm}[!htbp]
	\caption{PPA for \eqref{eq19}}\label{alg2}
	\hspace*{0.02in} {\bf Input:} 
	$A\in\Re^{n\times p},\ b\in\Re^n,\ \lambda\in\Re^+,\ \sigma_0>0$.\\
	\hspace*{0.02in} {\bf Initialization:} choose $x^0\in\Re^p,\ k:=0.$\\
	\hspace*{0.02in} {\bf S1:} Find an approximate solution of \eqref{eq11}, denoted as $x^{k+1}$.\\
	\hspace*{0.02in} {\bf S2:} If the stopping criteria of PPA are satisfied, stop; otherwise, go to {\bf S3}.\\
	\hspace*{0.02in} {\bf S3:} Update $\sigma_{k+1}\uparrow\sigma_\infty\le\infty;\ k:=k+1.$
\end{algorithm}

In order to analyze the convergence, we first give the termination criteria. Define $\mathcal{T}:\ \Re^p\rightrightarrows\Re^p$ as $\mathcal{T}(x)=-A^\top\partial h(b-Ax)+\lambda\partial \|x\|_1$. According to \cite[Theorem 23.9 and Theorem 23.8]{Rockafellar}, $x^*\in\mathop{\arg\min}\limits_{x\in\Re^p}\ h(b-Ax)+\lambda \|x\|_1$ is equal to $0\in \mathcal{T}(x^*)$. Define $S_k\left(x\right):=-A^\top\partial h(b-Ax)+\lambda\partial \|x\|_1+\frac{1}{\sigma_{k}}\left(x-x^{k}\right)$, and
\begin{equation*}
	\begin{split}
		{\rm dist}\left(0,S_k\left(x^{k+1}\right)\right)=&{\rm dist}\left(0,-A^\top\partial h(b-Ax)+\lambda\partial \|x\|_1+\frac{1}{\sigma_{k}}\left(x-x^{k}\right)\right)\\
		=&{\rm dist}\left(0,-A^\top\left(\partial h(b-Ax)-\alpha\right)-A^\top \alpha+\lambda\partial \|x\|_1+\frac{1}{\sigma_{k}}\left(x-x^{k}\right)\right)\\
		\le&{\rm dist}\left(0,-A^\top\left(\partial h(b-Ax)-\alpha\right)\right)+{\rm dist}\left(0,-A^\top \alpha+\lambda\partial \|x\|_1+\frac{1}{\sigma_{k}}\left(x-x^{k}\right)\right)\\
		\le&\left\|A\right\|{\rm dist}\left(0,\left(\partial h(b-Ax)-\alpha\right)\right)+{\rm dist}\left(0,-A^\top \alpha+\lambda\partial \|x\|_1+\frac{1}{\sigma_{k}}\left(x-x^{k}\right)\right)\\
		\le&{\rm dist}\left(0,\begin{pmatrix}
			I&0&0\\0&\|A\|&0\\0&0&0
		\end{pmatrix}
		\tilde{S}_k\left(x^{k+1},u^{k+1},\alpha\right)\right),\quad\text{(if $u^{k+1}=b-Ax^{k+1}$)},
	\end{split}
\end{equation*}
where 
\begin{equation*}
	\begin{split}
		\tilde{S}_k\left(x,u,\alpha\right)&=\tilde{\mathcal{T}}\left(x,u,\alpha\right)+\frac{1}{\sigma_k}\begin{bmatrix}
			x-x^{k}\\
			0_n\\
			0_n\\
		\end{bmatrix},\\
		\tilde{\mathcal{T}}(x,u,\alpha)&=\Big\{\left(-A^\top\alpha+t,v-\alpha,-u+b-Ax\right)\in\Re^p\times\Re^n\times\Re^{n}\ \Big|\ t\in\lambda\partial \left\|x\right\|_1,\ v\in\partial h(u) \Big\},
	\end{split}
\end{equation*}
and $\alpha$ is the Lagrange multiplier. 

In practice, we adopt the following criteria, which are analogous to those proposed in \cite{Rockafellar}:
	\begin{equation*}
		\begin{split}
			&(A)\quad {\rm dist}\left(0,\begin{pmatrix}
				I&0&0\\0&\|A\|&0\\0&0&0
			\end{pmatrix}
			\tilde{S}_k\left(x^{k+1},u^{k+1},\alpha\right)\right)\le \frac{\gamma_k}{\sigma_k},\ \gamma_k\ge 0,\ \sum_{k=0}^\infty\gamma_k<\infty,\\
			&(B)\quad {\rm dist}\left(0,\begin{pmatrix}
				I&0&0\\0&\|A\|&0\\0&0&0
			\end{pmatrix}
			\tilde{S}_k\left(x^{k+1},u^{k+1},\alpha\right)\right)\le \left(\frac{\delta_k}{\sigma_k}\right)\left\|x^{k+1}-x^k\right\|,\ \delta_k\ge 0,\ \sum_{k=0}^\infty\delta_k<\infty.\\
		\end{split}
\end{equation*}
%
	According to the above definition of $\mathcal{T}$ and its properties, we can get the convergence result in \cite{Rockafellar}.
	\begin{theorem}
		Let $\left\{x^k\right\}$ be an infinite sequence generated by Algorithm $\ref{alg2}$ under stopping criterion (A) for solving \eqref{eq19}, and $h(\cdot)$ is a proper convex function. The sequence $\left\{x^k\right\}$ converges to a solution $\bar{x}$ of \eqref{eq19}. Furthermore, if the criterion (B) is also executed in Algorithm $\ref{alg2}$, there exists $\bar{k}\ge 0$ such that for all $k\ge\bar{k}$, there is
		\begin{equation}\label{eq13}
			\left\|x^{k+1}-\bar{x}\right\|\le \theta_k\left\|x^k-\bar{x}\right\|,
		\end{equation}
		the convergence rate is given by
		\begin{equation*}
			1>\theta_k:=\frac{\mu_k+\delta_k}{1-\delta_k}\to\mu_\infty=\frac{\kappa}{\left(\kappa^2+\sigma_\infty^2\right)^{1/2}},\ \left(\mu_\infty=0\ \text{if}\ \sigma_\infty=\infty\right),\ \mu_k:=\kappa\left(\kappa^2+\sigma_k^2\right)^{-1/2}.
		\end{equation*}
		where $\kappa\ge0$ is a constant.
	\end{theorem}
	\begin{proof}
		In order to obtain similar convergence results as in \cite{Rockafellar}, we only need to verify three conditions: (i) the defined operator $\mathcal{T}$ is maximal monotone; (ii) $\left\{x^k\right\}$ is bounded; (iii) $\mathcal{T}^{-1}$ satisfies the error bound condition at the zero point.
		
		Notice that $\mathcal{T}$ is the subdifferential of the hdr lasso problem defined in \eqref{eq19} and the objective function in \eqref{eq19} is a continuous proper convex function. Therefore, $\mathcal{T}$ is maximal monotone \cite[Theorem A]{mono}.
		
		According to the convexity of the hdr lasso problem, we know that there is at least one solution $\bar{x}$ satisfy $\mathcal{T}(\bar{x}) = 0$. Together with stopping criterion (A) and $\mathcal{T}$ is maximal monotone, the sequence $\{x^k\}$ is bounded \cite{Rockafellar}.
	    
		By \cite[Proposition 2.2.4]{SJ}, we know that $\mathcal{T}$ is a polyhedral multivalued function. Therefore, the error bound condition holds at the origin with modulus $\kappa\ge 0$ \cite{Robinson1981}. That is, $\exists\ \epsilon'>0$ such that $\forall\ x\in\left\{x\ |\ {\rm dist}\left(0,\mathcal{T}(x)\right)<\epsilon'\right\}$, there is
		\begin{equation*}
			{\rm dist}\left(x, \mathcal{T}^{-1}(0)\right)\le \kappa {\rm dist}\left(0, \mathcal{T}(x)\right).
		\end{equation*}
	    Then we can get $\left\{x^k\right\}$ converges to a solution $\bar{x}$ of \eqref{eq19} by \cite[Theorem 1]{Rockafellar}. If the criterion (B) is also executed, the rate of convergence in \eqref{eq13} is established by \cite[Theorem 2]{Rockafellar}.
	\end{proof}
\section{ALM for \eqref{eq11}}\label{section4}
In this part, we present ALM for \eqref{eq11} and discuss its convergence. Note that what we need to consider now is how to efficiently solve subproblem \eqref{eq11} in PPA at each iteration $k$. By introducing $z\in\Re^p,\ u\in\Re^n$, we solve the following equivalent form of \eqref{eq11}:
\begin{equation}\label{eq46}
	\begin{split}
		\min\limits_{x\in\Re^p,u\in\Re^n,z\in\Re^p}\ &h(u)+\lambda \|z\|_1+\frac{1}{2\sigma_k}\left\|x-x^k\right\|^2\\
		\hbox{s.t.}\ \ \ \quad\quad &u=b-Ax,\\
		&z=x.
	\end{split}
\end{equation}	

Note that there are only equality constraints in \eqref{eq46}. A natural choice to solve \eqref{eq46} is the augmented Lagrangian method. The augmented Lagrangian function of \eqref{eq11} is given as follows
\begin{equation*}\label{eq33}
	\begin{split}
		&L_\rho\left(x,u,z,\alpha\right)\\
		=&h\left(u\right)+\lambda\left\|z\right\|_1-\langle u-b+Ax,\alpha_1\rangle-\langle z-x,\alpha_2\rangle+\frac{\rho}{2}\left\|u-b+Ax\right\|^2+\frac{\rho}{2}\left\|z-x\right\|^2+\frac{1}{2\sigma_k}\left\|x-x^k\right\|^2\\
		=&\underbrace{h\left(u\right)+\frac{\rho}{2}\left\|u-b+Ax-\frac{1}{\rho}\alpha_1\right\|^2}_{F_1(u,x)}+\underbrace{\lambda\left\|z\right\|_1+\frac{\rho}{2}\left\|z-x-\frac{1}{\rho}\alpha_2\right\|^2}_{F_2(z,x)}\underbrace{+\frac{1}{2\sigma_k}\left\|x-x^k\right\|^2}_{F_3(x)}-\frac{1}{2\rho}\left\|\alpha_1\right\|^2-\frac{1}{2\rho}\left\|\alpha_2\right\|^2,\\
	\end{split}
\end{equation*}
where $\rho>0$, and $\alpha=(\alpha_1,\alpha_2)\in\Re^{n+p}$. Here $\alpha_1\in\Re^n,\alpha_2\in\Re^p$ are the Lagrangian multipliers corresponding to the two types of equality constraints in \eqref{eq46}. 

At the $j$-th iteration, the augmented Lagrangian method generates $\left(x^{j+1},u^{j+1},z^{j+1},\alpha^{j+1}\right)$ in the following way $\left(\rho_j>0\right)$:

\begin{subnumcases}
	{ }\label{a}	&$\left(x^{j+1},u^{j+1},z^{j+1}\right)=\mathop{\arg\min}\limits_{x\in\Re^p,u\in\Re^{n},z\in\Re^p} L_{\rho_j}\left(x,u,z,\alpha^j\right),$\\ 
	\label{b}&$\alpha^{j+1}=\alpha^{j}-\rho_j\begin{bmatrix}
		u^{j+1}-b+Ax^{j+1}\\
		z^{j+1}-x^{j+1}
	\end{bmatrix}.$
\end{subnumcases}

The key point is how to solve \eqref{a}. Note that $L_{\rho_j}$ is a function of three variables $x, u, z$. It is difficult to find the minimum value of the function with respect to the three variables simultaneously. Notice that there is no coupling term between $u$ and $z$, and they are only related to $x$ respectively. Another fact is that with $x$ fixed, one can have closed form solutions for $u$ and $z$ respectively, by using the property of Moreau-Yosida regularization. Therefore, we first represent $u$ and $z$ by $x$ through the Moreau-Yosida regularization, denoted by $u=f_u(x),\ z=f_z(x)$, respectively. Then we solve the resulting minimization problem with respect to $x$, that is, $\min\limits_{x\in\Re^p} L_{\rho_j}\left(x,f_u(x),f_z(x)\right)$.

Firstly, we introduce the properties of Moreau-Yosida regularization. For any given proper closed convex function $q:\Re^n\to (-\infty,+\infty]$, the Moreau-Yosida regularization of $q$ is defined by
\begin{equation}\label{eq14}
	M_q(x):=\min\limits_{y\in\Re^n}\left\{q(y)+\frac{1}{2}\left\|y-x\right\|^2\right\},\ \forall\ x\in\Re^n.
\end{equation}
The proximal mapping associated with $q$ is the unique minimizer of \eqref{eq14} denoted by ${\rm Prox}_q(x)$ and it is Lipschitz continuous with modulus 1\ \cite{J}. Moreover, $M_q(x)$ is a continuously differentiable convex function\ \cite{Lemar}, and its gradient is given by $\nabla M_q(x)=x-{\rm Prox}_q (x),\ \forall\ x\in\Re^n.$

Next, we show how to use Moreau-Yosida regularization to derive $f_u(x)$ and $f_z(x)$. Denote $f_1(x)=b-Ax+\frac{1}{\rho_j}\alpha^j_1$ and $f_2(x)=x+\frac{1}{\rho_j}\alpha^j_2$. There is
\begin{equation}\label{eq20}
	\begin{split}
		\min\limits_{u\in\Re^n}L_{\rho_j}\left(x,u,z,\alpha^j\right)\Leftrightarrow&\min\limits_{u\in\Re^n}\left\{F_1(u,x)\right\}\\
		=\:&\rho_j\min\limits_{u\in\Re^n}\left\{\frac{1}{\rho_j} h(u)+\frac{1}{2}\left\|u-f_1(x)\right\|^2\right\}\\
		=\:&\rho_j M_{\frac{1}{\rho_j}h(\cdot)}\left(f_1(x)\right)\triangleq\:G_1(x),\\
	\end{split}
\end{equation}
where $M_{\tau h(\cdot)}(\cdot)$ is defined as in \eqref{eq14}. The optimal solution to \eqref{eq20} is $u={\rm Prox}_{\frac{1}{\rho_j} h(\cdot)}\left(f_1(x)\right)$. Similarly, for $z$, there is
\begin{equation}\label{eq24}
	\begin{split}
		\min\limits_{z\in\Re^p}L_{\rho_j}\left(x,u,z,\alpha^j\right)\Leftrightarrow&\min\limits_{z\in\Re^p}\left\{F_2(z,x)\right\}\\
		=\:&\rho_j\min\limits_{z\in\Re^p}\left\{\frac{\lambda}{\rho_j}\left\|z\right\|_1+\frac{1}{2}\|z-f_2(x)\|^2\right\}\\
		=\:&\rho_jM_{\frac{\lambda}{\rho_j}\left\|\cdot\right\|_1}\left(f_2(x)\right)\triangleq\:G_2(x),\\
	\end{split}
\end{equation}
where the unique minimizer of \eqref{eq24} is $z={\rm Prox}_{\frac{\lambda}{\rho_j}\left\|\cdot\right\|_1}\left(f_2(x)\right)$. Then, we have 
\begin{equation*}
	\begin{split}
		&\min\limits_{x\in\Re^p,u\in\Re^{n},z\in\Re^p} L_{\rho_j}\left(x,u,z,\alpha^j\right)\\
		\Leftrightarrow& \min\limits_{x\in\Re^p}\left\{\min\limits_{u\in\Re^n}\left\{F_1(u,x)\right\}+\min\limits_{z\in\Re^p}\left\{F_2(z,x)
		\right\}+F_3(x)\right\}\\
		\Leftrightarrow& \min\limits_{x\in\Re^p}\left\{G_1(x)+G_2(x)+F_3(x)\right\}\quad \left(\text{by}\ \eqref{eq20}\ \text{and}\ \eqref{eq24}\right)\\
		:=&\min\limits_{x\in\Re^p}\varphi_j(x).
	\end{split}
\end{equation*}

Now we are ready to present ALM for \eqref{eq11} in Algorithm \ref{alg3}.

\begin{algorithm}[!htbp]
	\caption{ALM for \eqref{eq11}}\label{alg3}
	\hspace*{0.02in} {\bf Input:} 
	$A\in\Re^{n\times p},\ b\in\Re^n,\ \lambda\in\Re^+,\ \sigma_{k}>0,\ x^k\in\Re^p,\ \rho_0>0.$\\
	\hspace*{0.02in} {\bf Initialization:}
	choose $x^0\in\Re^p,\alpha^0\in\Re^{n}$, set $j=0$.\\
	\hspace*{0.02in} {\bf S1:} Solve
	\begin{equation}\label{eq22}
		x^{j+1}=\mathop{\rm \arg \min}\limits_{x\in\Re^p}\varphi_j(x).
	\end{equation}\\
	\hspace*{0.02in} {\bf S2:} Calculate $u^{j+1}={\rm Prox}_{\frac{1}{\rho_j} h(\cdot)}\left(f_1\left(x^{j+1}\right)\right),\ z^{j+1}={\rm Prox}_{\frac{\lambda}{\rho_j}\left\|\cdot\right\|_1}\left(f_2\left(x^{j+1}\right)\right).$\\
	\hspace*{0.02in} {\bf S3:} If the stopping criteria are satisfied, stop; otherwise, go to {\bf S4}.\\
	\hspace*{0.02in} {\bf S4:} Update $\alpha^{j+1}$ by \eqref{b}; $j:= j+1.$
\end{algorithm}
We will
use the following stopping criteria considered by Rockafellar \cite{Rockafellar}, \cite{RT1} for terminating Algorithm \ref{alg3}:
\begin{equation*}\label{eq52}
	\begin{split}
		\text{(C)}\quad L_{\rho_{j}}\left(x^{j+1},u^{j+1},z^{j+1}\right)-\inf L_{\rho_j}\le \frac{\epsilon_j^2}{2\rho_j},\ \ \epsilon_j\ge 0,\ \ \sum_{j=0}^{\infty}\epsilon_j<\infty.
	\end{split}
\end{equation*}
The global convergence result of ALM has been extensively studied in \cite{Rockafellar,RT1,asy,LXD}. ALM is applied to solve a general form
of \eqref{eq11}, i.e., problem {\rm (D)} in \cite{LXD}. The convergence result is obtained therein. Therefore, as a special case of problem {\rm (D)} in \cite{LXD}, we can get the following global convergence of Algorithm \ref{alg3}.
\begin{theorem}
	Let $\left\{\left(x^j, u^j, z^j, \alpha^j\right)\right\}$ be the infinite sequence generated by Algorithm $\ref{alg3}$ with stopping criterion {\rm (C)}. The dual problem of \eqref{eq11} is as follows
	\begin{equation}\label{eq48}
			\max\limits_{\alpha\in\Re^{n+p}}\ \Big\{-\phi(\alpha):=-\frac{\sigma_{k}}{2}\|A^\top\alpha_1-\alpha_2\|^2+\langle y,\alpha_1\rangle-\langle x^k,A^\top \alpha_1-\alpha_2\rangle-h^*(\alpha_1)-f^*(\alpha_2)\Big\},
	\end{equation}
	where $\alpha=(\alpha_1,\alpha_2)^\top,\ f(z)=\lambda\|z\|_1$, $h^*(\cdot)$ and $f^*(\cdot)$ are the conjugate functions of $h(\cdot)$ and $f(\cdot)$. Then, the sequence $\left\{\alpha^j\right\}$ is bounded and converges
	to an optimal solution of \eqref{eq48}. In addition, $\left\{\left(x^j, u^j, z^j\right)\right\}$ is also bounded and converges
	to the unique optimal solution $\left(x^*, u^*, z^*\right)$ of \eqref{eq11}.
\end{theorem}
To state the local convergence rate, we need the following stopping criteria which are popular used such as in \cite{Con} and \cite{LXD}:
\begin{equation*}
	\begin{split}
		\text{(D1)}&\quad L_{\rho_{j}}\left(x^{j+1},u^{j+1},z^{j+1}\right)-\inf L_{\rho_j}\le\left(\kappa_j^2/2\rho_j\right)\|\alpha^{j+1}-\alpha^j\|^2,\ \ \sum\limits_{j=0}^\infty \kappa_j<+\infty,\\
		\text{(D2)}&\quad\|\nabla \varphi_j\left(x^{j+1}\right)\|\le\left(\kappa'_j/\rho_j\right)\|\alpha^{j+1}-\alpha^j\|,\ 0\le \kappa'_j\to 0.
	\end{split}
\end{equation*}
Define the Lagrangian function of \eqref{eq46} as 
\begin{equation*}
	L_k(x,u,z;\alpha)=h\left(u\right)+\lambda\left\|z\right\|_1-\langle u-b+Ax,\alpha_1\rangle-\langle z-x,\alpha_2\rangle+\frac{1}{2\sigma_k}\left\|x-x^k\right\|^2.
\end{equation*}
Correspondingly, define the maximal monotone operator $\mathcal{T}_\phi$ and $\mathcal{T}_{L_k}$ by
\begin{equation}\label{eq51}
	\mathcal{T}_\phi=\partial \phi \left(\alpha\right),\ \mathcal{T}_{L_k}\left(x,u,z,\alpha\right)=\left\{\left(x',u',z',-\alpha'\right)\in\partial {L_k}\left(x,u,z,\alpha\right)\right\}.
\end{equation}
By \cite[Proposition 2.2.4]{SJ}, we know that the corresponding operator $\mathcal{T}_\phi$ defined in \eqref{eq51} is a polyhedral multivalued function. Therefore, by \cite[Proposition 3.8]{YYQ}, the error bound condition holds at the origin with modulus $a_\phi$. Similarly, $\mathcal{T}_{L_k}$ defined in \eqref{eq51} also satisfies the error bound condition at the origin with modulus $a_{L_k}$. We can get the following theorem.
\begin{theorem}
	Let $\left\{\left(x^j, u^j, z^j\right)\right\}$ be the infinite sequence generated by ALM with stopping
	criteria {\rm (D1)} and {\rm (D2)}. The following results hold.
	\begin{itemize}
		\item [$(i)$] The sequence $\{\alpha^j\}$ converges to $\alpha^*$, one of the optimal solutions of \eqref{eq48}, and for all $j$	sufficiently large, there is
		\begin{equation*}
			{\rm dist}\left(\alpha^{j+1},\Omega\right)\le\theta_j{\rm dist}\left(\alpha^j,\Omega\right),
		\end{equation*}
		where $\Omega$ is the set of the optimal solutions of \eqref{eq48} and $\theta_j=a_\phi\left(a_\phi^2+\rho_j^2\right)^{-1/2}+2\rho_j/\left(1-\rho_j\right)\to\theta_\infty=a_\phi\left(a_\phi^2+\rho_\infty^2\right)^{-1/2}<1$ as $j\to +\infty$. 
		\item[$(ii)$] If the stopping criterion {\rm (D2)} is also used, then for all $j$ sufficiently large, there is
		\begin{equation*}
			\left\|\left(x^{j+1},u^{j+1},z^{j+1}\right)-\left(x^*,u^*,z^*\right)\right\|\le \theta'_j\|\alpha^{j+1}-\alpha^j\|,
		\end{equation*}
		where $\theta'_j=a_{L_k}\left(1+\kappa'_j\right)/\rho_j\to a_{L_k}/\rho_\infty$ as $j\to+\infty$.
	\end{itemize}
\end{theorem}
\begin{proof}
	The proof is similar to that in \cite[Theorem 3.3]{LXD}. By definition 2.2.1 in \cite{SJ}, \eqref{eq48} is a convex piecewise linear quadratic programming problem. According to the optimal solution set is not empty, $\mathcal{T}_{L_k}^{-1}(0)\neq\emptyset$, so $\mathcal{T}_{L_k}$ is metrically subregular at $\left(x^*,u^*,z^*,\alpha^*\right)$ for the origin. Following Theorem 3.3 in \cite{LXD}, (i), (ii) holds. The proof is finished.
\end{proof}

\section{The SSN Method for \eqref{eq22}}\label{section4.2}
To solve \eqref{eq22}, since $\varphi_j$ is strongly convex and continuously differentiable, the minimization problem \eqref{eq22} has a unique solution $x^*$ which can be obtained via solving the following system of equation: 
\begin{equation}\label{eq53}
	\begin{split}
		0=\nabla \varphi_j(x)=\nabla G_1(x)+\nabla G_2(x)+\frac{1}{\sigma_{k}}(x-x^k),\\
	\end{split}
\end{equation}
where
\begin{equation*}\label{eq43}
	\begin{split}
		&\nabla G_1(x)=-\rho_j A^\top\left(f_1(x)-{\rm Prox}_{\frac{1}{\rho_j} h(\cdot)}\left(f_1(x)\right)\right),\\
		&\nabla G_2(x)=\rho_j\left( f_2(x)-{\rm Prox}_{\frac{\lambda}{\rho_j} \left\|\cdot\right\|_1}\left( f_2(x)\right)\right).
	\end{split}
\end{equation*} 

We know that ${\rm Prox}_{\frac{1}{\rho_j} h(\cdot)}(\cdot)$ and ${\rm Prox}_{\frac{\lambda}{\rho_j}\|\cdot\|_1}(\cdot)$ are strongly semismooth everywhere as special cases in \cite[Theorem 2]{Li}. Therefore, we apply the SSN method to solve the nonsmooth equation \eqref{eq53}. 

A key point is to characterize the generalized Jacobian of $\nabla \varphi_j(x)$. To that end, define the multivalued mapping $\hat{\partial}^2 \varphi_j(\cdot)\ :\ \Re^p\rightrightarrows \Re^p\times\Re^p$ as follows: at $x\in\Re^p$,
\begin{align*}\label{eq25}
	\hat{\partial}^2 \varphi_j(x):=\Big\{{\rho_j}\left(A^\top\left(I- V_1\right)A+I-V_2\right)+\frac{1}{\sigma_{k}}I\ \Big|\ V_1\in\mathcal{M}(f_1(x)),\nonumber\ V_2\in\partial{\rm Prox}_{\frac{\lambda}{\rho_j} \left\|\cdot\right\|_1}\left(f_2(x)\right)\Big\},
\end{align*}
where $\mathcal{M}(\cdot)$ is a multivalued mapping defined in the following Proposition \ref{pro4}, which can be viewed as the generalized Jacobian of ${\rm Prox}_{\frac{1}{\rho_j} h(\cdot)}\left(f_1(x)\right)$ \cite[Proposition 2.9]{H}, $\partial{\rm Prox}_{\frac{\lambda}{\rho_j} \left\|\cdot\right\|_1}\left(\cdot\right)$ is the Clarke's subdifferential of ${\rm Prox}_{\frac{\lambda}{\rho_j} \left\|\cdot\right\|_1}\left(\cdot\right)$.

According to the chain rules in \cite{partical}, there is $\hat{\partial}^2 \varphi_j(x)d=\partial^2 \varphi_j(x)d,\ \forall\ d\in\Re^p,$ where $\partial^2 \varphi_j(x)$ is the Clarke's generalized Jacobian of $\varphi_j(x)$.

Now we are ready to present the globalized version of SSN method for \eqref{eq22} which is used in \cite{Sun2006,Qi,Yin,LQN2010}. 
\begin{algorithm}[!htbp]
	\caption{Globalized version of the SSN method for \eqref{eq22}}\label{alg4}
	\hspace*{0.02in} {\bf Input:} 
	$\bar{\mu}\in\left(0,\frac{1}{2}\right)$, $\delta\in\left(0,1\right)$, $\epsilon_{SSN}>0$, $\bar{\eta}_0>0,\ $and $\bar{\eta}_1>0$.\\
	\hspace*{0.02in} {\bf Initialization:}
	choose $x^{0}\in\Re^p,\ $$i=0$.\\
	\hspace*{0.02in} {\bf S1:} Select an element $\mathcal{H}_i\in\partial^2 \hat{\varphi}_j\left(x^{i}\right)$. Apply the conjugate gradient (CG)
	method to find an approximate solution $d^i\in\Re^p$ of the following system
	\begin{equation*}\label{eq26}
		\mathcal{H}_i\left(d^i\right)=-\nabla \varphi_j\left(x^{i}\right),
	\end{equation*}
	such that
	\begin{equation*}
		\left\|\mathcal{H}_i\left(d^i\right)+\nabla \varphi_j\left(x^{i}\right)\right\|\le\eta_i\|\nabla\varphi_j\left(x^i\right)\|,
	\end{equation*} 
	where $\eta_i:= \min\left(\bar{\eta}_0,\bar{\eta}_1\left\|\nabla \varphi_j\left(x^{i}\right)\right\|\right)$.\\
	\hspace*{0.02in} {\bf S2:} Let $\gamma_i=\delta^{m_i}$, where $m_i$ is the smallest nonnegative integer $m$ such that the following holds
	\begin{equation*}
		\varphi_j\left(x^{i}+\delta^m d^i\right)\le \varphi_j\left(x^{i}\right)+\bar{\mu}\delta^m\left(\nabla \varphi_j\left(x^{i}\right)\right)^\top d^i.
	\end{equation*}\\
	\hspace*{0.02in} {\bf S3:} Let $x^{i+1}=x^{i}+\gamma_i d^i.$\\
	\hspace*{0.02in} {\bf S4:} If $\|\nabla \varphi_j\left(x^{i+1}\right)\|\le \epsilon_{SSN}$, stop; otherwise, $i:=i+1$, go to {\bf S1}.
\end{algorithm}
To demonstrate the convergence of Algorithm \ref{alg4}, we need the characterization of $\partial {\rm Prox}_{\tau \|\cdot\|_1}(\cdot)$ and $\mathcal{M}(\cdot)$. It is easy to show that
\begin{equation*}
		\partial {\rm Prox}_{\tau \left\|\cdot\right\|_1}\left(x\right)=\big\{{\rm Diag}\left(v\right)\in\Re^{p\times p}\ \big|\ v_i=1,\ \text{if}\ |x_i|>\tau;\ v_i=0,\ \text{if}\ |x_i|<\tau;\ v_i\in[0,1],\ \text{if}\ |x_i|=\tau\big\}.
\end{equation*}
The following results about ${\rm Prox}_{\tau h(\cdot)}(\cdot)$ come from \cite{H}.
\begin{proposition}{\rm \cite[Proposition 2.3]{H}}\label{pro6}
	Given $y\in\Re^n$, there exists a permutation matrix $P_y\in\Re^{n\times n}$ such that $y^\downarrow=P_yy$, and $y^\downarrow_1\ge y^\downarrow_2\ge \dots \ge y^\downarrow_n$. Then it holds
	\begin{equation*}
		{\rm Prox}_{\tau h}\left(y\right)=P_y^{\top}\Pi_\mathcal{D}\left(P_yy-\frac{2\tau}{n(n-1)} w\right),
	\end{equation*}
	where the vector $w\in\Re^n$ is defined by $w_k=n-2k+1,\ k=1,\dots,n$, $\mathcal{D}=\left\{x\in\Re^n\ |\ Cx\ge 0\right\}$, and $C$ is a matrix such that $Cx=[x_1-x_2;\dots;x_{n-1}-x_n]\in\Re^{n-1}$. $\Pi_\mathcal{D}(z)$ is the optimal solution of 
	\begin{equation}\label{eq9}
		\min\limits_{x\in\Re^n}\left\{\left. \frac{1}{2}\|x-z\|^2\ \right|\ x\in \mathcal{D}\right\}.
	\end{equation}
\end{proposition}

Define $\tilde{y}:=P_yy-\frac{2\tau}{n(n-1)}w$, denote the active index set by
\begin{equation}\label{eq56}
	\mathcal{I}_\mathcal{D}\left(\tilde{y}\right):=\left\{i\ |\ C_i\Pi_\mathcal{D}\left(\tilde{y}\right)=0,\ i=1,2,\dots ,n-1\right\},
\end{equation}
where $C_i$ is the $i$-th row of $C$. Define a collection of index subsets of $\{1,\dots,n-1\}$ as
follows: $\mathcal{Q}_{\mathcal{D}}(\tilde{y}):=\big\{Q\ |\ \exists\ \lambda\in\mathcal{W}_\mathcal{D}(\tilde{y}),\ {\rm s.t.}\ {\rm supp}(\lambda)\subset Q \subset \mathcal{I}_\mathcal{D}(\tilde{y}),\ C_Q\ \text{is of}$
$\text{full row rank}\big\}$, where supp($\lambda$) denotes the support of $\lambda$, $C_Q$ is the matrix consisting of the rows of $C$ indexed by $Q$, and $\mathcal{W}_\mathcal{D}(\tilde{y}):=\big\{\lambda\in\Re^{n-1}\ |\ (\tilde{y},\lambda)\ \text{satisfies}$
$\text{the KKT conditions of \eqref{eq9}}\big\}$. Let $\Sigma_Q\in\Re^{\left(n-1\right)\times\left(n-1\right)}$ is a diagonal matrix, where
\begin{equation}\label{eq8}
	\Sigma_{ii}=\begin{cases}
		1,& i\in Q,\\
		0,&\text{otherwise}.
	\end{cases}
\end{equation}
\begin{proposition}{\rm \cite[Proposition 2.6]{H}}\label{pro4}
	Given $y\in\Re^n$, we define
	\begin{equation*}
		\mathcal{M}\left(y\right):=\left\{U\in\Re^{n\times n}\ |\ U=P_y^\top\left(I_n-C_Q^\top(C_Q C_Q^\top)^{-1}C_Q\right)P_y,\ Q\in\mathcal{Q}_\mathcal{D}(\tilde{y})\right\},
	\end{equation*}
	which can be viewed as the generalized Jacobian of ${\rm Prox}_{\tau h(\cdot)}\left(y\right)$. We can get
	\begin{equation*}
		P_y^{\top}\left(I_n-C^{\top}\left(\Sigma_Q CC^{\top}\Sigma_Q\right)^\dag C\right)P_y=P_y^\top\left(I_n-C_Q^\top(C_Q C_Q^\top)^{-1}C_Q\right)P_y,
	\end{equation*}
	where $\Sigma_Q$ is defined in \eqref{eq8}, $C$ and $P_y$ are defined in Proposition $\ref{pro6}$, $\left(\cdot\right)^\dag$ denotes the pseudoinverse. 
\end{proposition}
To guarantee the quadratic convergence rate of the SSN method, one need to verify that each element in $\hat{\partial}^2\varphi_j(x)$ is positive definite. Below we confirm the result in Lemma \ref{lem1} with the proof postponed in Appendix.
\begin{lemma}\label{lem1}
	Any $\mathcal{H}\in\hat{\partial}^2 \varphi_j\left(x\right)$ is symmetric and positive definite.
\end{lemma}
Based on Lemma \ref{lem1}, the convergence rate of the SSN method is given below.
\begin{theorem}{\rm \cite[Theorem 3.2]{1993}}
	Let $x^*$ be the solution of \eqref{eq53}. Then the sequence $\left\{x^i\right\}$ generated by Algorithm $\ref{alg4}$ is quadratically convergent to $x^*$. 
\end{theorem}
\begin{proof}
	According to Lemma \ref{lem1}, each element in $\hat{\partial}^2\varphi_j(x)$ is positive definite, and since ${\rm Prox}_{\frac{1}{\rho_j} h(\cdot)}(\cdot)$ and ${\rm Prox}_{\frac{\lambda}{\rho_j}\|\cdot\|_1}(\cdot)$ are strongly semismooth everywhere, $\nabla \varphi_j(x)$ is also strongly semismooth \cite{Li}. By \cite[Theorem 3.2]{LXD}, the sequence $\left\{x^i\right\}$ generated by Algorithm \ref{alg4} converges to $x^*$ quadratically.
\end{proof}

According the strong convexity of $\varphi_j$, we have
\begin{equation*}
	L_{\rho_{j}}\left(x^{j+1},u^{j+1},z^{j+1}\right)-\inf L_{\rho_j}=\varphi_j\left(x^{j+1}\right)-\inf\varphi_j\le\frac{1}{2}\left\|\nabla \varphi_j\left(x^{j+1}\right)\right\|^2.
\end{equation*}
Therefore, the stopping criteria {\rm (C)} and {\rm (D1)} can be simplified to
\begin{equation*}\label{eq34}
	\begin{split}
		\text{(C')}&\quad\left\|\nabla \varphi_j\left(x^{j+1}\right)\right\|\le \frac{\epsilon_j}{\sqrt{\rho_j}},\ \ \epsilon_j\ge 0,\ \ \sum_{j=0}^{\infty}\epsilon_j<\infty,\\
		\text{(D1')}&\quad\left\|\nabla \varphi_j\left(x^{j+1}\right)\right\|\le\left(\kappa_j/\sqrt{\rho_j}\right)\|\alpha^{j+1}-\alpha^j\|,\ \ \sum\limits_{j=0}^\infty \kappa_j<+\infty.
	\end{split}
\end{equation*}
That is, the stopping criteria {\rm (C)}, {\rm (D1)}, and {\rm (D2)} will be satisfied as long as $\nabla \varphi_j\left(x^{j+1}\right)$ is sufficiently small.

\section{Numerical Results}\label{section5}
In this section, we conduct numerical experiments to evaluate the performance of AS-PPA for solving \eqref{eq19}. It contains five parts. In the first part, we provide the test problems and the implementation of the algorithm. In the second part, we investigate the performance of Algorithm \ref{alg1} based on different parameters. Furthermore, we verified the effect of the enhanced technique proposed in subsection \ref{sec2.3} applied to the hdr lasso problem. In the third and fourth parts, we compare Algorithm \ref{alg1} with other methods for small-scale and large-scale hdr lasso problems. In the last part, we compare Algorithm \ref{alg1} with other methods for Square-Root Lasso models with different scales. All of our results are obtained by running Matlab R2020a on a windows workstation (Intel Core i5-7200U @ 2.50GHz, 12G RAM).
\subsection{Test Problems}\label{section5.1} 
We test the following examples. 
\begin{itemize}
	\item[E1] \cite{Wang} We simulate	random data from the regression model $Y_i=X^{\top}_i x^*+\varepsilon_i,\ i= 1, \dots , n$, where $X_i$ is generated from a $p$-dimensional multivariate normal distribution $N_p\left(0,\Sigma\right)$ and is independent of $\varepsilon_i$, $x^*=\big(\sqrt{3},\sqrt{3},\sqrt{3},$
	$0,\dots,0\big)^\top\in\Re^p$. The correlation matrix $\Sigma$ has a compound symmetry structure: $\Sigma_{(i,j)}= 0.5$ for $i\neq j$; and $\Sigma_{(i,j)}= 1$ for $i=j$. We consider six different distributions for $\varepsilon_i$: (i) normal distribution with mean 0 and variance 0.25 (denoted by $N\left(0,0.25\right)$); 
	(ii) normal distribution with mean 0 and variance 1 (denoted by $N\left(0,1\right)$); (iii) normal distribution with mean 0 and variance 2 (denoted by $N\left(0,2\right)$); (iv) mixture normal distribution $\varepsilon\thicksim 0.95N\left(0,1\right) + 0.05N\left(0,100\right)$ (denoted by $MN$); (v)$\varepsilon\thicksim\sqrt{2}t\left(4\right)$, where $t\left(4\right)$ denotes the $t$ distribution with 4 degree of freedom; and (vi) $\varepsilon\thicksim$ Cauchy(0,1), where Cauchy(0,1) denotes the standard Cauchy distribution.
	\item[E2] \cite{Wang} Similar as in E1, we take $x^*=\big(2,2,2,2,1.75,1.75,1.75,1.5,1.5,1.5,1.25,1.25,1.25,1,1,1,0.75,0.75,$
	$0.75,0.5,0.5,0.5,0.25,0.25,0.25,0_{p-25}\big)^{\top}$, where $0_{p-25}$ is a $(p-25)$-dimensional vector of zeros. Comparing with E1, this is a considerably more challenging scenario with 25 active variables and a number of weak signals, such as 0.5, 0.25, etc. 
	\item[E3] \cite{Wang} $x^*$ is the same as in E1. We use the following three different choices of $\Sigma$: (i) the compound symmetry correlated correlation matrix with correlation coefficient $0.8$, denoted as $\Sigma_1$; (ii) the compound symmetry correlation matrix with correlation coefficient $0.2$, denoted as $\Sigma_2$; and (iii)
	the compound symmetry correlation matrix with correlation coefficient $0.5$, denoted as $\Sigma_3$. For each choice of $\Sigma$, we consider three different error distributions: $N\left(0,1\right)$, the mixture normal distribution in E1, and the Cauchy distribution in E1.  
	\item[E4] \cite{SafeRank} The simulation data is generated from $Y=X^\top x^*+\varepsilon$ with $n=10$ and $p\in\left\{5000,10000,150000\right\}$. Here, elements of $X$ are randomly generated from the exponential
	distribution $e(3)$ and the error $\varepsilon\thicksim N(0,0.01)$. The number of nonzero coefficients of $x^*$ is $20\%*p$, and the nonzero part of $x^*$ is generated from the standard normal distribution. 
		\item[E5] \cite{square} $X_i$ is generated from a $p$-dimensional multivariate normal distribution $N_p(0,\Sigma)$. The correlation matrix is the Toeplitz correlation matrix, i.e., $\Sigma_{jk}=(1/2)^{|j-k|}$, $\varepsilon_i\sim N(0,1),\ x^*=(1,1,1,1,1,0_{p-5})$. 
		\item[E6] \cite{square} Similar as in E5, we take $\varepsilon_i\sim t_4/\sqrt{2}$.
\end{itemize} 

We choose $\lambda$ following the way in \cite[Theorem 1]{Wang}\footnote{	$\lambda=cG^{-1}_{\left\|S_n\right\|_\infty}\left(1-\alpha_0\right)$, where $\alpha_0=0.10$, $c=1.1$ and $G^{-1}_{\left\|S_n\right\|_\infty}\left(1-\alpha_0\right)$ denotes the $\left(1-\alpha_0\right)$-quantile of the distribution of $\left\|S_n\right\|_\infty$, $S_n:=-2[n\left(n-1\right)]^{-1}X^\top\xi$, $X$ is the design matrix, the random vector $\xi:=2r-(n+1)\in\Re^n$, with $r$ following the uniform distribution on the permutations of the integers $\{1,2,\dots,n\}$.}. The parameters are taken as $\epsilon=1.0\times 10^{-6},\ \epsilon_{SSN}=1.0\times 10^{-6}$, $\{\delta_k'\}=\{\delta_k\}=\{\gamma_k\}=\{0.8,0.8^2,\dots\}$. 

Taking account of numerical rounding errors, we define the number of nonzero elements of a vector $x\in\Re^p$ and the index set of nonzero components of $x$ as follows (see \cite[Page 17]{WCJ})
\begin{equation*}
	\hat{k}:=\min\left\{k\ \left|\ \sum\limits_{i=1}^k|\tilde{x}_i|\ge 0.9999\left\|x\right\|_1\right.\right\},\ \ I(x)=\left\{ i\ \left|\ x_i\ge \tilde{x}_{\hat{k}} \right.,\ i=1,\dots,p \right\},
\end{equation*}
where $\tilde{x}$ is sorted, i.e., $|\tilde{x}_1|\ge|\tilde{x}_2|\ge\dots\ge|\tilde{x}_p|$. Let $\bar{I}(x)=\left\{1,\dots,p\right\}\backslash I(x)$.

We report the following information: the relative KKT residual ($\eta_{KKT}$) calculated by  
\begin{equation*}\label{eq50}
	\begin{split}
		\max\left\{\frac{\|u^l-{\rm Prox}_{h(\cdot)}\left(u^l+\alpha^l\right)\|}{1+\|u^l\|},\frac{\|x^l-{\rm Prox}_{\lambda \|\cdot\|_1}\left(x^l+A^\top\alpha^l\right)\|}{1+\|x^l\|},\frac{\|u^l-b+Ax^l\|}{1+\|u^l\|}\right\};
	\end{split}
\end{equation*} the function value of \eqref{eq19} (val); the $L_1$ estimation error $||x^l-x^*||_1$; the $L_2$ estimation error $||x^l-x^*||_2$; the model error ME$=\left(x^l-x^*\right)^{\top}\Sigma_X\left(x^l-x^*\right)$ where $\Sigma_X$ is the population covariance matrix of $X$; the number of false positive variables FP$=|I(x^l)\bigcap \bar{I}(x^*)|$; the number of false negative variables FN$=|\bar{I}(x^l)\bigcap I(x^*)|$; the correct component ratio CR$=\frac{|I^l\bigcap I(\hat{x})|}{|I^l|}$, where $\hat{x}$ is the solution of the original problem by taking $I^0=\left\{1,2,\dots,p\right\}$; the running time in second ($\hat{t}$).
Moreover, we additionally report the number of iterations in Algorithm \ref{alg1} (It-AS), the total number of iterations in Algorithm \ref{alg2} (It-PPA), the total number of iterations in Algorithm \ref{alg3} (It-ALM), the total number of iterations in Algorithm \ref{alg4} (It-SSN) and the total time of SSN ($t_{SSN}$).
\subsection{Role of Parameters}\label{section5.2}
{\bf The number of components added to $I^l$.} In Algorithm \ref{alg1}, the number of components we choose to add to $I^l$ in each iteration may affect the performance of the algorithm. We first fix the maximum number of added components (denoted as $M$) as 25, that is, if the sieved components are less than or equal to 25, then all are added. We use $m$ to represent the number of components added to $I^l$ each time. We conduct experiments on E2 with $n=200,\ p=1000, \varepsilon$ of (i) for $m=\left\{\frac{p}{200},\frac{p}{100},\frac{3p}{200},\frac{p}{50}\right\}$, that is, $m\in\left\{5,10,15,20\right\}$. The results are shown in Figure \ref{fig-1} and Table \ref{tab-num-7}. We can see from Figure \ref{fig-1} that with the iteration of AS, new components are gradually added to $I^l$, and the relative KKT residuals are getting smaller and smaller. Furthermore, $I^l$ gradually covers $I(\hat{x})$. For the different numbers of components added each time, the relative KKT residuals decline differently. Obviously, the larger $m$ is, the faster the decline is. For the comparison of CR, we can see that the smaller $m$ is, the higher the CR will be. From $t_{SSN}$, we can see that $t_{SSN}$ grows in the slowest way if $m$ takes the smallest value. However, because the number of AS iterations increases, the overall time of SSN may be longer. The results in Table \ref{tab-num-7} show the overall performance for the four values of $m$. It can be seen that four choices of $m$ lead to similar quality of output whereas $m=10$ leads to the smallest time. Therefore, we choose $m=\frac{p}{100}$ in our following test. 
\begin{figure}[!htb]
	\vspace{5pt}
	\centering
	\begin{minipage}[t]{.45\linewidth}
		\includegraphics[width=1\textwidth]{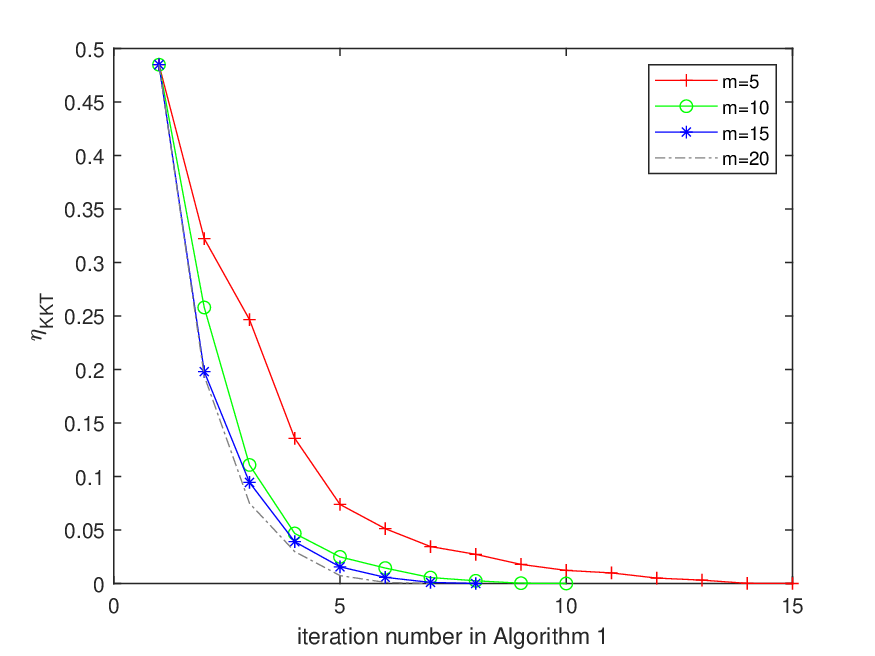}
	\end{minipage}
	\begin{minipage}[t]{.45\linewidth}
		\includegraphics[width=1\textwidth]{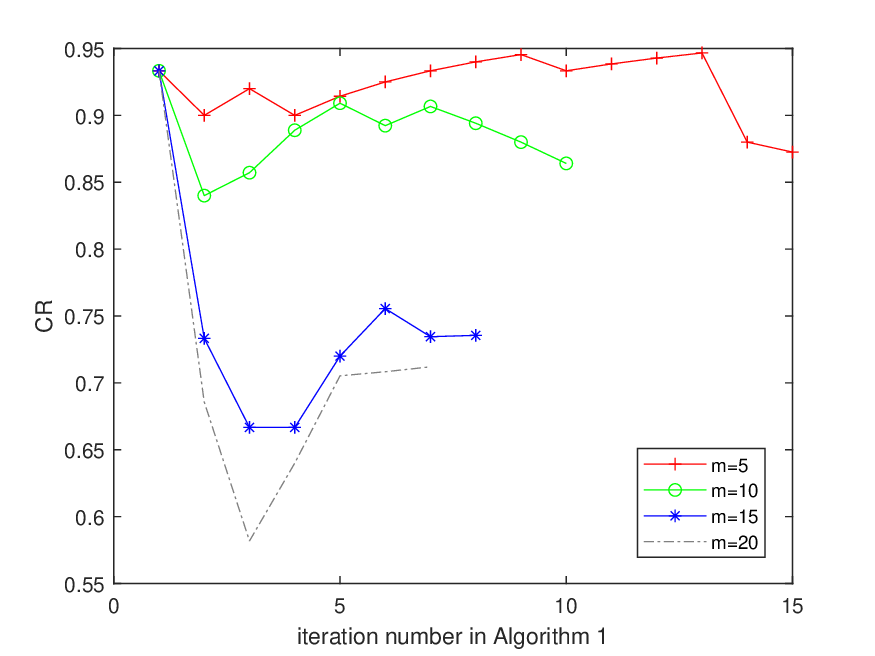}
	\end{minipage}
	\vfill
	\begin{minipage}[t]{.45\linewidth}
		\includegraphics[width=1\textwidth]{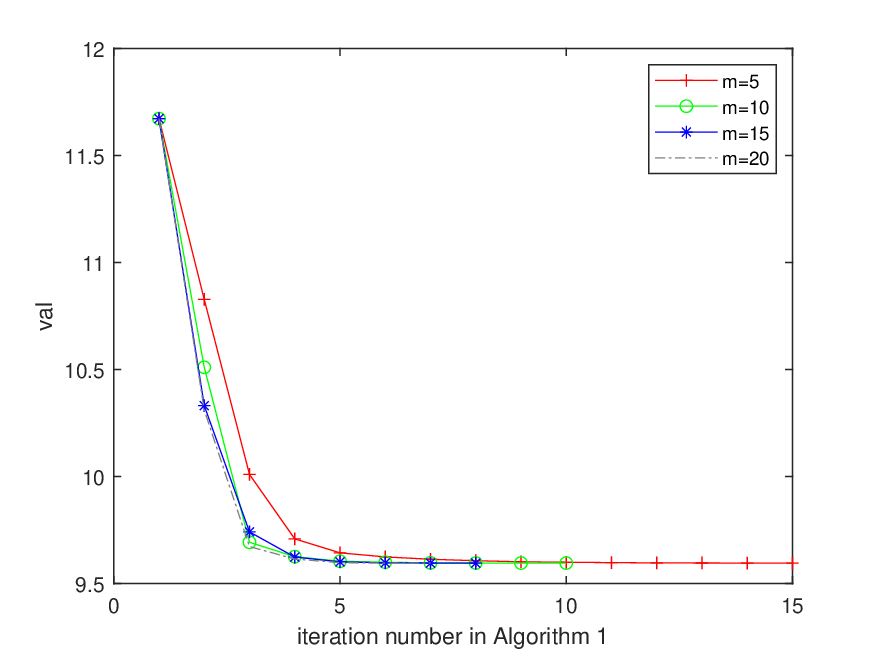}
	\end{minipage}
	\begin{minipage}[t]{.45\linewidth}
		\includegraphics[width=1\textwidth]{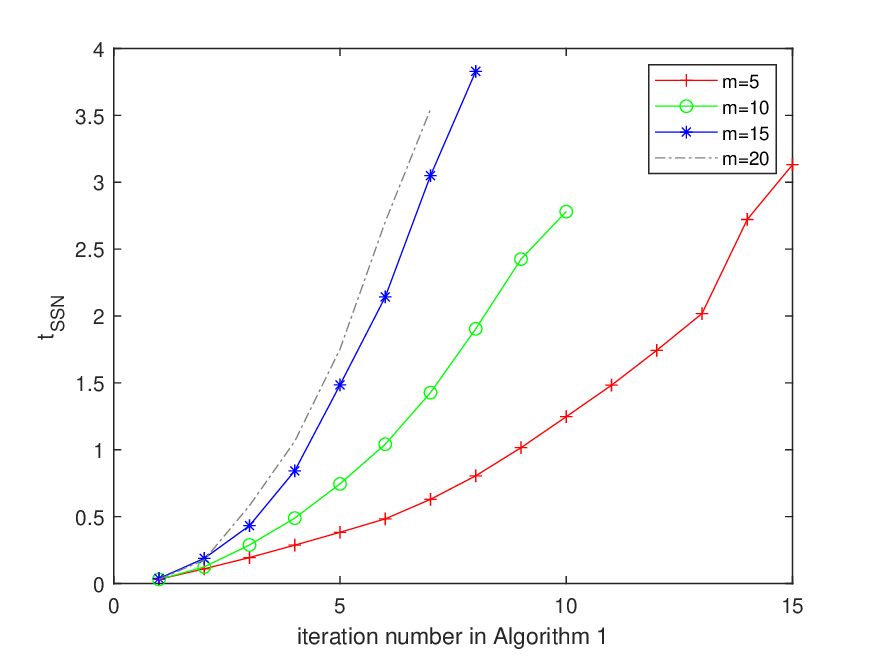}
	\end{minipage}
	\caption{$\eta_{KKT}$, CR, val and $t_{SSN}$ for different $m$'s in AS-PPA on E2.}\label{fig-1}
\end{figure}

\begin{table}[!htb]
	\centering
	\caption{Results for different $m$'s in AS-PPA on E2 ($n=200,\ p=1000$).}
			\begin{tabular}{|c|c| c|c| c| c| c| c| c|c|c|c|c|c|}
				\hline
				\makebox[0.005\textwidth][c]{$m$} &
				\makebox[0.03\textwidth][c]{$\eta_{KKT}$}& 
				\makebox[0.02\textwidth][c]{$val$}& 
				\makebox[0.02\textwidth][c]{$L_1$} & 
				\makebox[0.02\textwidth][c]{$L_2$}& 
				\makebox[0.02\textwidth][c]{ME} &
				\makebox[0.01\textwidth][c]{FP}& 
				\makebox[0.01\textwidth][c]{FN}& 
				\makebox[0.02\textwidth][c]{$\hat{t}(s)$}& 
				\makebox[0.03\textwidth][c]{It-AS}& 
				\makebox[0.05\textwidth][c]{It-PPA}& 
				\makebox[0.05\textwidth][c]{It-ALM}& 
				\makebox[0.05\textwidth][c]{It-SSN}& 
				\makebox[0.03\textwidth][c]{$t_{SSN}$} \\
				\hline
				5&4.9-7&9.60&4.42&0.68&0.15&64&0&3.24&15&239&829&4343&3.13
				\\
				\hline
				10&9.5-7&9.60&4.42&0.68&0.15&64&0&\textbf{2.87}&10&211&665&3438&\textbf{2.78}
				\\
				\hline 
				15&3.8-7&9.60&4.42&0.68&0.15&64&0&3.93&8&191&537&3526&3.83
				\\ 
				\hline
				20&6.0-7&9.60&4.42&0.68&0.15&64&0&3.63&7&132&460&3318&3.55\\ 
				
				\hline
			\end{tabular}\label{tab-num-7}
\end{table}

{\bf The role of the maximum number added to $I^{l}$ (denoted as $M$).} We still test on E2. That is, $n=200,\ p=1000$, and $M\in\left\{\frac{p}{40},\frac{p}{20},\frac{3p}{40},\frac{p}{10}\right\}=\left\{25,50,75,100\right\}$. The results are reported in Figure \ref{fig-2} and Table \ref{tab-num-8}. We can see from Figure \ref{fig-2} that in the first five iterations of AS, the results of the four choices are almost the same. This is because the maximum number of additions is not reached. In this case, only 10 components are added each time, so the results of the four choices of $M$ are the same. From iteration $l=6$, different $M$'s lead to different CR's and $t_{SSN}$'s. Specifically, a higher value of $M$ results in lower residuals $\eta_{KKT}$ as well as lower CR's, but it also requires more cputime, as indicated by $t_{SSN}$. On the other hand, different values of $M$ hardly affect the function value. From Table \ref{tab-num-8}, one can see that $M=25$ is a reasonable choice compared with other choices since it leads to less cputime. Therefore, in our following test, we choose $M=\frac{p}{40}$.  
\begin{figure}[!htb]
	\vspace{5pt}
	\centering
	\begin{minipage}[t]{.45\linewidth}
		\includegraphics[width=1\textwidth]{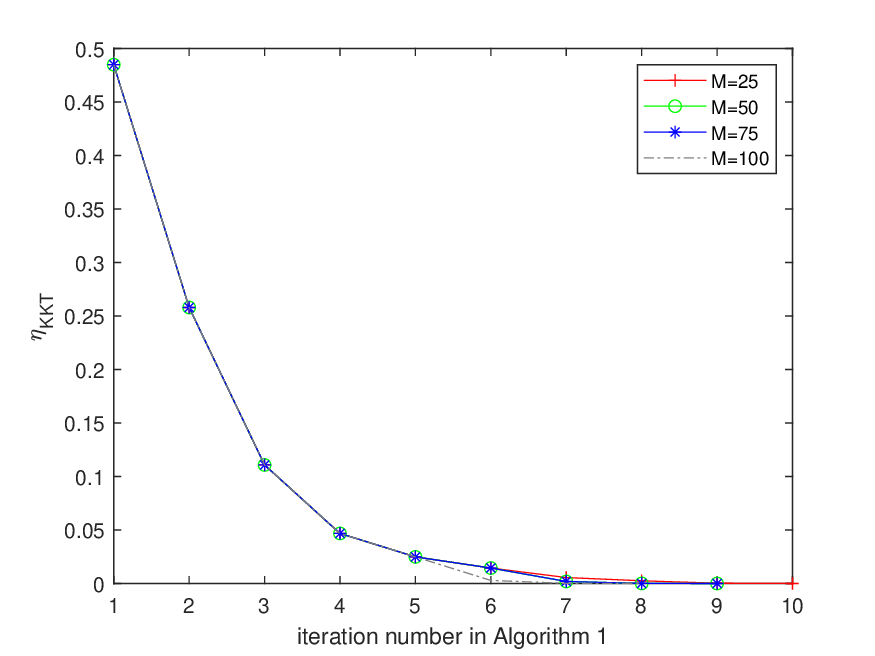}
	\end{minipage}
	\begin{minipage}[t]{.45\linewidth}
		\includegraphics[width=1\textwidth]{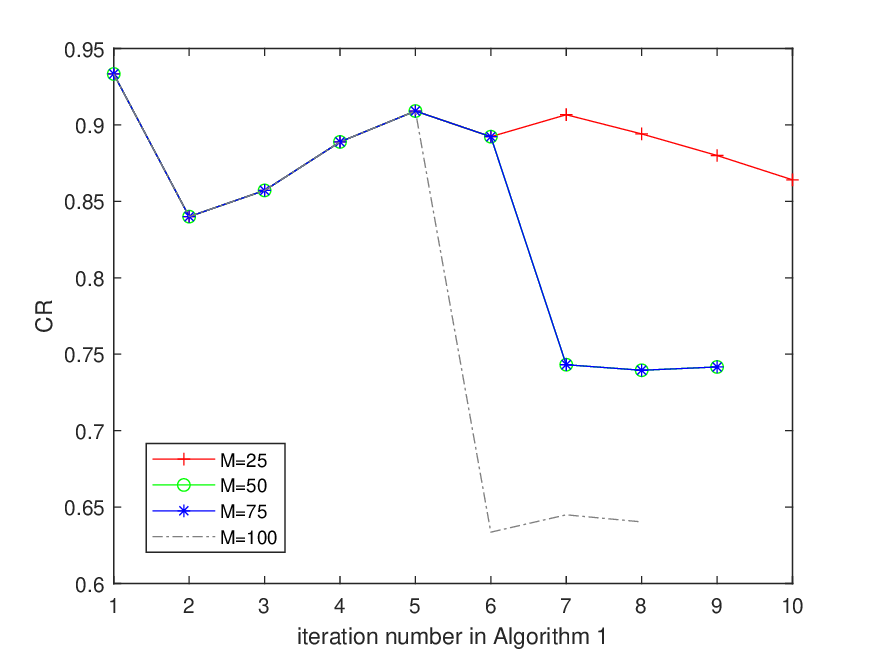}
	\end{minipage}
	\vfill
	\begin{minipage}[t]{.45\linewidth}
		\includegraphics[width=1\textwidth]{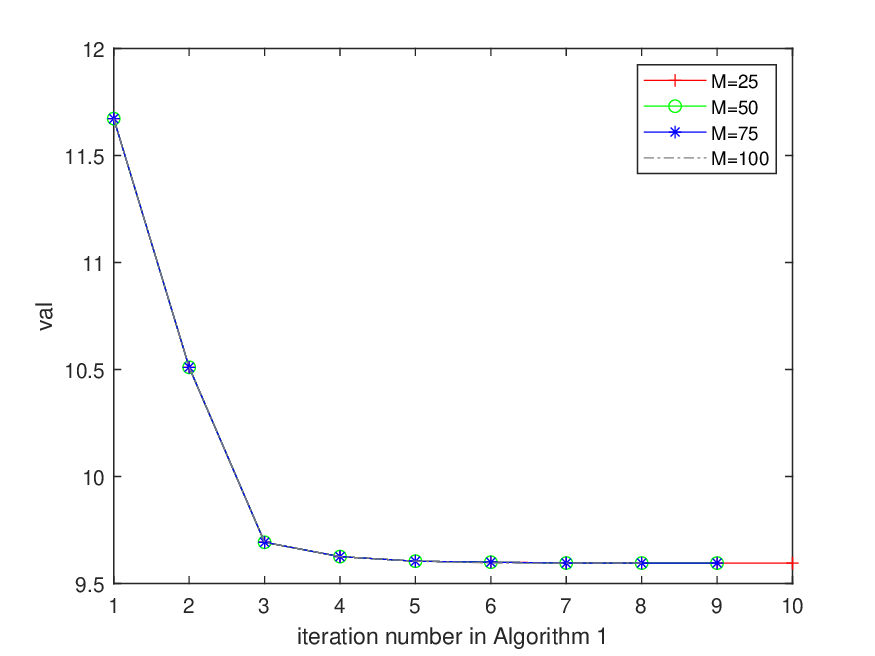}
	\end{minipage}
	\begin{minipage}[t]{.45\linewidth}
		\includegraphics[width=1\textwidth]{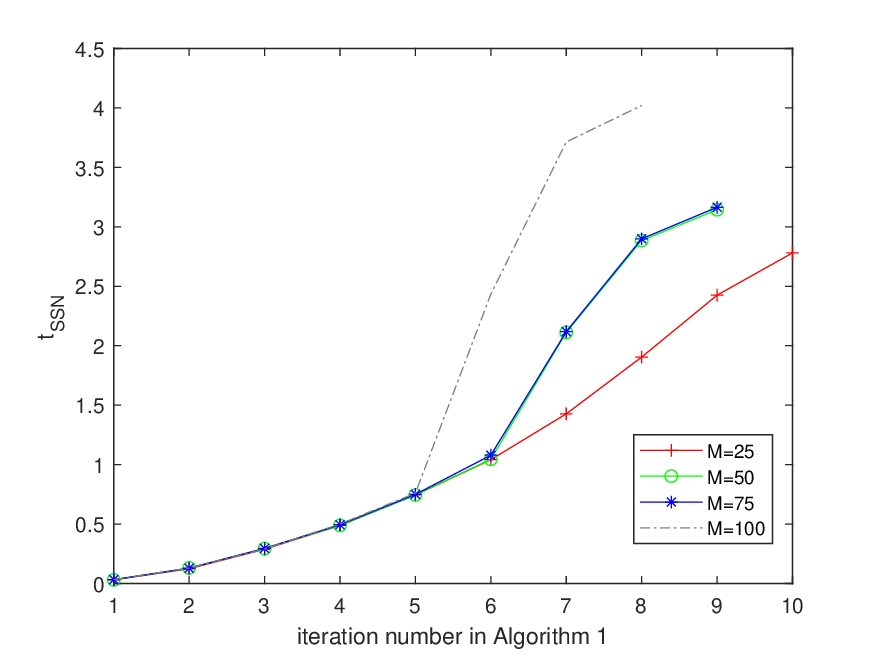}
	\end{minipage}
	\caption{$\eta_{KKT}$, CR, val and $t_{SSN}$ for different $M$'s in AS-PPA on E2.}\label{fig-2}
\end{figure}

\begin{table}[!htb]
	\centering
	\caption{Results for $M=\left\{25,50,75,100\right\}$ in AS-PPA on E2 ($n=200,\ p=1000$).}
			\begin{tabular}{|c|c| c|c| c| c| c| c| c|c|c|c|c|c|}\hline
				\makebox[0.005\textwidth][c]{M} &
				\makebox[0.03\textwidth][c]{$\eta_{KKT}$}& 
				\makebox[0.01\textwidth][c]{$val$}& 
				\makebox[0.01\textwidth][c]{$L_1$} & 
				\makebox[0.01\textwidth][c]{$L_2$}& 
				\makebox[0.01\textwidth][c]{ME} &
				\makebox[0.01\textwidth][c]{FP}& 
				\makebox[0.01\textwidth][c]{FN}& 
				\makebox[0.01\textwidth][c]{$\hat{t}(s)$}& 
				\makebox[0.03\textwidth][c]{It-AS}& 
				\makebox[0.05\textwidth][c]{It-PPA}& 
				\makebox[0.05\textwidth][c]{It-ALM}& 
				\makebox[0.05\textwidth][c]{It-SSN}& 
				\makebox[0.01\textwidth][c]{$t_{SSN}$} \\
				\hline
				25&9.5-7&9.60&4.42&0.68&0.15&64&0&\textbf{2.87}&10&211&665&3438&\textbf{2.78}
				\\
				\hline
				50&6.4-7&9.60&4.42&0.68&0.15&64&0&3.24&9&155&573&3132&3.14
				\\
				\hline 
				75&6.4-7&9.60&4.42&0.68&0.15&64&0&3.26&9&155&573&3132&3.16
				\\ 
				\hline
				100&7.6-7&9.60&4.42&0.68&0.15&63&0&4.10&8&163&515&3028&4.02
				\\ 
				\hline
			\end{tabular}\label{tab-num-8}
\end{table}

{\bf The initial points in the subproblem.} For the initial point $x^{l+1}$ in the subproblem, notice that $x^l$ has been obtained in the last iteration, so we set $x^{l+1}_I=x^l$. For the added index set $J$, we compare four different options $x^{l+1}_J=\left\{0,\frac{1}{n},\frac{1}{|I_{l+1}|},\frac{1}{|I_{J}|}\right\}$. The results for four different initial points are in Table \ref{tab-num-5}. From the results, we can see that there is no big difference between the four initial points in the effect of the solution, for example, $L_1$, $L_2$, etc. have basically no difference. 
Considering the total time of SSN, the method of taking 0 as the initial point will be faster in time. Therefore, we take 0 as the initial points for the following test.
	\begin{table}[!htb]
		\centering
		\caption{Results for four different initial points ($n=200,\ p=1000$).}
				\begin{tabular}{|c|c| c|c| c| c| c| c| c|c|c|c|c|c|}\hline
					\makebox[0.005\textwidth][c]{$x^{l+1}_J$} &
					\makebox[0.03\textwidth][c]{$\eta_{KKT}$}& 
					\makebox[0.01\textwidth][c]{$val$}& 
					\makebox[0.01\textwidth][c]{$L_1$} & 
					\makebox[0.01\textwidth][c]{$L_2$}& 
					\makebox[0.01\textwidth][c]{ME} &
					\makebox[0.01\textwidth][c]{FP}& 
					\makebox[0.01\textwidth][c]{FN}& 
					\makebox[0.01\textwidth][c]{$\hat{t}(s)$}& 
					\makebox[0.03\textwidth][c]{It-AS}& 
					\makebox[0.05\textwidth][c]{It-PPA}& 
					\makebox[0.05\textwidth][c]{It-ALM}& 
					\makebox[0.05\textwidth][c]{It-SSN}& 
					\makebox[0.01\textwidth][c]{$t_{SSN}$} \\
					
					\hline
					$0$&9.1-7&9.60&4.42&0.68&0.15&64&0&\textbf{2.87}&10&211&665&3438&\textbf{2.78}
					\\
					\hline
					$\frac{1}{n}$&9.1-7&9.60&4.42&0.68&0.15&64&0&3.08&10&190&656&3386&2.98
					\\
					\hline 
					$\frac{1}{|I_{l+1}|}$&9.1-7&9.60&4.42&0.68&0.15&64&0&3.42&10&191&657&3442&3.30
					
					\\ 
					\hline
					$\frac{1}{|I_{J}|}$&9.1-7&9.60&4.42&0.68&0.15&64&0&4.00&10&162&572&4496&3.91
					\\
					\hline
				\end{tabular}\label{tab-num-5}
	\end{table}

{\bf Comparison between eaAS-PPA, eAS-PPA and AS-PPA.} 
To evaluate the performance of the enhanced technique, we compare the following three versions, eaAS-PPA, eAS-PPA and AS-PPA, where eaAS-PPA is the enhanced all AS-PPA that we initially utilize this approach to recompute the optimal values of the Lagrange multipliers; eAS-PPA is the enhanced AS-PPA that we choose to employ the enhanced technique only when approaching the optimal solution of the original problem, that is, $h(x^l)+\lambda \|x^l\|_1-h(x^{l-1})-\lambda \|x^{l-1}\|_1< 1.0\times 10^{-4}\epsilon$. We conduct 50 experiments on E2 with $n=200$ and $p=1000$. The results are presented in Table \ref{tab-num-12}. It can be observed that the quality of solutions obtained by the three methods are almost the same. This is evidenced by the fact that, apart from slight differences in $\eta_{KKT}$, other values such as $L_1$ and $L_2$ are equal. Moreover, the number of iterations in Algorithm \ref{alg1} are also identical, indicating that resolving the multiplier only slightly improves the sieving strategy. This might be caused by the fact that the function $h$ lacks sparsity. According to the definition in \ref{app1}, most components $(\partial h)_i$ in $\partial h$ are single-valued. Even for those components where $(\partial h)_i$ is a set, too few equal components in the vector $u$ results in a small size for the set $(\partial h)_i$. Consequently, the difference between the optimal Lagrange multipliers and those obtained from the subproblem is small. Given that the enhanced technique introduces additional computation cost due to the quadratic programming solver, AS-PPA surprisingly exhibits the shortest computation time. Even eAS-PPA, which involves fewer quadratic programming solver iterations, tends to be faster than eaAS-PPA. Therefore, we use AS-PPA in our following experiments.
	\begin{table}[!htb]
	\centering
	\caption{Results on E2 in different AS-PPA ($n=200,\ p=1000$).}
	\begin{tabular}{|c|r|c|c| c| c| c| c| c|c|c|c|c|c|}
		\hline
		\makebox[0.05\textwidth][c]{Error} &
		\mc{1}{c|}{Method} &
		\makebox[0.02\textwidth][c]{$\eta_{KKT}$}& 
		\makebox[0.005\textwidth][c]{$val$}& 
		\makebox[0.01\textwidth][c]{$L_1$} & 
		\makebox[0.01\textwidth][c]{$L_2$}& 
		\makebox[0.01\textwidth][c]{ME} &
		\makebox[0.01\textwidth][c]{FP}& 
		\makebox[0.01\textwidth][c]{FN}& 
		\makebox[0.01\textwidth][c]{$\hat{t}(s)$}& 
		\makebox[0.03\textwidth][c]{It-AS}& 
		\makebox[0.05\textwidth][c]{It-PPA}& 
		\makebox[0.05\textwidth][c]{It-ALM}& 
		\makebox[0.05\textwidth][c]{It-SSN} \\
		\hline
		\multirow{3}*{$N\left(0,0.25\right)$}&AS-PPA&7.8-7&9.5967&6.12&0.88&0.18&70.0&0.2&\textbf{5.1}&12.4&223.2&536.0&5132.7
		\\
		~&eAS-PPA&8.0-7&9.5967&6.12&0.88&0.18&70.0&0.2&5.2&12.4&222.3&543.0&5131.4
		\\
		~&eaAS-PPA&7.6-7&9.5967&6.12&0.88&0.18&70.0&0.2&6.9&12.4&218.5&523.0&5120.0
		\\
		\hline
		\multirow{3}*{$N\left(0,1\right)$}&AS-PPA&7.5-7&9.9378&11.48&1.66&0.69&68.4&1.8&\textbf{4.8}&11.9&209.8&558.0&5283.7
		\\
		~&eAS-PPA&7.3-7&9.9378&11.48&1.66&0.69&68.4&1.8&5.8&11.9&207.8&558.0&5283.4
		\\
		~&eaAS-PPA&7.0-7&9.9378&11.48&1.66&0.69&68.4&1.8&6.4&11.9&207.0&544.0&5290.1
		\\
		\hline
		\multirow{3}*{$N\left(0,2\right)$}&AS-PPA&6.9-7&10.2303&15.37&2.24&1.32&68.1&3.0&\textbf{4.8}&11.6&206.4&610.0&5370.3
		\\
		~&eAS-PPA&7.3-7&10.2303&15.37&2.24&1.32&68.1&3.0&5.8&11.6&207.9&610.0&5367.5
		\\
		~&eaAS-PPA&7.3-7&10.2303&15.37&2.24&1.32&68.1&3.0&6.2&11.6&209.4&619.0&5365.5
		\\
		\hline
		\multirow{3}*{$MN$}&AS-PPA&6.9-7&9.9725&12.16&1.76&0.78&68.0&2.0&\textbf{4.7}&11.6&202.2&561.0&5267.0
		\\
		~&eAS-PPA&6.7-7&9.9725&12.16&1.76&0.78&68.0&2.0&5.7&11.6&202.3&561.0&5270.3
		\\
		~&eaAS-PPA&6.8-7&9.9725&12.16&1.76&0.78&68.0&2.0&6.2&11.5&199.6&616.0&5264.6
		\\
		\hline
		\multirow{3}*{$\surd 2t_4$}&AS-PPA&6.6-7&10.5935&18.58&2.71&2.07&65.7&4.2&\textbf{4.3}&11.1&193.6&489.0&5146.6
		\\
		~&eAS-PPA&6.5-7&10.5935&18.58&2.71&2.07&65.7&4.2&4.6&11.1&193.6&489.0&5148.6
		\\
		~&eaAS-PPA&6.5-7&10.5935&18.58&2.71&2.07&65.7&4.2&5.7&11.1&197.1&490.0&5149.2
		\\
		\hline
		\multirow{3}*{Cauthy}&AS-PPA&6.1-7&18.1334&26.03&3.81&5.03&62.8&7.1&\textbf{3.9}&10.5&194.8&531.0&5112.7
		\\
		~&eAS-PPA&6.3-7&18.1334&26.03&3.81&5.03&62.8&7.1&4.5&10.5&194.3&529.0&5107.0
		\\
		~&eaAS-PPA&6.8-7&18.1334&26.03&3.81&5.03&62.8&7.1&5.1&10.5&196.0&508.0&5068.3
		\\
		\hline
	\end{tabular}\label{tab-num-12}
\end{table}

{\bf Effect of Sieving strategy.} In order to visualize the effect of our sieving strategy, we draw the following diagram for $m=10,\ M=25$ on E2. In this way, we can observe whether the added components gradually cover $I(\hat{x})$ after each iteration of AS. In Figure \ref{fig-3}, the white circles represent the nonzero components of the true solution $\hat{x}$, and the black circles represent the components selected in each sieving. We can find that as the number of iterations increases, the white circles are gradually covered. This indicates that AS gradually selects the real nonzero components. It verifies that our AS strategy can effectively sieve the components that are nonzero after adding new components each time.

\begin{figure}[!htb]
	\vspace{5pt}
	\centering
	\begin{minipage}[t]{1.0\linewidth}
		\includegraphics[width=1\textwidth]{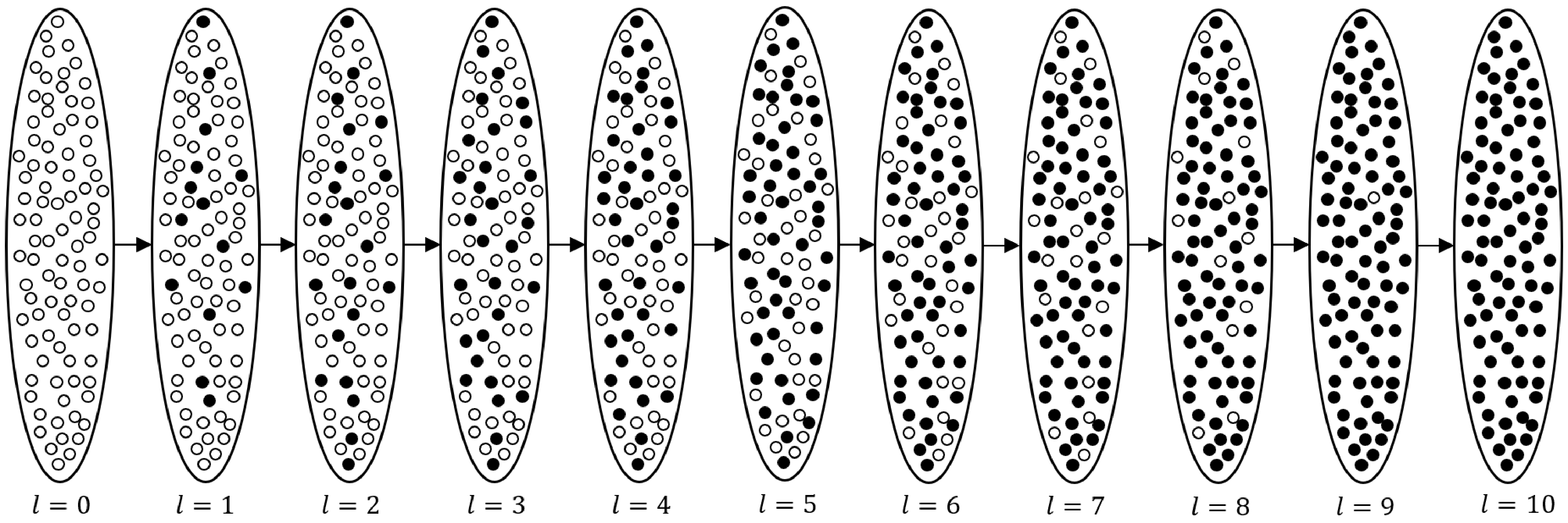}
	\end{minipage}
	\caption{Schematic diagram of sieving effect in AS-PPA.}\label{fig-3}
\end{figure}
\subsection{Comparison with Other Methods}\label{section5.3}
In this part, we compare Algorithm \ref{alg1} with other methods for \eqref{eq19}, including Gurobi for the linear programming formulation of \eqref{eq19}, Barrodale-Roberts modified simplex algorithm \cite{BR} (BR) for LAD formulation of \eqref{eq19} and the iteratively reweighted least squares (IRWLS). The following linear programming formulation is used to solve \eqref{eq19} \cite[Page 1709]{Wang}:
\begin{equation}\label{eq7}
	\begin{split}
		\min\limits_{\xi,\eta}&\left\{2[n\left(n-1\right)]^{-1}\mathop{\sum\sum}\limits_{i\neq j}\left(\xi_{ij}^++\xi_{ij}^-\right)+\lambda\sum_{k=1}^{p}\eta_k\right\}\\
		{\rm s.t.}\ &\xi_{ij}^+-\xi_{ij}^-=\left(b_i-b_j\right)-\left(a_i-a_j\right)^\top x,\ i,j=1,2,\dots,n,\ i\neq j,\\
		&\xi_{ij}^+\ge 0,\ \xi_{ij}^-\ge 0,\ i,j=1,2,\dots,n,i\neq j,\\
		&\eta_k\ge x_k,\ \eta_k\ge -x_{k},\ k=1,2,\dots,p.
	\end{split}
\end{equation}
We use the Gurobi 9.5.1 package\footnote[1]{https://www.gurobi.com/} to solve \eqref{eq7}\footnote[2]{The 'deterministic concurrent' method is used in Gurobi.}. There are $\frac{n(n-1)}{2}$ equality constraints, $n(n-1)$ box constraints, $2p$ inequality constraints, and the dimension of variables is $n(n-1)+2p$. For LAD, we use the Fortran code in \cite{BRcode}. For IRWLS, we use IRWLS provided in the RobustSP toolbox\footnote[3]{https://github.com/RobustSP/toolbox} by Zoubir et al. \cite{code}.

Table \ref{tab-num-1} and \ref{tab-num-1-1} summarizes the results for the small-scale test problems with $n=100,\ p=400$. All the results are taken average over 50 random runs. From Table \ref{tab-num-1}, we can see that AS-PPA, BP and Gurobi return almost the same solutions due to almost the same function values in $L_1,\ L_2$, ME, FP and FN. For IRWLS, FN is significantly different from those given by AS-PPA and Gurobi, indicating that the returned solutions by IRWLS may provide extra false negative components. In terms of cputime, both AS-PPA and Gurobi are much faster than IRWLS and BR, whereas AS-PPA is the fastest one. This can be explained by the fact that BR, Gurobi and IRWLS do not explore the sparsity of solutions. Furthermore, the linear programming problem contains 4950 equality constraints, 9900 box constraints, 800 inequality constraints, and variable dimension is 10700. In BR algorithm, the LAD problem is an $L_1$ minimization problem on a 5350-dimensional vector. For our AS-PPA, there are 900 variables and 500 equality constraints. This also explains why AS-PPA is faster than Gurobi and BR.
\begin{table}[!htb]
		\caption{Results on E1 ($n=100$, $p=400$).}
		 \label{tab-num-1}
					\begin{tabular}{|c|c| c|r| r|c|c| r| c| r|}
					\hline
						\makebox[0.04\textwidth][c]{$\Sigma$}&
						\makebox[0.05\textwidth][c]{Error} &
						\makebox[0.05\textwidth][c]{Method}&
						\mc{1}{c|}{val}
						&
						\mc{1}{c|}{$L_1$}
						& 
						\makebox[0.05\textwidth][c]{$L_2$}& 
						\makebox[0.05\textwidth][c]{ME} &
						\mc{1}{c|}{FP}
						& 
						\makebox[0.05\textwidth][c]{FN}& 
						\mc{1}{c|}{$\hat{t}(s)$}\\
						\hline
						\multirow{29}*{$\Sigma_3$}&\multirow{4}*{$N\left(0,0.25\right)$} &AS-PPA&2.7125&0.77$\ $&0.34&0.08&\textbf{8.4}&0.0&\textbf{0.1}\\
						~&~&Gurobi&2.7125&0.77$\ $&0.34&0.08&\textbf{8.4}&0.0&4.2\\
						~&~&IRWLS&2.7125&0.77$\ $&0.34&0.08&27.3&0.0&67.1\\
						~&~&BR&2.7125&0.77$\
						$&0.34&0.08&\textbf{8.4}&0.0&21.5\\
						\cmidrule{2-10}
						~&\multirow{4}*{$N\left(0,1\right)$} &AS-PPA&\textbf{3.1889}&1.53$\ $&0.68&0.32&\textbf{8.5}&0.0&\textbf{0.1}\\
						~&~&Gurobi&\textbf{3.1889}&1.53$\ $&0.68&0.32&\textbf{8.5}&0.0&4.1\\
						~&~&IRWLS&3.1890&1.53$\ $&0.68&0.32&39.9&0.0&66.7\\
						~&~&BR&\textbf{3.1889}&1.53$\ $&0.68&0.32&\textbf{8.5}&0.0&20.6\\
						\cmidrule{2-10}
						~&\multirow{4}*{$N\left(0,2\right)$} &AS-PPA&\textbf{3.5836}&2.17$\ $&0.97&0.65&\textbf{8.5}&0.0&\textbf{0.1}\\
						~&~&Gurobi&\textbf{3.5836}&2.17$\ $&0.97&0.65&\textbf{8.5}&0.0&4.4\\
						~&~&IRWLS&3.5837&2.17$\ $&0.97&0.65&51.6&0.0&72.9\\
						~&~&BR&\textbf{3.5836}&2.17$\ $&0.97&0.65&\textbf{8.5}&0.0&20.1\\
						\cmidrule{2-10}
						~&\multirow{4}*{$MN$}&AS-PPA&\textbf{3.2537}&1.64$\ $&0.74&0.38&\textbf{8.2}&0.0&\textbf{0.1}\\
						~&~&Gurobi&\textbf{3.2537}&1.64$\ $&0.74&0.38&\textbf{8.2}&0.0&4.3\\
						~&~&IRWLS&3.2538&1.64$\ $&0.74&0.38&44.9&0.0&72.1\\
						~&~&BR&\textbf{3.2537}&1.64$\ $&0.74&0.38&\textbf{8.2}&0.0&20.4\\
						\cmidrule{2-10}
						~&\multirow{4}*{$\surd 2t_4$} &AS-PPA&\textbf{4.0392}&2.62$\ $&1.17&0.95&\textbf{8.6}&0.0&\textbf{0.1}\\
						~&~&Gurobi&\textbf{4.0392}&2.62$\ $&1.17&0.95&\textbf{8.6}&0.0&4.3\\
						~&~&IRWLS&4.0393&2.62$\ $&1.17&0.95&61.5&0.0&73.3\\
						~&~&BR&\textbf{4.0392}&2.62$\ $&1.17&0.95&\textbf{8.6}&0.0&20.3\\
						\cmidrule{2-10}
						~&\multirow{4}*{Cauthy}&AS-PPA&\textbf{14.1917}&3.44$\ $&1.53&1.84&\textbf{7.5}&0.1&\textbf{0.2}\\
						~&~&Gurobi&\textbf{14.1917}&3.44$\ $&1.53&1.84&\textbf{7.5}&0.1&4.5\\
						~&~&IRWLS&14.1918&3.44$\ $&1.53&1.84&81.7&\textbf{0.0}&72.5\\
						~&~&BR&\textbf{14.1917}&3.44$\ $&1.53&1.84&\textbf{7.5}&0.1&18.8\\
						\hline
					\end{tabular}
				\end{table}
			
\begin{table}[!htb]
		\caption{Results on E2 - E3 ($n=100$, $p=400$).} 
		\label{tab-num-1-1}
		\begin{tabular}{|c|c| c|r| r|c|c| r| c| r|}\hline
				\makebox[0.04\textwidth][c]{$\Sigma$}&
				\makebox[0.05\textwidth][c]{Error} &
				\makebox[0.05\textwidth][c]{Method}&
				\mc{1}{c|}{val}
				&
				\mc{1}{c|}{$L_1$}
				& 
				\makebox[0.05\textwidth][c]{$L_2$}& 
				\makebox[0.05\textwidth][c]{ME} &
				\mc{1}{c|}{FP}
				& 
				\makebox[0.05\textwidth][c]{FN}& 
				\mc{1}{c|}{$\hat{t}(s)$}\\
				\hline
				\mc{10}{|c|}{E2}\\
				\hline
				\multirow{29}*{$\Sigma_3$}&\multirow{4}*{$N\left(0,0.25\right)$} &AS-PPA&\textbf{12.5636}&\textbf{15.79}&2.49&0.60&45.4&3.2&\textbf{3.5}\\
				~&~&Gurobi&\textbf{12.5636}&\textbf{15.79}&2.49&0.60&45.4&3.2&5.8\\
				~&~&IRWLS&12.5638&15.80&2.49&0.60&45.4&3.2&71.4\\
				~&~&BR&\textbf{12.5636}&\textbf{15.79}&2.49&0.60&45.4&3.2&583.9\\
				\cmidrule{2-10}
				~&\multirow{4}*{$N\left(0,1\right)$} &AS-PPA&\textbf{12.7503}&\textbf{20.87}&\textbf{3.32}&1.43&\textbf{44.5}&5.1&\textbf{3.9}\\
				~&~&Gurobi&\textbf{12.7503}&\textbf{20.87}&\textbf{3.32}&1.43&\textbf{44.5}&5.1&5.4\\
				~&~&IRWLS&12.7504&20.89&3.33&1.43&44.8&5.1&71.3\\
				~&~&BR&\textbf{12.7503}&\textbf{20.87}&\textbf{3.32}&1.43&\textbf{44.5}&5.1&430.6\\
				\cmidrule{2-10}
				~&\multirow{4}*{$N\left(0,2\right)$} &AS-PPA&\textbf{12.9379}&\textbf{24.09}&\textbf{3.86}&\textbf{2.33}&\textbf{42.8}&6.4&\textbf{2.4}\\
				~&~&Gurobi&\textbf{12.9379}&\textbf{24.09}&\textbf{3.86}&\textbf{2.33}&42.9&6.4&5.4\\
				~&~&IRWLS&12.9380&24.10&3.87&2.34&43.0&6.4&71.5\\
				~&~&BR&\textbf{12.9379}&\textbf{24.09}&\textbf{3.86}&\textbf{2.33}&42.9&6.4&381.4\\
				\cmidrule{2-10}
				~&\multirow{4}*{$MN$}&AS-PPA&\textbf{12.7426}&\textbf{21.51}&\textbf{3.41}&1.62&\textbf{44.6}&5.4&\textbf{2.5}\\
				~&~&Gurobi&\textbf{12.7426}&\textbf{21.51}&\textbf{3.41}&1.62&\textbf{44.6}&5.4&5.6\\
				~&~&IRWLS&12.7427&21.52&3.42&1.62&44.8&5.4&71.5\\
				~&~&BR&\textbf{12.7426}&\textbf{21.51}&\textbf{3.41}&1.62&\textbf{44.6}&5.4&441.1\\
				\cmidrule{2-10}
				~&\multirow{4}*{$\surd 2t_4$} &AS-PPA&\textbf{13.2316}&26.61&4.31&3.40&41.1&7.8&\textbf{2.1}\\
				~&~&Gurobi&\textbf{13.2316}&26.61&4.31&3.40&\textbf{41.0}&7.8&5.1\\
				~&~&IRWLS&13.2317&26.61&4.31&3.40&41.1&7.8&70.9\\
				~&~&BR&\textbf{13.2316}&26.61&4.31&3.40&\textbf{41.0}&7.8&281.8\\
				\cmidrule{2-10}
				~&\multirow{4}*{Cauthy}&AS-PPA&\textbf{22.9788}&\textbf{34.29}&5.66&9.30&35.8&12.1&\textbf{1.4}\\
				~&~&Gurobi&\textbf{22.9788}&\textbf{34.29}&5.66&9.30&35.8&12.1&4.9\\
				~&~&IRWLS&22.9789&34.30&5.66&9.30&\textbf{35.7}&12.1&71.3\\
				~&~&BR&\textbf{22.9788}&\textbf{34.29}&5.66&9.30&35.8&12.1&161.5\\
				\hline
				\mc{10}{|c|}{E3}\\
				\hline
				\multirow{14}*{$\Sigma_1$}&\multirow{4}*{$N\left(0,1\right)$}&AS-PPA&\textbf{2.8645}&2.59$\ $&1.01&0.24&\textbf{11.2}&0.0&\textbf{0.1}\\
				~&~&Gurobi&\textbf{2.8645}&2.59$\ $&1.01&0.24&\textbf{11.2}&0.0&4.4\\
				~&~&IRWLS&2.8646&2.59$\ $&1.01&0.24&122.9&0.0&70.9\\
				~&~&BR&\textbf{2.8645}&2.59$\ $&1.01&0.24&\textbf{11.2}&0.0&24.7\\
				\cmidrule{2-10}
				~&\multirow{4}*{MN}&AS-PPA&\textbf{2.9327}&2.76$\ $&1.09&0.29&\textbf{11.5}&0.0&\textbf{0.1}\\
				~&~&Gurobi&\textbf{2.9327}&2.76$\ $&1.09&0.29&\textbf{11.5}&0.0&4.4\\
				~&~&IRWLS&2.9328&2.76$\ $&1.09&0.29&127.4&0.0&70.7\\
				~&~&BR&\textbf{2.9327}&2.76$\ $&1.09&0.29&\textbf{11.5}&0.0&24.7\\
				\cmidrule{2-10}
				~&\multirow{4}*{Cauthy}&AS-PPA&\textbf{13.9108}&5.33$\ $&2.06&1.10&\textbf{11.0}&0.3&\textbf{0.2}\\
				~&~&Gurobi&\textbf{13.9108}&5.33$\ $&2.06&1.10&\textbf{11.0}&0.3&4.5\\
				~&~&IRWLS&13.9109&5.33$\ $&2.06&1.10&141.1&\textbf{0.0}&70.9\\
				~&~&BR&\textbf{13.9108}&5.33$\ $&2.06&1.10&\textbf{11.0}&0.3&21.9\\
				\hline
				\multirow{14}*{$\Sigma_2$}&\multirow{4}*{$N\left(0,1\right)$}&AS-PPA&3.2881&1.24$\ $&0.71&0.58&\textbf{2.1}&0.0&\textbf{0.1}\\
				~&~&Gurobi&3.2881&1.24$\ $&0.71&0.58&\textbf{2.1}&0.0&4.8\\
				~&~&IRWLS&3.2881&1.24$\ $&0.71&0.58&4.8&0.0&72.0\\
				~&~&BR&3.2881&1.24$\ $&0.71&0.58&\textbf{2.1}&0.0&15.8\\
				\cmidrule{2-10}
				~&\multirow{4}*{MN}&AS-PPA&\textbf{3.3529}&1.34$\ $&0.76&0.67&\textbf{2.0}&0.0&\textbf{0.1}\\
				~&~&Gurobi&\textbf{3.3529}&1.34$\ $&0.76&0.67&\textbf{2.0}&0.0&4.8\\
				~&~&IRWLS&3.3530&1.34$\ $&0.76&0.67&5.4&0.0&72.5\\
				~&~&BR&\textbf{3.3529}&1.34$\ $&0.76&0.67&\textbf{2.0}&0.0&15.6\\
				\cmidrule{2-10}
				~&\multirow{4}*{Cauthy}&AS-PPA&\textbf{14.1796}&3.15$\ $&1.78&3.98&\textbf{1.8}&0.3&\textbf{0.1}\\
				~&~&Gurobi&\textbf{14.1796}&3.15$\ $&1.78&3.98&\textbf{1.8}&0.3&4.8\\
				~&~&IRWLS&14.1797&3.15$\ $&1.78&3.98&80.0&\textbf{0.0}&70.9\\
				~&~&BR&\textbf{14.1796}&3.15$\ $&1.78&3.98&\textbf{1.8}&0.3&14.9\\
				\hline
				\multirow{14}*{$\Sigma_3$}&\multirow{4}*{$N\left(0,1\right)$}&AS-PPA&\textbf{3.1889}&1.53$\ $&0.68&0.32&\textbf{8.5}&0.0&\textbf{0.1}\\
				~&~&Gurobi&\textbf{3.1889}&1.53$\ $&0.68&0.32&\textbf{8.5}&0.0&4.3\\
				~&~&IRWLS&3.1890&1.53$\ $&0.68&0.32&39.9&0.0&72.0\\
				~&~&BR&\textbf{3.1889}&1.53$\ $&0.68&0.32&\textbf{8.5}&0.0&20.7\\
				\cmidrule{2-10}
				~&\multirow{4}*{MN}&AS-PPA&\textbf{3.2537}&1.64$\ $&0.74&0.38&\textbf{8.2}&0.0&\textbf{0.1}\\
				~&~&Gurobi&\textbf{3.2537}&1.64$\ $&0.74&0.38&\textbf{8.2}&0.0&4.4\\
				~&~&IRWLS&3.2538&1.64$\ $&0.74&0.38&44.9&0.0&76.7\\
				~&~&BR&\textbf{3.2537}&1.64$\ $&0.74&0.38&\textbf{8.2}&0.0&21.1\\
				\cmidrule{2-10}
				~&\multirow{4}*{Cauthy}&AS-PPA&\textbf{14.1917}&3.44$\ $&1.53&1.84&\textbf{7.5}&0.1&\textbf{0.1}\\
				~&~&Gurobi&\textbf{14.1917}&3.44$\ $&1.53&1.84&\textbf{7.5}&0.1&4.5\\
				~&~&IRWLS&14.1918&3.44$\ $&1.53&1.84&81.7&\textbf{0.0}&72.6\\
				~&~&BR&\textbf{14.1917}&3.44$\ $&1.53&1.84&\textbf{7.5}&0.1&18.7\\
				\hline
\end{tabular}
\end{table}

To further investigate the cputime consumed by each method, we choose E2 as a typical example with error $N(0,0.25)$ for different $n$'s and $p=5n$. As we can see from Figure \ref{fig-4a}, the cputime taken by BR and IRWLS increases dramatically whereas Gurobi increases in a slower trend than the above two. In order to see the time comparison more clearly, we show the results of Figure \ref{fig-4a} in the Table \ref{tab-num-2}. In Table \ref{tab-num-2}, '--' indicates the cputime is more than one hour, that is after $n$ reaches 150, the cputime taken by BR exceeds one hour, while Gurobi and IRWLS exceed one hour after $n$ reaches 300. For those results that take more than an hour, we do not show them in Figure \ref{fig-4a}. In contrast, the time taken by AS-PPA increases the slowest as $n$ grows. Figure \ref{fig-4b} further zooms in to demonstrate the cputime taken by AS-PPA and Gurobi. As $n$ grows, it can be seen that the cputime by AS-PPA grows slowly whereas the cputime by Gurobi increases dramatically. For $n=250,\ p=1250$, AS-PPA takes less than 10 seconds which is much smaller than 150 seconds by Gurobi. The fast speed of AS-PPA can be further explained by the smaller scale of subproblems due to the efficient adaptive sieving strategy, as shown in Figure \ref{fig-5}. In Figure \ref{fig-5}, the size of each subproblem is much smaller (less than $11\%$), compared with the scale of the original problem. Due to the above observations, below we only compare with Gurobi.
	\begin{figure}[H]
	\centering
	\subfigure[]{
		\includegraphics[width=0.45\textwidth]{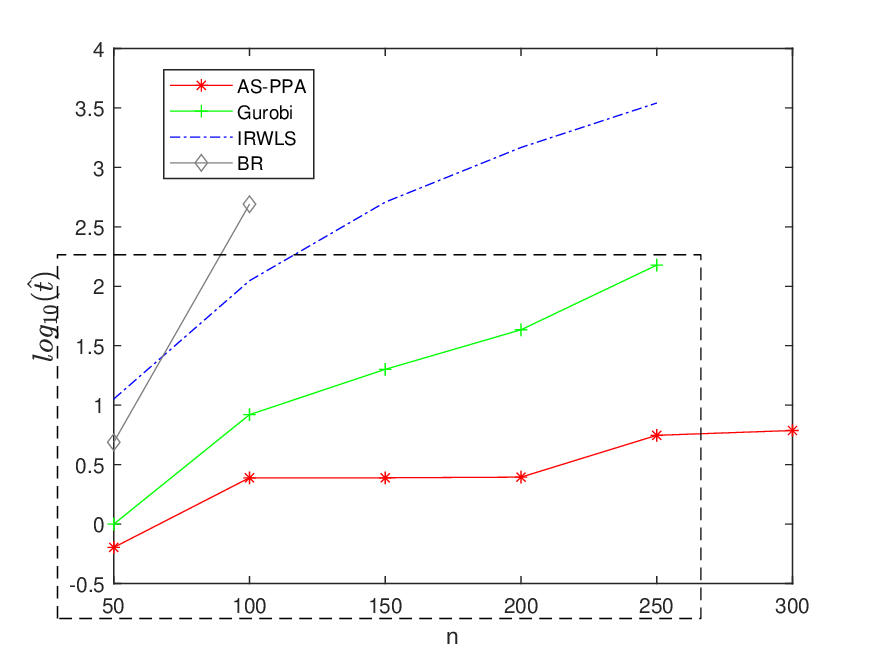}
		\label{fig-4a}
	}
	\subfigure[]{
		\includegraphics[width=0.45\textwidth]{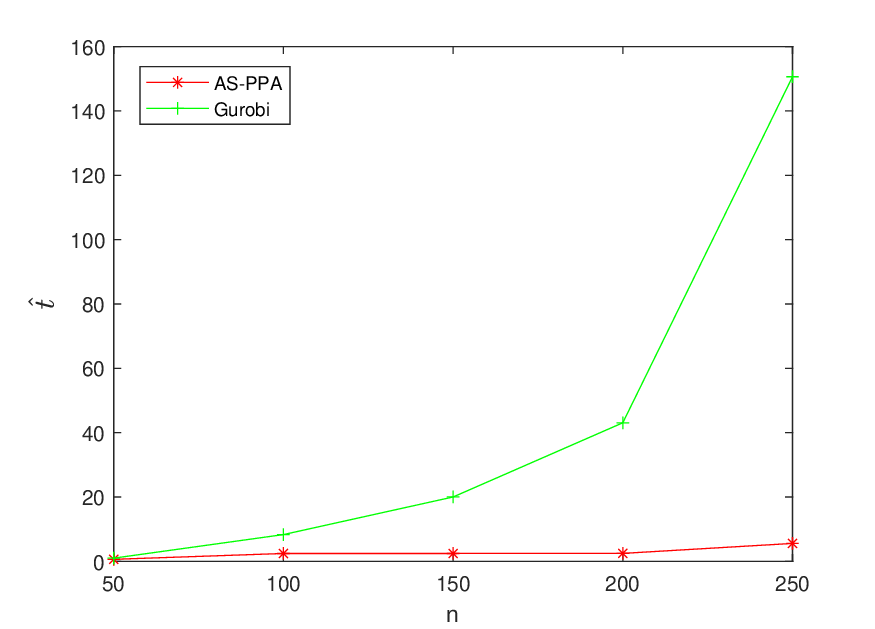}
		\label{fig-4b}
	}
	\caption{$\hat{t}(s)$ for different methods ($p=5n$)}
\end{figure}
\begin{figure}[H]
	\centering
	\subfigure[]{
		\includegraphics[width=0.45\textwidth]{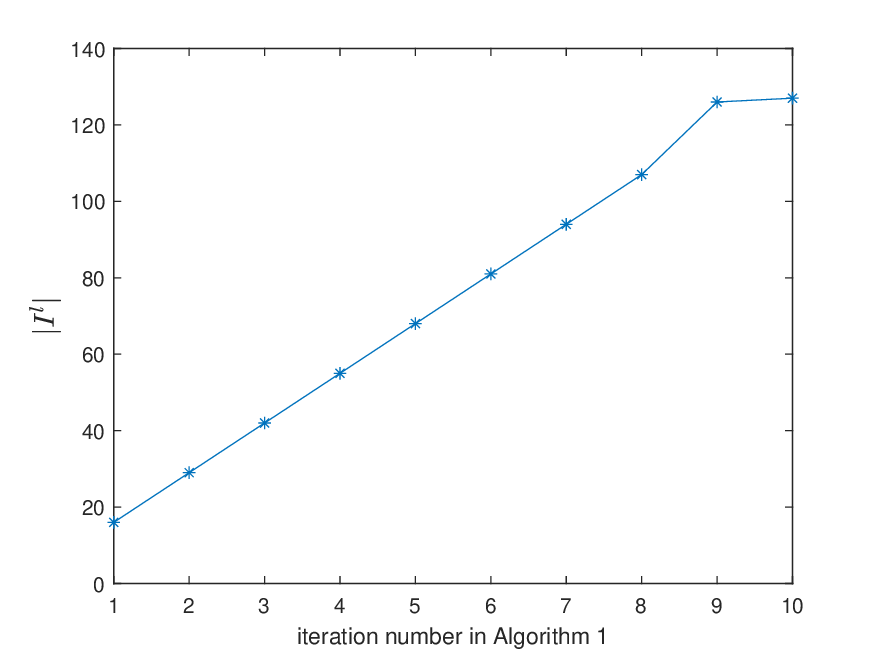}
		\label{fig-5a}
	}
	\subfigure[]{
		\includegraphics[width=0.45\textwidth]{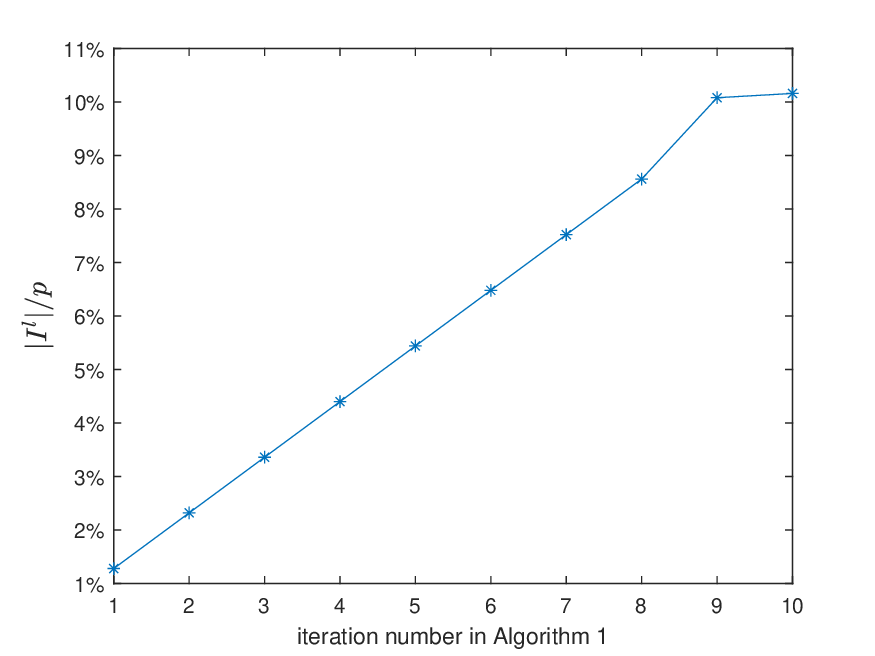}
		\label{fig-5b}
	}
	\caption{${|I^l|}$ and $\frac{|I^l|}{p}$ in AS-PPA for E2 with $n=250$, $p=1250$.}
	\label{fig-5}
\end{figure}

	\begin{table}[htb]
		\centering
		\caption{Cputime for different methods on E2 ($n=200,\ p=1000$).}
				\begin{tabular}{|c|c| c|c| c|}\hline
					\makebox[0.05\textwidth][c]{$(n,p)$}& 
					\makebox[0.05\textwidth][c]{BR} &
					\makebox[0.05\textwidth][c]{IRWLS}& 
					\makebox[0.05\textwidth][c]{Gurobi} & 
					\makebox[0.08\textwidth][c]{AS-PPA} \\
					\hline
					$(50,250)$&4.88&11.27&1.00&0.92\\
					\hline
					$(100,500)$&490.28&111.39&8.32&2.70\\
					\hline
					$(150,750)$&--&510.40&20.01&2.68\\
					\hline
					$(200,1000)$&--&1470.19&43.07&2.80\\
					\hline
					$(250,1250)$&--&3480.74&150.63&4.91\\
					\hline
					$(300,1500)$&--&--&--&5.35\\
					\hline
				\end{tabular}\label{tab-num-2}
\end{table}

Next, we test E1 - E3 on larger-scale problems. That is, $n=200,\ p=1000$. As we can see from Table \ref{tab-num-4}, the quality of solutions provided by AS-PPA and Gurobi are the same, indicate by the same values in val, $L_1,\ L_2$, ME, FP and FN. However, AS-PPA is much faster than Gurobi. Even for E2, which is a difficult case, the cputime consumed by AS-PPA is ten times faster than Gurobi. The reason is that the linear programming problem has 41800 variables with 19900 equality constraints, 39800 box constraints and 2000 inequality constraints. In contrast, for AS-PPA, there are only 2200 variables and 1200 equality constraints and the subproblem in AS algorithm is even in a much smaller scale.

\begin{table}[h]
		\caption{Results on E1 - E3 ($n=200$, $p=1000$).}
		\label{tab-num-4}
			\begin{tabular}{|c|c|r|c|c|c|c|c|r|} \hline
				\multirow{2}*{$\Sigma$}&\multirow{2}*{Error} &
				\multirow{2}*{val$\ \ \ $} &\multirow{2}*{$L_1$} &  \multirow{2}*{$L_2$} & \multirow{2}*{ME}& \multirow{2}*{ FP}&  \multirow{2}*{FN}&  \mc{1}{c|}{$\hat{t}(s)$}\\ 
				~&~ &~&~&~&~&~&~&\mc{1}{c|}{\tiny AS-PPA$\ $$|$$\ $Gurobi \normalsize}\\
				\hline
				\mc{9}{|c|}{E1}\\
				\hline
				\multirow{6}*{$\Sigma_3$}&$N\left(0,0.25\right)$&2.1744&0.59&0.25&0.04&10.78&0&\textbf{0.1}$|$$\ \ \ $34.5\\
				~&$N\left(0,0.1\right)$&2.6893&1.18&0.50&0.18&10.84&0&\textbf{0.1}$|$$\ \ \ $34.4\\
				~&$N\left(0,0.2\right)$&3.1158&1.66&0.71&0.35&10.90&0&\textbf{0.1}$|$$\ \ \ $33.4\\
				~&MN&2.7576&1.24&0.53&0.20&10.96&0&\textbf{0.1}$|$$\ \ \ $34.5\\
				~&$\sqrt{2}t_4$&3.5774&2.00&0.85&0.51&10.80&0&\textbf{0.1}$|$$\ \ \ $33.9\\
				~&Cauthy&11.3205&2.63&1.09&0.86&12.14&0&\textbf{0.2}$|$$\ \ \ $33.6\\
				
				\hline
				\mc{9}{|c|}{E2}\\
				\hline
				\multirow{6}*{$\Sigma_3$}&$N\left(0,0.25\right)$&9.5967&6.12&0.88&0.18&69.96&0.16&\textbf{4.9}$|$$\ \ \ $48.1
				\\
				~&$N\left(0,0.1\right)$&9.9378&11.48&1.66&0.69&68.42&1.84&\textbf{5.5}$|$$\ \ \ $47.0
				\\
				~&$N\left(0,0.2\right)$&10.2303&15.37&2.24&1.32&68.06&3.04&\textbf{6.6}$|$$\ \ \ $46.4
				\\
				~&MN&9.9725&12.16&1.76&0.78&68.04&1.98&\textbf{5.6}$|$$\ \ \ $47.6
				\\
				~&$\sqrt{2}t_4$&10.5935&18.58&2.71&2.07&65.68&4.16&\textbf{4.5}$|$$\ \ \ $45.9
				\\
				~&Cauthy&18.1334&26.03&3.81&5.03&62.80&7.12&\textbf{3.9}$|$$\ \ \ $44.0
				\\
				\hline
				\mc{9}{|c|}{E3}\\
				\hline
				\multirow{3}*{$\Sigma_1$}&$N\left(0,0.1\right)$&2.4285&2.04&0.77&0.15&15.12&0&\textbf{0.2}$|$$\ \ \ $33.6
				\\
				~&MN&2.4994&2.15&0.82&0.16&15.10&0&\textbf{0.2}$|$$\ \ \ $34.1
				\\
				~&Cauthy&11.0742&4.45&1.64&0.70&15.92&0.04&\textbf{0.3}$|$$\ \ \ $33.7
				\\
				\hline
				\multirow{3}*{$\Sigma_2$}&$N\left(0,0.1\right)$&2.7884&0.84&0.47&0.26&2.84&0&\textbf{0.1}$|$$\ \ \ $33.1
				\\
				~&MN&2.8546&0.89&0.50&0.30&2.96&0&\textbf{0.1}$|$$\ \ \ $33.5
				\\
				~&Cauthy&11.3876&1.92&1.04&1.32&3.26&0&\textbf{0.1}$|$$\ \ \ $33.6
				\\
				\hline
				\multirow{3}*{$\Sigma_3$}&$N\left(0,0.1\right)$&2.6893&1.18&0.50&0.18&10.84&0&\textbf{0.1}$|$$\ \ \ $34.3
				\\
				~&MN&2.7576&1.24&0.53&0.20&10.96&0&\textbf{0.1}$|$$\ \ \ $34.4
				\\
				~&Cauthy&11.3205&2.63&1.09&0.86&12.14&0&\textbf{0.2}$|$$\ \ \ $33.4
				\\
				\hline
			\end{tabular}		
	\end{table}

To further demonstrate the performance of AS-PPA, we expand the data scale to $n=2000,\ p=10000$. This scale is already beyond the reach of Gurobi, because it needs to generate a matrix of 1999000$\times$10000, occupying nearly 150G of memory, exceeding the capacity of our laptop. So we only show the results of AS-PPA in Table \ref{tab-num-9}. It can be seen that AS-PPA is quite efficient, even for the difficult example E2. In particular, for the Cauchy distribution of $\Sigma_3$ in E2, it takes less than one minute to return a solution.

\begin{table}[h]
	\caption{Results on E1 - E3 in AS-PPA ($n=2000$, $p=10000$).}
	\label{tab-num-9}
			\begin{tabular}{|c|c| c|c| c|c|c| c| r| c|c|c|c|}\hline
				\makebox[0.04\textwidth][c]{$\Sigma$}&
				\makebox[0.05\textwidth][c]{Error} &
				\makebox[0.05\textwidth][c]{val}&
				\makebox[0.05\textwidth][c]{$L_1$} & 
				\makebox[0.05\textwidth][c]{$L_2$}& 
				\makebox[0.05\textwidth][c]{ME} &
				\makebox[0.05\textwidth][c]{FP}& 
				\makebox[0.05\textwidth][c]{FN}&
				\mc{1}{c|}{ $\hat{t}(s)$}
				& 
				\makebox[0.05\textwidth][c]{It-AS}& 
				\makebox[0.05\textwidth][c]{It-PPA}& 
				\makebox[0.05\textwidth][c]{It-ALM}& 
				\makebox[0.05\textwidth][c]{It-SSN} \\
				\hline
				\mc{13}{|c|}{E1}\\
				\hline
				\multirow{6}*{$\Sigma_3$}&$N\left(0,0.25\right)$ &1.1403&0.2326&0.0893&0.0055&17.9&0.0&1.7&2.0&22.5&74.0&928.0
				\\
				~&$N\left(0,1\right)$ &1.6967&0.4651&0.1786&0.0221&18.5&0.0&1.4&2.0&26.0&72.0&797.8
				\\
				~&$N\left(0,2\right)$ &2.1576&0.6578&0.2525&0.0441&18.6&0.0&1.4&2.0&23.2&75.0&816.6
				\\
				~&$MN$&1.7765&0.4952&0.1905&0.0253&18.1&0.0&1.4&2.0&27.7&100.0&793.3
				\\
				~&$\surd 2t_4$ &2.6392&0.7690&0.2946&0.0599&19.4&0.0&1.5&2.0&26.5&80.0&812.9
				\\
				~&Cauthy&19.3751&0.8081&0.3151&0.0692&17.7&0.0&2.6&3.0&49.5&135.0&1592.5
				\\
				\hline
				\mc{13}{|c|}{E2}\\
				\hline
				\multirow{6}*{$\Sigma_3$}&$N\left(0,0.25\right)$ &3.7933&1.7679&0.2227&0.0246&132.6&0.0&35.6&5.3&61.7&230.0&6148.2
				\\
				~&$N\left(0,1\right)$ & 4.3276&3.5358&0.4453&0.0983&135.8&0.0&33.8&4.5&58.6&164.0&5867.6
				\\
				~&$N\left(0,2\right)$ & 4.7702&5.0004&0.6298&0.1965&137.0&0.0&31.9&4.2&56.6&161.0&5671.2
				\\
				~&$MN$& 4.3945&3.8052&0.4797&0.1140&133.4&0.0&34.0&4.6&57.6&212.0&5915.4
				\\
				~&$\surd 2t_4$ & 
				5.2331&5.8990&0.7440&0.2745&136.0&0.0&31.5&4.2&65.1&170.0&5791.4
				\\
				~&Cauthy& 21.9557&6.9800&0.8821&0.3875&136.5&0.2&34.4&4.2&69.5&189.0&6567.0
				
				\\
				\hline
				\mc{13}{|c|}{E3}\\
				\hline
				\multirow{3}*{$\Sigma_1$}&$N\left(0,1\right)$&1.5900&0.8098&0.2803&0.0198&24.5&0.0&1.6&2.0&24.7&89.0&950.2
				\\
				~&MN& 1.6696&0.8631&0.2997&0.0229&23.2&0.0&1.6&2.0&26.7&83.0&972.8
				\\
				~&Cauthy&19.2677&1.4082&0.4937&0.0615&23.6&0.0&3.7&3.0&53.1&151.0&2026.2
				\\
				\hline
				\multirow{3}*{$\Sigma_2$}&$N\left(0,1\right)$& 1.7489&0.2993&0.1523&0.0278&6.5&0.0&0.5&2.0&20.1&63.0&318.2
				\\
				~&MN&1.8293&0.3190&0.1625&0.0318&6.2&0.0&0.5&2.0&22.6&71.0&328.8
				\\
				~&Cauthy&19.4256&0.5275&0.2724&0.0894&6.1&0.0&1.3&2.8&38.0&137.0&839.7
				\\
				\hline
				\multirow{3}*{$\Sigma_3$}&$N\left(0,1\right)$& 1.6967&0.4651&0.1786&0.0221&18.5&0.0&1.5&2.0&26.0&72.0&797.8
				\\
				~&MN& 1.7765&0.4952&0.1905&0.0253&18.1&0.0&1.5&2.0&27.7&100.0&793.3
				\\
				~&Cauthy& 19.3751&0.8081&0.3151&0.0692&17.7&0.0&2.8&3.0&49.5&135.0&1592.5\\
				\hline
			\end{tabular}
\end{table}

\subsection{Comparison with the Safe Feature Screening Method}\label{section5.4}
In this part, we compare AS-PPA with the safe feature screening method proposed in \cite{SafeRank}. The method is denoted as Safe-PPA, where PPA is also used to solve the subproblem. We test a sequence of $\lambda$ with $[0.1\lambda_{max}:0.1\lambda_{max}:\lambda_{max}]$, where $\lambda_{max}$ is derived in \cite{SafeRank}. We report the following information as in \cite{SafeRank}. The rejection ratio (RR) is the ratio between the number of discarded features by the screening rules and the actual inactive features corresponding to $\lambda_{max}$, that is, RR$=\frac{|I^s|}{|\bar{I}|}$, where $\bar{I}=\left\{i\ |\ x_i = 0,\ i=1,\dots,p\right\}$ is for the solution $x$ obtained for the original problem, and $I^s$ is the set of components that are removed after screening, that is, the components that the screening rule takes zeros. The ratio of speed up (Rsp) is defined as $T_f/T_s$, where $T_f$ is the time of solving the original problem before screening, $T_s$ is the time of solving the subproblem after screening. 

Table \ref{tab-num-10} shows the results under different dimensions, where $\hat{t}$ is the total time of taking 10 values for $\lambda=[0.1\lambda_{max}:0.1\lambda_{max}:\lambda_{max}]$. From Table \ref{tab-num-10}, we can see that AS-PPA is the winner in every way. It takes the smallest cputime after using our sieving strategy, and the ratio of speeding up (Rsp) by AS strategy is more than 80 times. For $n=10,\ p=40000$, the ratio of speeding up even reaches about 2148 times, which verifies that AS strategy is indeed very efficient. For the datasets with larger $n$, the sieving efficiency of Safe-PPA decreases. Especially on $n=100, p=5000$ and $n=100, p=10000$ datasets, Safe-PPA can hardly remove the components which is reflected in the RR value of 0. AS-PPA can still remove more than 90\% of nonzero components.

\begin{table}[h]
	\caption{Results on E4 \protect\footnotemark[1].}
	\label{tab-num-10}
				\begin{tabular}{|c|c| c|c|r| r|}\hline
					$(n,p)$&Method&RR &	\mc{1}{c|}{$T_f$}&\mc{1}{c|}{	$T_s$}&\mc{1}{c|}{Rsp}\\
					\hline
					\multirow{2}*{(10,5000)}&AS-PPA&0.9908&\multirow{2}*{8.685}&0.023&377.90
					\\
					~&Safe-PPA&0.9345&~&0.258&33.67
					\\
					\hline
					\multirow{2}*{(10,10000)}&AS-PPA&0.9904&\multirow{2}*{14.583}&0.019&781.89\\
					~&Safe-PPA&0.9966&~&0.339&42.99\\
					\hline
					\multirow{2}*{(10,40000)}&AS-PPA&0.9902&\multirow{2}*{145.436}&0.068&2148.60
					\\
					~&Safe-PPA&0.9892&~&3.289&44.22\\
					\hline
					\multirow{2}*{(50,5000)}&AS-PPA&0.9610&\multirow{2}*{72.545}&0.544&133.33\\
					~&Safe-PPA&0.2589&~&3.574&20.30
					\\
					\hline
					\multirow{2}*{(50,10000)}&AS-PPA&0.9763&\multirow{2}*{191.147}&0.483&396.14
					\\
					~&Safe-PPA&0.0491&~&9.414&20.31
					\\
					\hline
					\multirow{2}*{(50,40000)}&AS-PPA&0.9885&\multirow{2}*{1542.970}&1.024&1506.59
					\\
					~&Safe-PPA&0.2883&~&77.103&20.01
					\\
					\hline
					\multirow{2}*{(100,5000)}&AS-PPA&0.9290&\multirow{2}*{361.111}&4.220&85.57
					\\
					~&Safe-PPA&0.0000&~&17.869&20.21
					\\
					\hline
					\multirow{2}*{(100,10000)}&AS-PPA&0.9551&\multirow{2}*{1268.488}&6.153&206.15
					\\
					~&Safe-PPA&0.0000&~&57.868&21.92
					\\
					\hline
					\multirow{2}*{(100,40000)}&AS-PPA&0.9676&\multirow{2}*{8823.702}&4.588&1923.39
					\\
					~&Safe-PPA&0.0002&~&627.587&14.06\\
					
					\hline
				\end{tabular}
	\end{table}

	\subsection{Comparison on Square-Root Lasso Model}\label{section5.5}
	In addition to the hdr model, a model named square-root lasso is also of the form \eqref{eq19}. The square-root lasso model has the form of 
	\begin{equation*}
		\min\limits_{x\in\Re^p}\left\|Ax-b\right\|_2+\lambda\left\|x\right\|_1.
	\end{equation*}
	In \cite{square}, the authors transformed the square-root lasso problem into a cone programming problem (CCP) to solve. In \cite{PMM}, Tang et al. proposed a proximal majorization-minimization (PMM) algorithm to solve the square-root lasso problem. In this subsection, we contrast our method with these two methods of square-root lasso model. PMM can solve the square-root lasso model with proper nonconvex regularizations. For comparison, we only fix its loss function as $L_1$ norm and do not compare with non-convex regularizations. We choose $\lambda$ following the way in \cite{square}\footnote{$\lambda$ is set to $1.1\Phi^{-1}(1-0.05/(2n))$ with $\Psi$ being the cumulative normal distribution function.}.

	We report the numerical results in Tables \ref{tab-num-11} for different scales of test problems. We report the primal objective value (val), and the computation time ($\hat{t}$) in seconds. All the results are taken average over 50 random runs. On the small-scale test problems, we can see that AS-PPA and PMM return the same solutions due to the same objective value. For CCP, the value is slightly larger than the other two. Nevertheless, CCP still takes more time to solve the problem, 20 times as long as AS-PPA. In contrast, the performance of PMM is much better, even though it also takes more than 0.05s. For the problems with $n=1000,\ p=5000$, the solutions obtained by AS-PPA and PMM are still basically the same, and the quality is higher than that of CCP. It can be seen that the time consumption of CCP is still the longest, and that of PMM is also longer than that of AS-PPA. This may be because PMM also exploits the sparsity of the model, so although the scale increases, the time increase is not particularly large.
\begin{table}[h]
	\caption{Results on E5 - E6.}
	\label{tab-num-11}
	\begin{tabular}{|c|c|c|c|r|r|c| c|r| r|}\hline
		\makebox[0.05\textwidth][c]{Dataset}&
		\makebox[0.05\textwidth][c]{Error} &
		\makebox[0.05\textwidth][c]{$(n,p)$} &
		\makebox[0.05\textwidth][c]{Method}&
		\mc{1}{c|}{val}& 
		\mc{1}{c|}{$\hat{t}(s)$}&
		\makebox[0.05\textwidth][c]{$(n,p)$} &
		\makebox[0.05\textwidth][c]{Method}&
		\mc{1}{c|}{val}& 
		\mc{1}{c|}{$\hat{t}(s)$}\\
		\hline
		\multirow{3}*{E5}&\multirow{3}*{$N(0,1)$}&\multirow{6}*{$(100,500)$} &AS-PPA&\textbf{26.8793}&\textbf{0.0100}&\multirow{6}*{$(1000,5000)$}&AS-PPA&\textbf{247.6759}&\textbf{1.2942}\\
		~&~&~&CCP&26.9093&0.2145&~&CCP&247.6897&9.0384\\
		~&~&~&PMM&\textbf{26.8793}&0.1162&~&PMM&\textbf{247.6759}&2.2582\\
		\cmidrule{1-2}\cmidrule{4-6}\cmidrule{8-10}
		\multirow{3}*{E6}&\multirow{3}*{$t_4/\sqrt{2}$}&~ &AS-PPA&\textbf{34.3528}&\textbf{0.0093}&~&AS-PPA&\textbf{272.2574}&\textbf{1.2861}\\
		~&~&~&CCP&34.4057&0.1754&~&CCP&272.2851&8.9201\\
		~&~&~&PMM&\textbf{34.3528}&0.0620&~&PMM&\textbf{272.2574}&2.1847\\
		\hline
	\end{tabular}
\end{table}
\section{Conclusion}\label{section6}
In this paper, we designed an adaptive sieving algorithm for the hdr lasso problem, which takes full advantage of the sparsity of the solution. When solving the subproblem with the same form as the original model but with a small scale, we used PPA to solve it, which takes full use of the nondifferentiability of the loss function. Extensive numerical results demonstrate that AS-PPA is robust for different types of noises, which verifies the attractive statistical property as shown in \cite{Wang}. Moreover, AS-PPA is also highly efficient, especially for the case of high-dimensional features, compared with other methods.

\bmhead{Acknowledgments}
We would like to thank the team of Prof. Lingchen Kong from Beijing Jiaotong University for providing the code of the safe screening rule in \cite{SafeRank}. We would also like to thank the editor for landing our submission as well as two anonymous reviewers for their valuable comments based on which the manuscript was further improved. 
\bmhead{Data availability statement}
For the simulated data, they can be regenerated according to the instructions in the paper or obtained from the corresponding author.
\section*{Statements and Declarations}
\bmhead{Funding}
The research of Qingna Li is supported by the National Science Foundation of China (NSFC) 12071032. 
\bmhead{Conflict of interest}
All authors certify that they have no affiliations with or involvement in any organization or entity with any
financial interest or non-financial interest in the subject matter or materials discussed in this manuscript.

\begin{appendices}
\section{Calculation of $\partial h\left(u\right)$}\label{app1}
\begin{proposition}
	Let $h(\cdot)$ be defined as in \eqref{eq55}. The following result holds 
	\begin{equation*}\label{eq28}
	\partial h\left(u\right)=\left\{g+s\in\Re^n\ |\ g_i=c_i-d_i,i=1,\dots,n,\ s=\sum\limits_{k=1}^{m}s_{Q_k}\right\},
\end{equation*}
where $c_i$ is the number of $u$ less than $u_i$, $d_i$ is the number of $u$ greater than $u_i$. Define the index $Q:=\{j\ |\ \exists\ u_i,\ {\rm s.t.}\ u_i=u_j,\ i=1,\dots,n,\ j=1,\dots,n,\ i\neq j\}$, i.e. $Q$ is the index corresponding to equal components in $u$. $Q=\bigcup\limits_{i=1}^m Q_i$, where $m$ is the number of equal groups in the vector $u$, $Q_k$ is the index corresponding to each equality group. Each $s_{Q_k}$ satisfies $\sum\limits_{p\in Q_k} s_p=0,\ s_p\in[-|Q_k|+1,|Q_k|-1],\ p\in Q_k$. 
\end{proposition}
\begin{proof} 
	Define $f\left(r\right)=\left\|r\right\|_1$ and 
\begin{equation*}
	T=\begin{pmatrix}
		1 & -1 & 0  & 0 & \dots &0& 0\\
		1 & 0  & -1 & 0 & \dots &0& 0\\
		\vdots & \vdots &\vdots &\vdots &\ddots & \vdots &\vdots\\ 
		1 & 0  & 0 & 0 & \dots &0& -1\\
		0 & 1 & -1  & 0 & \dots &0& 0\\
		0 & 1  & 0 & -1 & \dots &0& 0\\
		\vdots & \vdots &\vdots &\vdots &\ddots & \vdots &\vdots\\ 
		0 & 1  & 0 & 0 & \dots &0& -1\\
		\vdots & \vdots &\vdots &\vdots &\ddots & \vdots &\vdots\\
		0 & 0  & 0 & 0 & \dots &1& -1\\
	\end{pmatrix}\in\Re^{\frac{n\left(n-1\right)}{2}\times n},\ Tu=\begin{pmatrix}
		u_1-u_2\\u_1-u_3\\\vdots\\u_1-u_n\\u_2-u_3\\u_2-u_4\\\vdots\\u_2-u_n\\\vdots\\u_{n-1}-u_n\\
	\end{pmatrix}\in\Re^{\frac{n\left(n-1\right)}{2}},
\end{equation*} 
thus $h\left(u\right)=f\left(Tu\right)$, ${\rm dom}f=\Re^{\frac{n\left(n-1\right)}{2}}$, the range of A contains at least a point of ${\rm ri}\left({\rm dom}f\right)$. So 
\begin{equation}\label{eq1}
	\partial h\left(u\right)=T^\top\partial f\left(Tu\right)=T^\top\partial f\left(r\right),
\end{equation}
where $r=Tu$, according to \cite[Theorem 23.9]{Rockafellar}. 

We define an index $P:=\left\{i\ |\ r_i=0,\ i=1,\dots,\frac{n\left(n-1\right)}{2}\right\},\ \bar{P}:=\left\{i\ |\ r_i\neq 0,\ i=1,\dots,\frac{n\left(n-1\right)}{2}\right\}$, 
\\$\partial f\left(r\right):=v+\Xi$, where
\begin{equation*}
	v_i=\begin{cases}
		{\rm sign}(r_i)&i\in \bar{P},\\ 0&i\in P,\\
	\end{cases}\ \Xi=\left\{x\in\Re^{\frac{n\left(n-1\right)}{2}}\ |\  x_i\in[-1,1],\ \forall\ i\in P,\ x_i=0,\ \forall\ i\in\bar{P}\right\}.
\end{equation*}
By \eqref{eq1}, it holds that
\begin{equation*}
		\partial h\left(u\right)=T^\top v+T^\top\Xi=\left(T^\top\right) _{\bar{P}}v_{\bar{P}}+\left\{s\in\Re^n\ |\ s=\left(T^\top\right)_{P} u,\ u\in\Re^{|P|},\ u_i\in[-1,1],\ i\in P\right\}.
\end{equation*}
Next, we consider the two parts in the above equation separately. 

For the first part, we first define the function $\delta_{i,j}\left(u\right)=\begin{cases}
	1& u_i>u_j\\-1&u_i<u_j\\0&u_i=u_j
\end{cases}$. For any $i$-th row of $T^\top$, 
\begin{equation*}
	\begin{split}
		\sum\limits_{t=1}^{\frac{n\left(n-1\right)}{2}}\left(T^\top\right)_{i,t}v_t=\sum\limits_{t\in \bar{P}}\left(T^\top\right)_{i,t}v_t=\sum\limits_{j=1}^{i-1}\left(-1\right)\delta_{j,i}\left(u\right)+\sum\limits_{j=i+1}^{n}\delta_{i,j}\left(u\right)=\sum\limits_{j=1}^{n}\delta_{i,j}\left(u\right)=c_i-d_i,
	\end{split}
\end{equation*}
where $c_i$ is the number of $u$ less than $u_i$, $d_i$ is the number of $u$ greater than $u_i$.

For the second part, the index of zero in $r$ corresponding to $Q_k$ of $u$ is marked as $P_k$. Obviously, $Q_p\bigcap Q_q=\emptyset,\ P_p\bigcap P_q=\emptyset,\ p\neq q$, and $\bigcup\limits_{k=1}^m P_k=P$. We continue to analyze $\left\{s\in\Re^n\ |\ s=\left(T^\top\right)_{P} u,\ u\in\Re^{|P|},\ u_i\in[-1,1],\ i\in P\right\}$, 
\begin{equation*}
	s=\left(T^\top\right)_P u=\sum_{k=1}^m \left(T^\top\right)_{P_k} u_{P_k}.
\end{equation*}
In fact, $\left(T^\top\right)_{P_k}$ is all zero except the rows indexed by $Q_k:=\{q_1^k,q_2^k,\dots,q_{m_k}^k\}$. The components of $\left(T^\top\right)_{P_k}u_{P_k}$ is all equal to zero except components indexed by $q_1^k,q_2^k,\dots,q_{m_k}^k$. Actually, 
\setcounter{MaxMatrixCols}{20}
\begin{equation*}
	\left(T^\top\right)_{P_k}u_{P_k}=\left[0,\cdots,0,s_{q_1^k}^\top,0,\cdots,0,s_{q_2^k}^\top,0,\cdots,s_{q_{m_k}^k}^\top,0,\cdots,0\right]^\top.
\end{equation*}
As a result, $s_{\bar{Q}}=0,\ s_{Q}$ can be divided into $s_{Q_1},s_{Q_2},\dots,s_{Q_m}$ according to the equal parts in $u$, and
\begin{equation*}
	s_{Q_k}=\left(T^\top\right)_{Q_k,P_k}u_{P_k},
\end{equation*}
where $\left(T^\top\right)_{Q_k,P_k}$ represents a matrix consisting of rows indexed by $Q_k$ and columns indexed by $P_k$ of $T^\top$. This is equivalent to removing the rows in $\left(T^\top\right)_{P_k}$ that are all zeros. 

Now we only focus on the part about $s_{Q_k}$, after removing the part that is 0, it holds that
\setcounter{MaxMatrixCols}{20}
\begin{equation*}
	\left(T^\top\right)_{Q_k,P_k}u_{P_k}=\begin{pmatrix}
		1&1&\dots&1&0&0&\dots&0&\dots&0\\
		-1&0&\dots&0&1&1&\dots&1&\dots&0\\
		0&-1&\dots&0&-1&0&\dots&0&\dots&0\\
		0&0&\dots&0&0&-1&\dots&0&\dots&0\\
		\vdots&\vdots& \ddots &\vdots&\vdots&\vdots& \ddots &\vdots& \ddots &\vdots\\
		0&0&\dots&0&0&0&\dots&0&\dots&1\\
		0&0&\dots&-1&0&0&\dots&-1&\dots&-1\\
	\end{pmatrix}u_{P_k}=s_{Q_k}.
\end{equation*}
We can find
\begin{equation*}
	\sum\limits_{p\in Q_k} s_p=\begin{pmatrix}
		1,\dots,1
	\end{pmatrix}\left(T^\top\right)_{Q_k,P_k}u_{P_k}=\begin{pmatrix}
		0,\dots,0
	\end{pmatrix}u_{P_k}=0.
\end{equation*}
Overall, the form of $\partial h(u)$ is $\left\{g+s\in\Re^n\ |\ g_i=c_i-d_i,i=1,\dots,n,\ s=\sum\limits_{k=1}^{m}s_{Q_k}\right\}$.
\end{proof}
\section{Proof of Lemma 1}\label{secA1}
To prove Lemma \ref{lem1}, we need the following proposition.
\begin{proposition}{\rm \cite[Proposition 2.8]{H}}\label{pro5}
	$\Sigma_Q$ and $C$ have the same definitions as above, $\Sigma_Q$ can be expressed as $\Sigma_Q={\rm Diag}\left(\Lambda_1,\dots,\Lambda_N\right)$, where $\Lambda_i$ is either the $n_i\times n_i$ zero matrix
	$O_{n_i}$ or the $n_i\times n_i$ identity matrix $I_{n_i}$, and any two consecutive blocks cannot be of the
	same type. Denote $K:=\left\{k\ |\ \Lambda_k=I_{n_k},\ k=1,\right.$
	$\left.\dots,N\right\}$. Then it holds that 
	\begin{equation*}
		I_n-C^{\top}\left(\Sigma_Q CC^{\top}\Sigma_Q\right)^\dag C={\rm Diag}\left(\Gamma_1,\dots,\Gamma_N\right),
	\end{equation*}
	where 
	\begin{equation*}
		\Gamma_i=\begin{cases}
			\frac{1}{n_i+1}E_{n_i+1},&if\ i\in K,\\
			I_{n_i},&if \ i\notin K\ and\ i\in\{1,N\},\\
			I_{n_i-1},&{\rm \text{otherwise}}.
		\end{cases}
	\end{equation*}
\end{proposition}
{\bf Proof of Lemma \ref{lem1}.}
For any $\mathcal{H}\in\hat{\partial}^2 \varphi_j\left(x\right)$, there is 
\begin{equation*}
	\mathcal{H}={\rho_j}\left(A^\top\left(I-V_1\right)A+I-V_2\right)+\frac{1}{\sigma_{k}}I,
\end{equation*}
where $V_1\in\mathcal{M}(f_1(x))$ and $V_2\in\partial {\rm Prox}_{\frac{\lambda}{\rho_j} \left\|\cdot\right\|_1}\left(f_2(x)\right)$. For any $d\in\Re^p$, $d\neq 0$, there is 
\begin{equation*}
	d^\top \mathcal{H} d={\rho_j}d^\top\left(A^\top\left(I-V_1\right)A+I-V_2\right)d+\frac{1}{\sigma_{k}}\|d\|_2^2.
\end{equation*}
Notice that by \eqref{eq56}, $V_2={\rm Diag}\left(v\right)$, where $v_i\in[0,1]$. We can see that $I-V_2$ is also a diagonal matrix with diagonal elements in $[0,1]$. Therefore, we have $d^\top\left(I-V_2\right)d\ge0$.

On the other hand, by Proposition \ref{pro4} and Proposition \ref{pro5}, we have 
\begin{equation*}
	V_1=P_y^{\top}\left({\rm Diag}\left(\Gamma_1,\dots,\Gamma_N\right)\right)P_y,
\end{equation*}
where $y=f_1(x)$, $P_y\in\Re^{n\times n}$ is a permutation matrix such that $y^\downarrow=P_yy$. $
\Gamma_i$ and $K$ are defined as in Proposition \ref{pro5}. Therefore, for $i\in N$, define 
\begin{equation}\label{eq44}
	\Theta_i=\begin{cases}
		I_{n_i+1}-\frac{1}{n_i+1}E_{n_i+1},&if\ i\in K,\\
		\textbf{O}_n,&if \ i\notin K\ and\ i\in\{1,N\},\\
		\textbf{O}_{n_i-1},&{\rm \text{otherwise}}.
	\end{cases}
\end{equation}
It is easy to verify that 
\begin{equation}\label{eq10}
	\begin{split}
		d^\top A^\top(I-V_1)Ad=&d^\top A^\top P_y^{\top}\left(I- {\rm Diag}\left(\Gamma_1,\dots,\Gamma_N\right)\right)P_y A d\\
		=&\hat{d}^\top\left(I- {\rm Diag}\left(\Gamma_1,\dots,\Gamma_N\right)\right)\hat{d}\\
		=&\hat{d}^\top{\rm Diag}\left(\Theta_1,\dots,\Theta_N\right)\hat{d}\quad \left(\text{by}\ \eqref{eq44}\right)\\
		=&\sum\limits_{i=1}^{N}\hat{d}_{\Theta_i}^\top\Theta_i \hat{d}_{\Theta_i},\\
	\end{split}
\end{equation}
where $\hat{d}=P_yAd$ and $\hat{d}_{\Theta_i}$ represents the components of $\hat{d}$ corresponding to $\Theta_i,\ i=1,\dots,N$. 

Obviously, for $i\notin K$, $\hat{d}_{\Theta_i}^\top\Theta_i \hat{d}_{\Theta_i}= 0$. If $i\in K$, we only need to show that for any $v\in\Re^r$, $v^\top\left(I_r-\frac{1}{r}E_r\right)v\ge 0$. Here $r\in\left\{n_i+1\ |\ i\in K\right\}$. Indeed, note that
\begin{equation*}\label{eq54}
	\begin{split}
		v^\top\left(I_r-\frac{1}{r}E_r\right)v=&\|v\|^2-\frac{1}{r}\sum\limits_{1\le j\le r}v_j\left(\sum\limits_{1\le i\le r}v_i\right)\\
		=&\|v\|^2-\frac{1}{r}\|v\|^2-\frac{1}{r}\sum\limits_{1\le i<j\le r}2v_jv_i\\
		=&\|v\|^2-\frac{1}{r}\|v\|^2-\frac{1}{r}\sum\limits_{1\le i<j\le r}\left(v_i^2+v_j^2\right)+\frac{1}{r}\sum\limits_{1\le i<j\le r}\left(v_i-v_j\right)^2\\
		=&\|v\|^2-\frac{1}{r}\|v\|^2-\frac{1}{r}\left(r-1\right)\|v\|^2+\frac{1}{r}\sum\limits_{1\le i<j\le r}\left(v_i-v_j\right)^2\\
		=&\frac{1}{r}\sum\limits_{1\le i< j\le r}\left(v_i-v_j\right)^2\ge 0.\\
	\end{split}
\end{equation*}
Together with \eqref{eq10}, we have $d^\top A^\top \left(I-V_1\right)Ad\ge 0,$ for any $d\neq0$. Consequently, for any $d\in\Re^p,\ d\neq 0,\ $there is $d^\top \mathcal{H} d\ge \frac{1}{\sigma_{k}}\|d\|_2^2>0$. That is, $\mathcal{H}$ is positive definite. The proof is finished.
\qed
\end{appendices}


\bibliography{ASPPA-bib}


\begin{thebibliography}{46}
\ifx \bisbn   \undefined \def \bisbn  #1{ISBN #1}\fi
\ifx \binits  \undefined \def \binits#1{#1}\fi
\ifx \bauthor  \undefined \def \bauthor#1{#1}\fi
\ifx \batitle  \undefined \def \batitle#1{#1}\fi
\ifx \bjtitle  \undefined \def \bjtitle#1{#1}\fi
\ifx \bvolume  \undefined \def \bvolume#1{\textbf{#1}}\fi
\ifx \byear  \undefined \def \byear#1{#1}\fi
\ifx \bissue  \undefined \def \bissue#1{#1}\fi
\ifx \bfpage  \undefined \def \bfpage#1{#1}\fi
\ifx \blpage  \undefined \def \blpage #1{#1}\fi
\ifx \burl  \undefined \def \burl#1{\textsf{#1}}\fi
\ifx \doiurl  \undefined \def \doiurl#1{\url{https://doi.org/#1}}\fi
\ifx \betal  \undefined \def \betal{\textit{et al.}}\fi
\ifx \binstitute  \undefined \def \binstitute#1{#1}\fi
\ifx \binstitutionaled  \undefined \def \binstitutionaled#1{#1}\fi
\ifx \bctitle  \undefined \def \bctitle#1{#1}\fi
\ifx \beditor  \undefined \def \beditor#1{#1}\fi
\ifx \bpublisher  \undefined \def \bpublisher#1{#1}\fi
\ifx \bbtitle  \undefined \def \bbtitle#1{#1}\fi
\ifx \bedition  \undefined \def \bedition#1{#1}\fi
\ifx \bseriesno  \undefined \def \bseriesno#1{#1}\fi
\ifx \blocation  \undefined \def \blocation#1{#1}\fi
\ifx \bsertitle  \undefined \def \bsertitle#1{#1}\fi
\ifx \bsnm \undefined \def \bsnm#1{#1}\fi
\ifx \bsuffix \undefined \def \bsuffix#1{#1}\fi
\ifx \bparticle \undefined \def \bparticle#1{#1}\fi
\ifx \barticle \undefined \def \barticle#1{#1}\fi
\bibcommenthead
\ifx \bconfdate \undefined \def \bconfdate #1{#1}\fi
\ifx \botherref \undefined \def \botherref #1{#1}\fi
\ifx \url \undefined \def \url#1{\textsf{#1}}\fi
\ifx \bchapter \undefined \def \bchapter#1{#1}\fi
\ifx \bbook \undefined \def \bbook#1{#1}\fi
\ifx \bcomment \undefined \def \bcomment#1{#1}\fi
\ifx \oauthor \undefined \def \oauthor#1{#1}\fi
\ifx \citeauthoryear \undefined \def \citeauthoryear#1{#1}\fi
\ifx \endbibitem  \undefined \def \endbibitem {}\fi
\ifx \bconflocation  \undefined \def \bconflocation#1{#1}\fi
\ifx \arxivurl  \undefined \def \arxivurl#1{\textsf{#1}}\fi
\csname PreBibitemsHook\endcsname

\bibitem[\protect\citeauthoryear{Yuan et~al.}{2023}]{SDF}
\begin{botherref}
\oauthor{\bsnm{Yuan}, \binits{Y.}},
\oauthor{\bsnm{Lin}, \binits{M.}},
\oauthor{\bsnm{Sun}, \binits{D.}},
\oauthor{\bsnm{Toh}, \binits{K.-C.}}:
Adaptive sieving: {A} dimension reduction technique for sparse optimization
  problems.
arXiv preprint arXiv:2306.17369
(2023)
\end{botherref}
\endbibitem

\bibitem[\protect\citeauthoryear{Wang et~al.}{2020}]{Wang}
\begin{barticle}
\bauthor{\bsnm{Wang}, \binits{L.}},
\bauthor{\bsnm{Peng}, \binits{B.}},
\bauthor{\bsnm{Bradic}, \binits{J.}},
\bauthor{\bsnm{Li}, \binits{R.}},
\bauthor{\bsnm{Wu}, \binits{Y.}}:
\batitle{A tuning-free robust and efficient approach to high-dimensional
  regression}.
\bjtitle{Journal of the American Statistical Association}
\bvolume{115}(\bissue{532}),
\bfpage{1700}--\blpage{1714}
(\byear{2020})
\end{barticle}
\endbibitem

\bibitem[\protect\citeauthoryear{Wang and Li}{2009}]{low-rank-1}
\begin{barticle}
\bauthor{\bsnm{Wang}, \binits{L.}},
\bauthor{\bsnm{Li}, \binits{R.}}:
\batitle{Weighted wilcoxon-type smoothly clipped absolute deviation method}.
\bjtitle{Biometrics}
\bvolume{65}(\bissue{2}),
\bfpage{564}--\blpage{571}
(\byear{2009})
\end{barticle}
\endbibitem

\bibitem[\protect\citeauthoryear{Hettmansperger and McKean}{2010}]{TP}
\begin{bbook}
\bauthor{\bsnm{Hettmansperger}, \binits{T.P.}},
\bauthor{\bsnm{McKean}, \binits{J.W.}}:
\bbtitle{Robust Nonparametric Statistical Methods}.
\bpublisher{CRC Press},
\blocation{Boca Raton}
(\byear{2010})
\end{bbook}
\endbibitem

\bibitem[\protect\citeauthoryear{Fan et~al.}{2021}]{FJQ}
\begin{barticle}
\bauthor{\bsnm{Fan}, \binits{J.}},
\bauthor{\bsnm{Wang}, \binits{W.}},
\bauthor{\bsnm{Zhu}, \binits{Z.}}:
\batitle{A shrinkage principle for heavy-tailed data: High-dimensional robust
  low-rank matrix recovery}.
\bjtitle{The Annals of Statistics}
\bvolume{49}(\bissue{3}),
\bfpage{1239}--\blpage{1266}
(\byear{2021})
\end{barticle}
\endbibitem

\bibitem[\protect\citeauthoryear{Fan et~al.}{2020}]{FJQcom}
\begin{barticle}
\bauthor{\bsnm{Fan}, \binits{J.}},
\bauthor{\bsnm{Ma}, \binits{C.}},
\bauthor{\bsnm{Wang}, \binits{K.}}:
\batitle{Comment on ``a tuning-free robust and efficient approach to
  high-dimensional regression"}.
\bjtitle{Journal of the American Statistical Association}
\bvolume{115}(\bissue{532}),
\bfpage{1720}--\blpage{1725}
(\byear{2020})
\end{barticle}
\endbibitem

\bibitem[\protect\citeauthoryear{Efron}{2004}]{Ef}
\begin{barticle}
\bauthor{\bsnm{Efron}, \binits{B.}}:
\batitle{The estimation of prediction error: {C}ovariance penalties and
  cross-validation}.
\bjtitle{Journal of the American Statistical Association}
\bvolume{99}(\bissue{467}),
\bfpage{619}--\blpage{632}
(\byear{2004})
\end{barticle}
\endbibitem

\bibitem[\protect\citeauthoryear{Tseng}{2001}]{Tseng}
\begin{barticle}
\bauthor{\bsnm{Tseng}, \binits{P.}}:
\batitle{Convergence of a block coordinate descent method for nondifferentiable
  minimization}.
\bjtitle{Journal of Optimization Theory and Applications}
\bvolume{109}(\bissue{3}),
\bfpage{475}--\blpage{494}
(\byear{2001})
\end{barticle}
\endbibitem

\bibitem[\protect\citeauthoryear{Tseng and Yun}{2009}]{Tseng2}
\begin{barticle}
\bauthor{\bsnm{Tseng}, \binits{P.}},
\bauthor{\bsnm{Yun}, \binits{S.}}:
\batitle{A coordinate gradient descent method for nonsmooth separable
  minimization}.
\bjtitle{Mathematical Programming}
\bvolume{117},
\bfpage{387}--\blpage{423}
(\byear{2009})
\end{barticle}
\endbibitem

\bibitem[\protect\citeauthoryear{Friedman et~al.}{2007}]{Frie}
\begin{barticle}
\bauthor{\bsnm{Friedman}, \binits{J.}},
\bauthor{\bsnm{Hastie}, \binits{T.}},
\bauthor{\bsnm{H{\"o}fling}, \binits{H.}},
\bauthor{\bsnm{Tibshirani}, \binits{R.}}:
\batitle{Pathwise coordinate optimization}.
\bjtitle{The Annals of Applied Statistics}
\bvolume{1}(\bissue{2}),
\bfpage{302}--\blpage{332}
(\byear{2007})
\end{barticle}
\endbibitem

\bibitem[\protect\citeauthoryear{Wu and Lange}{2008}]{Wu}
\begin{barticle}
\bauthor{\bsnm{Wu}, \binits{T.T.}},
\bauthor{\bsnm{Lange}, \binits{K.}}:
\batitle{Coordinate descent algorithms for lasso penalized regression}.
\bjtitle{The Annals of Applied Statistics}
\bvolume{2}(\bissue{1}),
\bfpage{224}--\blpage{244}
(\byear{2008})
\end{barticle}
\endbibitem

\bibitem[\protect\citeauthoryear{Beck and Teboulle}{2009}]{fista}
\begin{barticle}
\bauthor{\bsnm{Beck}, \binits{A.}},
\bauthor{\bsnm{Teboulle}, \binits{M.}}:
\batitle{A fast iterative shrinkage-thresholding algorithm for linear inverse
  problems}.
\bjtitle{SIAM Journal on Imaging Sciences}
\bvolume{2}(\bissue{1}),
\bfpage{183}--\blpage{202}
(\byear{2009})
\end{barticle}
\endbibitem

\bibitem[\protect\citeauthoryear{Li et~al.}{2015}]{ADMM}
\begin{barticle}
\bauthor{\bsnm{Li}, \binits{X.}},
\bauthor{\bsnm{Zhao}, \binits{T.}},
\bauthor{\bsnm{Yuan}, \binits{X.}},
\bauthor{\bsnm{Liu}, \binits{H.}}:
\batitle{The flare package for high dimensional linear regression and precision
  matrix estimation in {R}}.
\bjtitle{Journal of Machine Learning Research}
\bvolume{16}(\bissue{18}),
\bfpage{553}--\blpage{557}
(\byear{2015})
\end{barticle}
\endbibitem

\bibitem[\protect\citeauthoryear{Li et~al.}{2018}]{LXD}
\begin{barticle}
\bauthor{\bsnm{Li}, \binits{X.}},
\bauthor{\bsnm{Sun}, \binits{D.}},
\bauthor{\bsnm{Toh}, \binits{K.-C.}}:
\batitle{A highly efficient semismooth newton augmented {L}agrangian method for
  solving lasso problems}.
\bjtitle{SIAM Journal on Optimization}
\bvolume{28}(\bissue{1}),
\bfpage{433}--\blpage{458}
(\byear{2018})
\end{barticle}
\endbibitem

\bibitem[\protect\citeauthoryear{Kim et~al.}{2015}]{ranklasso}
\begin{bchapter}
\bauthor{\bsnm{Kim}, \binits{H.-J.}},
\bauthor{\bsnm{Ollila}, \binits{E.}},
\bauthor{\bsnm{Koivunen}, \binits{V.}}:
\bctitle{New robust {LASSO} method based on ranks}.
In: \bbtitle{2015 23rd European Signal Processing Conference (EUSIPCO)},
\bconflocation{Nice},
pp. \bfpage{699}--\blpage{703}
(\byear{2015}).
\bcomment{IEEE}
\end{bchapter}
\endbibitem

\bibitem[\protect\citeauthoryear{Barrodale and Roberts}{1973}]{BR}
\begin{barticle}
\bauthor{\bsnm{Barrodale}, \binits{I.}},
\bauthor{\bsnm{Roberts}, \binits{F.D.}}:
\batitle{An improved algorithm for discrete $\ell_1$ linear approximation}.
\bjtitle{SIAM Journal on Numerical Analysis}
\bvolume{10}(\bissue{5}),
\bfpage{839}--\blpage{848}
(\byear{1973})
\end{barticle}
\endbibitem

\bibitem[\protect\citeauthoryear{Zoubir et~al.}{2018}]{code}
\begin{bbook}
\bauthor{\bsnm{Zoubir}, \binits{A.M.}},
\bauthor{\bsnm{Koivunen}, \binits{V.}},
\bauthor{\bsnm{Ollila}, \binits{E.}},
\bauthor{\bsnm{Muma}, \binits{M.}}:
\bbtitle{Robust Statistics for Signal Processing}.
\bpublisher{Cambridge University Press},
\blocation{Cambridge}
(\byear{2018})
\end{bbook}
\endbibitem

\bibitem[\protect\citeauthoryear{Tang et~al.}{2021}]{WCJ}
\begin{botherref}
\oauthor{\bsnm{Tang}, \binits{P.}},
\oauthor{\bsnm{Wang}, \binits{C.}},
\oauthor{\bsnm{Jiang}, \binits{B.}}:
A proximal-proximal majorization-minimization algorithm for nonconvex
  tuning-free robust regression problems.
arXiv preprint arXiv:2106.13683
(2021)
\end{botherref}
\endbibitem

\bibitem[\protect\citeauthoryear{El~Ghaoui et~al.}{2010}]{Safe}
\begin{botherref}
\oauthor{\bsnm{El~Ghaoui}, \binits{L.}},
\oauthor{\bsnm{Viallon}, \binits{V.}},
\oauthor{\bsnm{Rabbani}, \binits{T.}}:
Safe feature elimination in sparse supervised learning.
Technical Report Technical Report UC/EECS-2010-126,
Electrical Engineering and Computer Sciences Department, University of
  California at Berkeley,
Berkeley
(2010)
\end{botherref}
\endbibitem

\bibitem[\protect\citeauthoryear{Tibshirani}{1996}]{Tib}
\begin{barticle}
\bauthor{\bsnm{Tibshirani}, \binits{R.}}:
\batitle{Regression shrinkage and selection via the lasso}.
\bjtitle{Journal of the Royal Statistical Society: Series B (Statistical
  Methodology)}
\bvolume{58}(\bissue{1}),
\bfpage{267}--\blpage{288}
(\byear{1996})
\end{barticle}
\endbibitem

\bibitem[\protect\citeauthoryear{Bonnefoy et~al.}{2014}]{dyna}
\begin{bchapter}
\bauthor{\bsnm{Bonnefoy}, \binits{A.}},
\bauthor{\bsnm{Emiya}, \binits{V.}},
\bauthor{\bsnm{Ralaivola}, \binits{L.}},
\bauthor{\bsnm{Gribonval}, \binits{R.}}:
\bctitle{A dynamic screening principle for the {L}asso}.
In: \bbtitle{European Signal Processing Conference EUSIPCO 2014},
pp. \bfpage{1}--\blpage{5}.
\bpublisher{IEEE},
\blocation{Lisboa}
(\byear{2014})
\end{bchapter}
\endbibitem

\bibitem[\protect\citeauthoryear{Wang et~al.}{2013}]{DPP}
\begin{bchapter}
\bauthor{\bsnm{Wang}, \binits{J.}},
\bauthor{\bsnm{Zhou}, \binits{J.}},
\bauthor{\bsnm{Wonka}, \binits{P.}},
\bauthor{\bsnm{Ye}, \binits{J.}}:
\bctitle{Lasso screening rules via dual polytope projection}.
In: \bbtitle{Proceedings of the 26th International Conference on Neural
  Information Processing Systems},
pp. \bfpage{1070}--\blpage{1078}.
\bpublisher{Curran Associates Inc.},
\blocation{New York}
(\byear{2013})
\end{bchapter}
\endbibitem

\bibitem[\protect\citeauthoryear{Fercoq et~al.}{2015}]{Fer}
\begin{bchapter}
\bauthor{\bsnm{Fercoq}, \binits{O.}},
\bauthor{\bsnm{Gramfort}, \binits{A.}},
\bauthor{\bsnm{Salmon}, \binits{J.}}:
\bctitle{Mind the duality gap: {S}afer rules for the lasso}.
In: \bbtitle{International Conference on Machine Learning},
\bconflocation{Lille},
pp. \bfpage{333}--\blpage{342}
(\byear{2015}).
\bcomment{PMLR}
\end{bchapter}
\endbibitem

\bibitem[\protect\citeauthoryear{Shang et~al.}{2022}]{SafeRank}
\begin{barticle}
\bauthor{\bsnm{Shang}, \binits{P.}},
\bauthor{\bsnm{Kong}, \binits{L.}},
\bauthor{\bsnm{Liu}, \binits{D.}}:
\batitle{A safe feature screening rule for rank lasso}.
\bjtitle{IEEE Signal Processing Letters}
\bvolume{29},
\bfpage{1062}--\blpage{1066}
(\byear{2022})
\end{barticle}
\endbibitem

\bibitem[\protect\citeauthoryear{Yuan et~al.}{2022}]{SDF2}
\begin{barticle}
\bauthor{\bsnm{Yuan}, \binits{Y.}},
\bauthor{\bsnm{Chang}, \binits{T.-H.}},
\bauthor{\bsnm{Sun}, \binits{D.}},
\bauthor{\bsnm{Toh}, \binits{K.-C.}}:
\batitle{A dimension reduction technique for large-scale structured sparse
  optimization problems with application to convex clustering}.
\bjtitle{SIAM Journal on Optimization}
\bvolume{32}(\bissue{3}),
\bfpage{2294}--\blpage{2318}
(\byear{2022})
\end{barticle}
\endbibitem

\bibitem[\protect\citeauthoryear{Rockafellar}{1976}]{Rockafellar}
\begin{barticle}
\bauthor{\bsnm{Rockafellar}, \binits{R.T.}}:
\batitle{Monotone operators and the proximal point algorithm}.
\bjtitle{SIAM Journal on Control and Optimization}
\bvolume{14}(\bissue{5}),
\bfpage{877}--\blpage{898}
(\byear{1976})
\end{barticle}
\endbibitem

\bibitem[\protect\citeauthoryear{Rockafellar}{1970}]{mono}
\begin{barticle}
\bauthor{\bsnm{Rockafellar}, \binits{R.}}:
\batitle{On the maximal monotonicity of subdifferential mappings}.
\bjtitle{Pacific Journal of Mathematics}
\bvolume{33}(\bissue{1}),
\bfpage{209}--\blpage{216}
(\byear{1970})
\end{barticle}
\endbibitem

\bibitem[\protect\citeauthoryear{Sun}{1986}]{SJ}
\begin{botherref}
\oauthor{\bsnm{Sun}, \binits{J.}}:
On monotropic piecewise quadratic programming (network, algorithm, convex
  programming, decomposition method).
PhD thesis,
Computer Science Department, University of Washington,
Seattle
(1986)
\end{botherref}
\endbibitem

\bibitem[\protect\citeauthoryear{Robinson}{1981}]{Robinson1981}
\begin{bbook}
\bauthor{\bsnm{Robinson}, \binits{S.M.}}:
In: \beditor{\bsnm{K{\"o}nig}, \binits{H.}},
\beditor{\bsnm{Korte}, \binits{B.}},
\beditor{\bsnm{Ritter}, \binits{K.}} (eds.)
\bbtitle{Some continuity properties of polyhedral multifunctions},
pp. \bfpage{206}--\blpage{214}.
\bpublisher{Springer},
\blocation{Berlin, Heidelberg}
(\byear{1981})
\end{bbook}
\endbibitem

\bibitem[\protect\citeauthoryear{Moreau}{1965}]{J}
\begin{barticle}
\bauthor{\bsnm{Moreau}, \binits{J.-J.}}:
\batitle{Proximit{\'e} et dualit{\'e} dans un espace hilbertien}.
\bjtitle{Bulletin de la Soci{\'e}t{\'e} math{\'e}matique de France}
\bvolume{93},
\bfpage{273}--\blpage{299}
(\byear{1965})
\end{barticle}
\endbibitem

\bibitem[\protect\citeauthoryear{Lemar{\'e}chal and
  Sagastiz{\'a}bal}{1997}]{Lemar}
\begin{barticle}
\bauthor{\bsnm{Lemar{\'e}chal}, \binits{C.}},
\bauthor{\bsnm{Sagastiz{\'a}bal}, \binits{C.}}:
\batitle{Practical aspects of the {M}oreau--{Y}osida regularization:
  Theoretical preliminaries}.
\bjtitle{SIAM Journal on Optimization}
\bvolume{7}(\bissue{2}),
\bfpage{367}--\blpage{385}
(\byear{1997})
\end{barticle}
\endbibitem

\bibitem[\protect\citeauthoryear{Rockafellar}{1976}]{RT1}
\begin{barticle}
\bauthor{\bsnm{Rockafellar}, \binits{R.T.}}:
\batitle{Augmented {L}agrangians and applications of the proximal point
  algorithm in convex programming}.
\bjtitle{Mathematics of Operations Research}
\bvolume{1}(\bissue{2}),
\bfpage{97}--\blpage{116}
(\byear{1976})
\end{barticle}
\endbibitem

\bibitem[\protect\citeauthoryear{Luque}{1984}]{asy}
\begin{barticle}
\bauthor{\bsnm{Luque}, \binits{F.J.}}:
\batitle{Asymptotic convergence analysis of the proximal point algorithm}.
\bjtitle{SIAM Journal on Control and Optimization}
\bvolume{22}(\bissue{2}),
\bfpage{277}--\blpage{293}
(\byear{1984})
\end{barticle}
\endbibitem

\bibitem[\protect\citeauthoryear{Sun et~al.}{2021}]{Con}
\begin{barticle}
\bauthor{\bsnm{Sun}, \binits{D.}},
\bauthor{\bsnm{Toh}, \binits{K.-C.}},
\bauthor{\bsnm{Yuan}, \binits{Y.}}:
\batitle{Convex clustering: {M}odel, theoretical guarantee and efficient
  algorithm}.
\bjtitle{Journal of Machine Learning Research}
\bvolume{22}(\bissue{1}),
\bfpage{427}--\blpage{458}
(\byear{2021})
\end{barticle}
\endbibitem

\bibitem[\protect\citeauthoryear{Yan and Li}{2020}]{YYQ}
\begin{barticle}
\bauthor{\bsnm{Yan}, \binits{Y.}},
\bauthor{\bsnm{Li}, \binits{Q.}}:
\batitle{An efficient augmented {L}agrangian method for support vector
  machine}.
\bjtitle{Optimization Methods and Software}
\bvolume{35}(\bissue{4}),
\bfpage{855}--\blpage{883}
(\byear{2020})
\end{barticle}
\endbibitem

\bibitem[\protect\citeauthoryear{Li et~al.}{2018}]{Li}
\begin{barticle}
\bauthor{\bsnm{Li}, \binits{X.}},
\bauthor{\bsnm{Sun}, \binits{D.}},
\bauthor{\bsnm{Toh}, \binits{K.-C.}}:
\batitle{On efficiently solving the subproblems of a level-set method for fused
  lasso problems}.
\bjtitle{SIAM Journal on Optimization}
\bvolume{28}(\bissue{2}),
\bfpage{1842}--\blpage{1866}
(\byear{2018})
\end{barticle}
\endbibitem

\bibitem[\protect\citeauthoryear{Lin et~al.}{2019}]{H}
\begin{barticle}
\bauthor{\bsnm{Lin}, \binits{M.}},
\bauthor{\bsnm{Liu}, \binits{Y.-J.}},
\bauthor{\bsnm{Sun}, \binits{D.}},
\bauthor{\bsnm{Toh}, \binits{K.-C.}}:
\batitle{Efficient sparse semismooth newton methods for the clustered lasso
  problem}.
\bjtitle{SIAM Journal on Optimization}
\bvolume{29}(\bissue{3}),
\bfpage{2026}--\blpage{2052}
(\byear{2019})
\end{barticle}
\endbibitem

\bibitem[\protect\citeauthoryear{Hiriart-Urruty et~al.}{1984}]{partical}
\begin{barticle}
\bauthor{\bsnm{Hiriart-Urruty}, \binits{J.-B.}},
\bauthor{\bsnm{Strodiot}, \binits{J.-J.}},
\bauthor{\bsnm{Nguyen}, \binits{V.H.}}:
\batitle{Generalized hessian matrix and second-order optimality conditions for
  problems with {$C^{1,1}$} data}.
\bjtitle{Applied Mathematics and Optimization}
\bvolume{11}(\bissue{1}),
\bfpage{43}--\blpage{56}
(\byear{1984})
\end{barticle}
\endbibitem

\bibitem[\protect\citeauthoryear{Qi and Sun}{2006}]{Sun2006}
\begin{barticle}
\bauthor{\bsnm{Qi}, \binits{H.}},
\bauthor{\bsnm{Sun}, \binits{D.}}:
\batitle{A quadratically convergent newton method for computing the nearest
  correlation matrix}.
\bjtitle{SIAM Journal on Matrix Analysis and Applications}
\bvolume{28}(\bissue{2}),
\bfpage{360}--\blpage{385}
(\byear{2006})
\end{barticle}
\endbibitem

\bibitem[\protect\citeauthoryear{Qi}{2013}]{Qi}
\begin{barticle}
\bauthor{\bsnm{Qi}, \binits{H.-D.}}:
\batitle{A semismooth newton method for the nearest euclidean distance matrix
  problem}.
\bjtitle{SIAM Journal on Matrix Analysis and Applications}
\bvolume{34}(\bissue{1}),
\bfpage{67}--\blpage{93}
(\byear{2013})
\end{barticle}
\endbibitem

\bibitem[\protect\citeauthoryear{Yin and Li}{2019}]{Yin}
\begin{barticle}
\bauthor{\bsnm{Yin}, \binits{J.}},
\bauthor{\bsnm{Li}, \binits{Q.}}:
\batitle{A semismooth newton method for support vector classification and
  regression}.
\bjtitle{Computational Optimization and Applications}
\bvolume{73}(\bissue{2}),
\bfpage{477}--\blpage{508}
(\byear{2019})
\end{barticle}
\endbibitem

\bibitem[\protect\citeauthoryear{Li et~al.}{2010}]{LQN2010}
\begin{barticle}
\bauthor{\bsnm{Li}, \binits{Q.}},
\bauthor{\bsnm{Li}, \binits{D.}},
\bauthor{\bsnm{Qi}, \binits{H.}}:
\batitle{Newton's method for computing the nearest correlation matrix with a
  simple upper bound}.
\bjtitle{Journal of Optimization Theory and Applications}
\bvolume{147}(\bissue{3}),
\bfpage{546}--\blpage{568}
(\byear{2010})
\end{barticle}
\endbibitem

\bibitem[\protect\citeauthoryear{Qi and Sun}{1993}]{1993}
\begin{barticle}
\bauthor{\bsnm{Qi}, \binits{L.}},
\bauthor{\bsnm{Sun}, \binits{J.}}:
\batitle{A nonsmooth version of newton's method}.
\bjtitle{Mathematical Programming}
\bvolume{58}(\bissue{1-3}),
\bfpage{353}--\blpage{367}
(\byear{1993})
\end{barticle}
\endbibitem

\bibitem[\protect\citeauthoryear{Belloni et~al.}{2011}]{square}
\begin{barticle}
\bauthor{\bsnm{Belloni}, \binits{A.}},
\bauthor{\bsnm{Chernozhukov}, \binits{V.}},
\bauthor{\bsnm{Wang}, \binits{L.}}:
\batitle{Square-root lasso: {P}ivotal recovery of sparse signals via conic
  programming}.
\bjtitle{Biometrika}
\bvolume{98}(\bissue{4}),
\bfpage{791}--\blpage{806}
(\byear{2011})
\end{barticle}
\endbibitem

\bibitem[\protect\citeauthoryear{Barrodale and Roberts}{1974}]{BRcode}
\begin{barticle}
\bauthor{\bsnm{Barrodale}, \binits{I.}},
\bauthor{\bsnm{Roberts}, \binits{F.}}:
\batitle{Solution of an overdetermined system of equations in the $\ell_1$ norm
  {[F4]}}.
\bjtitle{Communications of the ACM}
\bvolume{17}(\bissue{6}),
\bfpage{319}--\blpage{320}
(\byear{1974})
\end{barticle}
\endbibitem

\bibitem[\protect\citeauthoryear{Tang et~al.}{2020}]{PMM}
\begin{barticle}
\bauthor{\bsnm{Tang}, \binits{P.}},
\bauthor{\bsnm{Wang}, \binits{C.}},
\bauthor{\bsnm{Sun}, \binits{D.}},
\bauthor{\bsnm{Toh}, \binits{K.-C.}}:
\batitle{A sparse semismooth newton based proximal majorization-minimization
  algorithm for nonconvex square-root-loss regression problem}.
\bjtitle{Journal of Machine Learning Research}
\bvolume{21}(\bissue{1}),
\bfpage{9253}--\blpage{9290}
(\byear{2020})
\end{barticle}
\endbibitem

\end{thebibliography}
\end{document}